\documentclass[reqno]{amsart}

\usepackage{amsmath}
\usepackage{amssymb}
\usepackage{verbatim}
\usepackage{dsfont}
\usepackage{pifont}
\usepackage{enumerate}
\usepackage[all]{xy}
\usepackage{amsthm}
\usepackage{stmaryrd}
\usepackage{amscd}
\usepackage{amsfonts}
\usepackage{latexsym}
\usepackage{amscd}
\usepackage{graphicx}
\usepackage{tikz}
\usepackage[vcentermath]{youngtab}
\usepackage{setspace}
\usepackage{wasysym}
\usepackage{bm}
\usetikzlibrary{decorations.pathmorphing}

\makeatletter

\def\a{{\mathsf{a}}}
\def\b{{\mathsf{b}}}
\def\c{{\mathsf{c}}}

\def\SS{{\operatorname{S}}}
\def\D{{\operatorname{D}}}
\def\Y{{\operatorname{Y}}}
\def\P{{\operatorname{P}}}
\def\Q{{\operatorname{Q}}}
\def\L{{\operatorname{L}}}
\def\V{{\operatorname{V}}}
\def\qGL{q\text{-}GL}

\newtheorem{theorem}{Theorem}[section]
\newtheorem{lemma}[theorem]{Lemma}

\newtheorem{corollary}[theorem]{Corollary} 
\theoremstyle{definition}  

\newtheorem{example}[theorem]{Example}

\newtheorem{remark}[theorem]{Remark}

\@addtoreset{equation}{section}

\def\bigsqcup{\coprod}
\def\trans{\mathsf{t}}
\def\s{\chi}
\def\i{\imath}
\def\j{\jmath}
\def\Fock{\mathcal{F}}
\def\ch{\operatorname{ch}}
\def\cont{\operatorname{cont}}
\def\type{\operatorname{type}}
\def\res{\operatorname{res}}
\def\ind{\operatorname{ind}}
\def\infl{\operatorname{infl}}
\def\op{\operatorname{op}}
\def\Tr{\operatorname{Tr}}
\def\squarelet{\heartsuit}
\def\up{{{\color{darkblue}\uparrow}}}
\def\down{{{\color{darkblue}\downarrow}}}
\def\SYM{\operatorname{Sym}}
\def\Sym{\mathfrak{S}}

\def\RR{\mathbb{R}}
\def\Bubble{\mathord{
\begin{tikzpicture}[baseline = .8mm]
  \draw[<-,thick,darkblue] (0.2,0.2) to[out=90,in=0] (0,.4);
  \draw[-,thick,darkblue] (0,0.4) to[out=180,in=90] (-.2,0.2);
\draw[-,thick,darkblue] (-.2,0.2) to[out=-90,in=180] (0,0);
  \draw[-,thick,darkblue] (0,0) to[out=0,in=-90] (0.2,0.2);
\end{tikzpicture}
}}
\def\JM{L}
\def\jm{l}
\def\bi{{\bm{i}}}
\def\bj{{\bm{j}}}

\def\ob{\operatorname{ob}}

\def\FOT{\mathcal{FOT}}
\def\OB{\mathcal{OB}}
\def\OS{\mathcal{OS}}
\def\AOB{\mathcal{AOB}}
\def\H{\mathcal{H}}
\def\AOS{\mathcal{AOS}}
\def\words{\langle \up,\down\rangle}
\def\Words{\langle \up,\down\rangle_I}
\def\Par{\operatorname{Bip}}
\def\RPar{e\!\operatorname{-Bip}}
\def\VV{{\mathcal V}}
\def\La{\Lambda}
\def\LA{{\bm{\lambda}}}
\def\MU{{\bm{\mu}}}
\def\KAPPA{{\bm{\kappa}}}
\def\NU{{\bm{\nu}}}
\def\NOTHING{{\bm\varnothing}}
\def\Mat{\operatorname{Mat}}

\def\Mod{\operatorname{Mod-}\!}
\def\lfdMod{\operatorname{lfdMod-}\!}
\def\fdMod{\operatorname{fdMod-}\!}
\def\pMod{\operatorname{pMod-}\!}
\def\deltaMod{\Delta\!\operatorname{Mod-}\!}
\newcommand{\End}{\operatorname{End}}
\def\ev{\operatorname{ev}}
\def\coev{\operatorname{coev}}
\newcommand{\wt}{\operatorname{wt}}
\newcommand{\WT}{{\text{\bf wt}}}
 
\newcommand{\Hom}{\operatorname{Hom}} 
 
\newcommand{\Rep}{\operatorname{Rep}}
\newcommand{\Tilt}{\operatorname{Tilt}}
\newcommand{\REP}{{\underline{\operatorname{Re}}\!\operatorname{p}\,}}
\newcommand{\id}{\operatorname{id}}
\newcommand{\g}{\mathfrak{g}}
\newcommand{\h}{\mathfrak{h}}
\newcommand{\ZZ}{\mathbb{Z}}
\newcommand{\NN}{\mathbb{N}}
\newcommand{\CC}{\mathbb{C}}
\newcommand{\QQ}{\mathbb{Q}}
\newcommand{\K}{\mathbb{K}}
\newcommand{\eps}{\varepsilon}

\newcommand{\unit}{\mathds{1}}

\newcommand{\rad}{{\operatorname{rad}}}
\def\la{\lambda}

\renewcommand{\k}{\Bbbk}

\definecolor{white}{HTML}{FFFFFF}
\definecolor{darkblue}{HTML}{111199}
\definecolor{darkgreen}{HTML}{336633}
\definecolor{darkred}{HTML}{993333}
\definecolor{darkpurple}{HTML}{995599}

\newdimen\hoogte    \hoogte=20pt    
\newdimen\breedte   \breedte=22pt   
\newdimen\dikte     \dikte=.9pt    

\newenvironment{Young}{\begingroup
       \def\vr{\vrule height0.7\hoogte width\dikte depth 0.45\hoogte}
       \def\fbox##1{\vbox{\offinterlineskip
                    \hrule height\dikte
                    \hbox to \breedte{\vr\hfill##1\hfill\vr}
                    \hrule height\dikte}}
       \vbox\bgroup \offinterlineskip \tabskip=-\dikte \lineskip=-\dikte
            \halign\bgroup &\fbox{##\unskip}\unskip  \crcr }
       {\egroup\egroup\endgroup}
\def\diagram#1{\relax\ifmmode\vcenter{\,\begin{Young}#1\end{Young}\,}\else%
              $\vcenter{\,\begin{Young}#1\end{Young}\,}$\fi}

\allowdisplaybreaks

\begin{document}

\title[Oriented skein category]{Representations of the oriented skein category}

\author[J. Brundan]{Jonathan Brundan}
\address{Department of Mathematics,
University of Oregon, Eugene, OR 97403, USA}
\email{brundan@uoregon.edu}

\thanks{2010 {\it Mathematics Subject Classification}: 17B10, 18D10.}
\thanks{Research supported in part by NSF grant DMS-1700905.}

\begin{abstract}
The {\em oriented skein category} $\OS(z,t)$
is a ribbon category which underpins the definition of the HOMFLY-PT invariant of an oriented link,
in the same way that the Temperley-Lieb category 
underpins the Jones polynomial. 
In this article, we develop its representation theory using a highest
weight theory approach. 
This allows us to determine the Grothendieck ring of its additive
Karoubi envelope for all possible choices of parameters, including
the (already well-known) semisimple case,
and all non-semisimple situations.
Then we construct a graded lift of $\OS(z,t)$ by realizing it as a
2-representation of a Kac-Moody 2-category.
We also discuss the degenerate analog of $\OS(z,t)$, which is the 
{\em oriented Brauer category} $\OB(\delta)$. 
\end{abstract}

\maketitle  
\section{Introduction}

{\noindent {\bf 1.1.}}
We begin by recalling briefly the definition of 
the category $\FOT$ of {\em framed oriented tangles}; this is the framed analog of the
oriented tangle category $\mathcal{OT}$ introduced by Turaev in
\cite{Turaev3} and
also appears in \cite[Remark 8.10.3]{EGNO} where it is denoted $\mathcal{FT}$.
By definition, it is the 
strict monoidal category with objects given by the set $\words$ of all words in the letters $\up$
and $\down$. Tensor product of objects is given by concatenation, e.g.,
$\up\otimes\up\otimes \down=\up\up\down$, and
the unit object $\unit$ is the empty word $\varnothing$.
For two words $\a = \a_m\cdots \a_1, \b = \b_n \cdots \b_1\in\words$,
morphisms $f:\a \rightarrow \b$ are isotopy classes of 
framed oriented tangles in 
$[0,1]\times[0,1]\times\RR$ with boundary 
$$
\Big\{
\big({\textstyle \frac{m+1-i}{m+1}},0,0\big)\:\Big|\:i=1,\dots,m\Big\}\cup
\Big\{\big({\textstyle \frac{n+1-j}{n+1}},1,0\big)\:\Big|\:j=1,\dots,n
\Big\},
$$
such that the orientation in the $y$-direction near the boundary points $\big(\frac{m+1-i}{m+1},0,0\big)$
and $\big(\frac{n+1-j}{n+1},1,0\big)$ are $\a_{i}$ and $\b_j$,
respectively.
We will draw such tangles by 
projecting onto the $xy$-plane in such a way that the implicit framing is
``blackboard,'' and
there are no triple
intersections or tangencies; we also keep track of
``over'' or ``under'' data at each crossing.
We call the resulting diagrams {\em
  $(\a,\b)$-ribbons} for short. For 
example, 
here is a
$(\down \up\up\down,\down \down\up\up)$-ribbon:
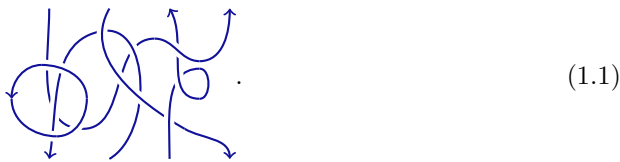
\begin{equation}
\mathord{\begin{tikzpicture}[baseline = 0]
\draw[-,thick,darkblue] (0,1) to [out=-90,in=180] (.45,-0.6);
\draw[-,thick,darkblue] (.45,-.6) to [out=0,in=180] (1.4,0.6);
\draw[->,thick,darkblue] (2,.3) to [out=0,in=-90] (2.4,1);
\draw[-,thick,darkblue] (2,0.2) to [out=0,in=0] (2,-.2);
\draw[-,line width=4pt,white] (0,-1) to [out=80,in=190] (.7,.7);
\draw[<-,thick,darkblue] (0,-1) to [out=80,in=190] (.7,.7);
\draw[-,line width=4pt,white] (.7,.7) to [out=10,in=30] (0.8,-1);
\draw[-,thick,darkblue] (.7,.7) to [out=10,in=30] (0.8,-1);
\draw[->,thick,darkblue] (1.8,-0.6) to [out=-30,in=90] (2.4,-1);
\draw[-,line width=4pt,white] (0.8,1) to [out=-120,in=150] (1.8,-0.6);
\draw[-,thick,darkblue] (0.8,1) to [out=-120,in=150] (1.8,-0.6);
\draw[-,line width=4pt,white] (1.6,-1) to [out=90,in=-180] (2,0.2);
\draw[-,thick,darkblue] (1.6,-1) to [out=90,in=-180] (2,0.2);
\draw[-,line width=4pt,white] (2,-.2) to [out=180,in=-60] (1.6,1);
 \draw[->,thick,darkblue] (2,-.2) to [out=180,in=-60] (1.6,1);
\draw[-,line width=4pt,white] (-.5,-.2) to [out=90,in=180] (-.1,.25); 
\draw[-,line width=4pt,white] (-.1,.25) to [out=0,in=90] (.5,-.2); 
\draw[<-,thick,darkblue] (-.5,-.2) to [out=90,in=180] (-.1,.25); 
\draw[-,thick,darkblue] (-.1,.25) to [out=0,in=90] (.5,-.2); 
\draw[-,line width=4pt,white] (.1,-.7) to [out=0,in=-90] (.5,-.2); 
\draw[-,thick,darkblue] (.1,-.7) to [out=0,in=-90] (.5,-.2); 
\draw[-,line width=4pt,white] (-.5,-.2) to [out=-90,in=180] (.1,-.7); 
\draw[-,thick,darkblue] (-.5,-.2) to [out=-90,in=180] (.1,-.7); 
\draw[-,line width=4pt,white] (1.4,.6) to [out=0,in=180] (2,0.3);
\draw[-,thick,darkblue] (1.4,.6) to [out=0,in=180] (2,0.3);
\end{tikzpicture}}\,.
\label{years}\end{equation}
Isotopy translates into the
equivalence relation on diagrams
generated by planar isotopy fixing the
boundary, together with the oriented Reidemeister moves ($\mathrm{FR}$I)
({\em not} the full $(\text{R}\text{I})$ due to framing!),
 ($\mathrm{R}$II) and
($\mathrm{R}$III) displayed in Figure 1.
Composition of morphisms in $\FOT$ is given by vertically
stacking diagrams, i.e.,
$f \circ g:= \substack{\displaystyle f \\ \displaystyle g}$, while tensor product is given by
horizontal concatenation, i.e., $f \otimes g := fg$.

\begin{figure}
\begin{align*}
\mathord{
\begin{tikzpicture}[baseline = -.5mm]
  \draw[-,thick,darkblue] (0.3,0) to (0.3,.4);
	\draw[-,thick,darkblue] (0.3,0) to[out=-90, in=0] (0.1,-0.4);
	\draw[-,thick,darkblue] (0.1,-0.4) to[out = 180, in = -90] (-0.1,0);
	\draw[-,thick,darkblue] (-0.1,0) to[out=90, in=0] (-0.3,0.4);
	\draw[-,thick,darkblue] (-0.3,0.4) to[out = 180, in =90] (-0.5,0);
  \draw[-,thick,darkblue] (-0.5,0) to (-0.5,-.4);
\end{tikzpicture}
}
&=
\mathord{\begin{tikzpicture}[baseline=-.5mm]
  \draw[-,thick,darkblue] (0,-0.4) to (0,.4);
\end{tikzpicture}
}=\mathord{
\begin{tikzpicture}[baseline = -.5mm]
  \draw[-,thick,darkblue] (0.3,0) to (0.3,-.4);
	\draw[-,thick,darkblue] (0.3,0) to[out=90, in=0] (0.1,0.4);
	\draw[-,thick,darkblue] (0.1,0.4) to[out = 180, in = 90] (-0.1,0);
	\draw[-,thick,darkblue] (-0.1,0) to[out=-90, in=0] (-0.3,-0.4);
	\draw[-,thick,darkblue] (-0.3,-0.4) to[out = 180, in =-90] (-0.5,0);
  \draw[-,thick,darkblue] (-0.5,0) to (-0.5,.4);
\end{tikzpicture}
}\,,
&
\mathord{
\begin{tikzpicture}[baseline = -1mm]
\draw[-,thick,darkblue](-.5,.4) to (0,-.3);
	\draw[-,thick,darkblue] (0.3,-0.3) to[out=90, in=0] (0,0.2);
	\draw[-,line width=4pt,white] (0,0.2) to[out = -180, in = 40] (-0.5,-0.3);
	\draw[-,thick,darkblue] (0,0.2) to[out = -180, in = 40] (-0.5,-0.3);
\end{tikzpicture}
}&=
\mathord{
\begin{tikzpicture}[baseline = -1mm]
\draw[-,thick,darkblue](.6,.4) to (.1,-.3);
	\draw[-,line width=4pt,white] (0.6,-0.3) to[out=140, in=0] (0.1,0.2);
	\draw[-,thick,darkblue] (0.6,-0.3) to[out=140, in=0] (0.1,0.2);
	\draw[-,thick,darkblue] (0.1,0.2) to[out = -180, in = 90] (-0.2,-0.3);
\end{tikzpicture}
}\,,
&
\mathord{
\begin{tikzpicture}[baseline = -1mm]
	\draw[-,thick,darkblue] (0,0.2) to[out = -180, in = 40] (-0.5,-0.3);
	\draw[-,thick,darkblue] (0.3,-0.3) to[out=90, in=0] (0,0.2);
\draw[-,line width=4pt,white](-.5,.4) to (0,-.3);
\draw[-,thick,darkblue](-.5,.4) to (0,-.3);
\end{tikzpicture}
}&=
\mathord{
\begin{tikzpicture}[baseline = -1mm]
	\draw[-,thick,darkblue] (0.6,-0.3) to[out=140, in=0] (0.1,0.2);
	\draw[-,thick,darkblue] (0.1,0.2) to[out = -180, in = 90] (-0.2,-0.3);
\draw[-,line width=4pt,white](.6,.4) to (.1,-.3);
\draw[-,thick,darkblue](.6,.4) to (.1,-.3);
\end{tikzpicture}
}\:\:\text{for all orientations.}\!
\qquad(\text{R}0)
\end{align*}
$$
\qquad\qquad\qquad\quad\qquad\qquad
\qquad\qquad
\mathord{
\begin{tikzpicture}[baseline = -3.5mm]
	\draw[-,thick,darkblue] (0,-0.9) to (0,-0.6);
	\draw[<-,thick,darkblue] (0,.3) to (0,-0.1);
	\draw[-,thick,darkblue] (-0.5,-.3) to [out=90,in=180](-.3,-.1);
	\draw[-,thick,darkblue] (-0.3,-0.5) to [out=180,in=-90](-.5,-0.3);
	\draw[-,thick,darkblue] (-0.3,-.1) to [out=0,in=90](0,-0.6);
	\draw[-,line width=4pt,white] (0,0) to [out=-90,in=0] (-.3,-0.5);
	\draw[-,thick,darkblue] (0,0) to [out=-90,in=0] (-.3,-0.5);
\end{tikzpicture}
}
=
\mathord{
\begin{tikzpicture}[baseline = -0.5mm]
	\draw[<-,thick,darkblue] (0,0.6) to (0,-0.6);
\end{tikzpicture}}
=
\mathord{
\begin{tikzpicture}[baseline = -0.5mm]
	\draw[<-,thick,darkblue] (0,0.6) to (0,0.3);
	\draw[-,thick,darkblue] (0.5,0) to [out=90,in=0](.3,0.2);
	\draw[-,thick,darkblue] (0,-0.3) to (0,-0.6);
	\draw[-,thick,darkblue] (0.3,-0.2) to [out=0,in=-90](.5,0);
	\draw[-,thick,darkblue] (0,0.3) to [out=-90,in=180] (.3,-0.2);
	\draw[-,line width=4pt,white] (0.3,.2) to [out=180,in=90](0,-0.3);
	\draw[-,thick,darkblue] (0.3,.2) to [out=180,in=90](0,-0.3);
\end{tikzpicture}
}
\,.
\qquad
\qquad\qquad\qquad\qquad\qquad
\text{($\text{R}$I)}
$$
\begin{align*}
\qquad\quad\mathord{
\begin{tikzpicture}[baseline = -1mm]
	\draw[-,thick,darkblue] (0.28,-.6) to[out=90,in=-90] (-0.28,0);
	\draw[->,thick,darkblue] (-0.28,0) to[out=90,in=-90] (0.28,.6);
	\draw[-,line width=4pt,white] (-0.28,-.6) to[out=90,in=-90] (0.28,0);
	\draw[-,thick,darkblue] (-0.28,-.6) to[out=90,in=-90] (0.28,0);
	\draw[-,line width=4pt,white] (0.28,0) to[out=90,in=-90] (-0.28,.6);
	\draw[->,thick,darkblue] (0.28,0) to[out=90,in=-90] (-0.28,.6);
\end{tikzpicture}
}&=
\mathord{
\begin{tikzpicture}[baseline = -1mm]
	\draw[->,thick,darkblue] (0.18,-.6) to (0.18,.6);
	\draw[->,thick,darkblue] (-0.18,-.6) to (-0.18,.6);
\end{tikzpicture}
}\:,\quad&
\mathord{
\begin{tikzpicture}[baseline = -1mm]
	\draw[->,thick,darkblue] (0.28,0) to[out=90,in=-90] (-0.28,.6);
	\draw[-,line width=4pt,white] (-0.28,0) to[out=90,in=-90] (0.28,.6);
	\draw[->,thick,darkblue] (-0.28,0) to[out=90,in=-90] (0.28,.6);
	\draw[-,thick,darkblue] (-0.28,-.6) to[out=90,in=-90] (0.28,0);
	\draw[-,line width=4pt,white] (0.28,-.6) to[out=90,in=-90] (-0.28,0);
	\draw[-,thick,darkblue] (0.28,-.6) to[out=90,in=-90] (-0.28,0);
\end{tikzpicture}
}&=
\mathord{
\begin{tikzpicture}[baseline = -1mm]
	\draw[->,thick,darkblue] (0.18,-.6) to (0.18,.6);
	\draw[->,thick,darkblue] (-0.18,-.6) to (-0.18,.6);
\end{tikzpicture}
}\:,\quad&
\mathord{
\begin{tikzpicture}[baseline = -1mm]
	\draw[-,thick,darkblue] (0.28,-.6) to[out=90,in=-90] (-0.28,0);
	\draw[->,thick,darkblue] (-0.28,0) to[out=90,in=-90] (0.28,.6);
	\draw[-,line width=4pt,white] (-0.28,-.6) to[out=90,in=-90] (0.28,0);
	\draw[<-,thick,darkblue] (-0.28,-.6) to[out=90,in=-90] (0.28,0);
	\draw[-,line width=4pt,white] (0.28,0) to[out=90,in=-90] (-0.28,.6);
	\draw[-,thick,darkblue] (0.28,0) to[out=90,in=-90] (-0.28,.6);
\end{tikzpicture}
}&=
\mathord{
\begin{tikzpicture}[baseline = -1mm]
	\draw[->,thick,darkblue] (0.18,-.6) to (0.18,.6);
	\draw[<-,thick,darkblue] (-0.18,-.6) to (-0.18,.6);
\end{tikzpicture}
}\:,\quad
&\mathord{
\begin{tikzpicture}[baseline = -1mm]
	\draw[->,thick,darkblue] (0.28,0) to[out=90,in=-90] (-0.28,.6);
	\draw[-,line width=4pt,white] (-0.28,0) to[out=90,in=-90] (0.28,.6);
	\draw[-,thick,darkblue] (-0.28,0) to[out=90,in=-90] (0.28,.6);
	\draw[-,thick,darkblue] (-0.28,-.6) to[out=90,in=-90] (0.28,0);
	\draw[-,line width=4pt,white] (0.28,-.6) to[out=90,in=-90] (-0.28,0);
	\draw[<-,thick,darkblue] (0.28,-.6) to[out=90,in=-90] (-0.28,0);
\end{tikzpicture}
}&=
\mathord{
\begin{tikzpicture}[baseline = -1mm]
	\draw[<-,thick,darkblue] (0.18,-.6) to (0.18,.6);
	\draw[->,thick,darkblue] (-0.18,-.6) to (-0.18,.6);
\end{tikzpicture}
}\:.
\quad\qquad\quad
\:\,(\mathrm{R}\mathrm{II})
\end{align*}
$$\qquad\qquad\:
\qquad\qquad\qquad\qquad\qquad
\mathord{
\begin{tikzpicture}[baseline = -1mm]
	\draw[->,thick,darkblue] (0.45,-.6) to (-0.45,.6);
        \draw[-,thick,darkblue] (0,-.6) to[out=90,in=-90] (-.45,0);
        \draw[-,line width=4pt,white] (-0.45,0) to[out=90,in=-90] (0,0.6);
        \draw[->,thick,darkblue] (-0.45,0) to[out=90,in=-90] (0,0.6);
	\draw[-,line width=4pt,white] (0.45,.6) to (-0.45,-.6);
	\draw[<-,thick,darkblue] (0.45,.6) to (-0.45,-.6);
\end{tikzpicture}
}
=
\mathord{
\begin{tikzpicture}[baseline = -1mm]
	\draw[->,thick,darkblue] (0.45,-.6) to (-0.45,.6);
        \draw[-,line width=4pt,white] (0,-.6) to[out=90,in=-90] (.45,0);
        \draw[-,thick,darkblue] (0,-.6) to[out=90,in=-90] (.45,0);
        \draw[->,thick,darkblue] (0.45,0) to[out=90,in=-90] (0,0.6);
	\draw[-,line width=4pt,white] (0.45,.6) to (-0.45,-.6);
	\draw[<-,thick,darkblue] (0.45,.6) to (-0.45,-.6);
\end{tikzpicture}
}\:.\!\qquad\qquad\quad\qquad\qquad\qquad(\mathrm{R}\mathrm{III})
$$
$$
\qquad\qquad\qquad\qquad\qquad
\qquad\qquad\qquad
\mathord{
\begin{tikzpicture}[baseline = -3.5mm]
	\draw[-,thick,darkblue] (0,-0.9) to (0,-0.6);
	\draw[<-,thick,darkblue] (0,.3) to (0,-0.1);
	\draw[-,thick,darkblue] (-0.5,-.3) to [out=90,in=180](-.3,-.1);
	\draw[-,thick,darkblue] (-0.3,-0.5) to [out=180,in=-90](-.5,-0.3);
	\draw[-,thick,darkblue] (-0.3,-.1) to [out=0,in=90](0,-0.6);
	\draw[-,line width=4pt,white] (0,0) to [out=-90,in=0] (-.3,-0.5);
	\draw[-,thick,darkblue] (0,0) to [out=-90,in=0] (-.3,-0.5);
\end{tikzpicture}
}
=
\mathord{
\begin{tikzpicture}[baseline = -0.5mm]
	\draw[<-,thick,darkblue] (0,0.6) to (0,0.3);
	\draw[-,thick,darkblue] (0.5,0) to [out=90,in=0](.3,0.2);
	\draw[-,thick,darkblue] (0,-0.3) to (0,-0.6);
	\draw[-,thick,darkblue] (0.3,-0.2) to [out=0,in=-90](.5,0);
	\draw[-,thick,darkblue] (0,0.3) to [out=-90,in=180] (.3,-0.2);
	\draw[-,line width=4pt,white] (0.3,.2) to [out=180,in=90](0,-0.3);
	\draw[-,thick,darkblue] (0.3,.2) to [out=180,in=90](0,-0.3);
\end{tikzpicture}
}
\:.\!
\qquad\qquad\qquad
\qquad\qquad\qquad\text{($\text{FR}$I)}
$$
\begin{align*}
\qquad\qquad
\qquad\qquad\qquad
\quad\mathord{
\begin{tikzpicture}[baseline = -.5mm]
	\draw[->,thick,darkblue] (0.28,-.3) to (-0.28,.4);
	\draw[line width=4pt,white,-] (-0.28,-.3) to (0.28,.4);
	\draw[thick,darkblue,->] (-0.28,-.3) to (0.28,.4);
\end{tikzpicture}
}-\mathord{
\begin{tikzpicture}[baseline = -.5mm]
	\draw[thick,darkblue,->] (-0.28,-.3) to (0.28,.4);
	\draw[line width=4pt,white,-] (0.28,-.3) to (-0.28,.4);
	\draw[->,thick,darkblue] (0.28,-.3) to (-0.28,.4);
\end{tikzpicture}
}&=
z\:\mathord{
\begin{tikzpicture}[baseline = -.5mm]
	\draw[->,thick,darkblue] (0.18,-.3) to (0.18,.4);
	\draw[->,thick,darkblue] (-0.18,-.3) to (-0.18,.4);
\end{tikzpicture}
}\,.
\qquad\qquad\qquad\qquad
\qquad\qquad\qquad\text{($\mathrm{S}$)}
\end{align*}
\begin{align*}
\:\qquad\qquad\quad\qquad\qquad
\qquad\qquad\qquad
\mathord{
\begin{tikzpicture}[baseline = -0.5mm]
	\draw[<-,thick,darkblue] (0,0.6) to (0,0.3);
	\draw[-,thick,darkblue] (0.5,0) to [out=90,in=0](.3,0.2);
	\draw[-,thick,darkblue] (0,-0.3) to (0,-0.6);
	\draw[-,thick,darkblue] (0.3,-0.2) to [out=0,in=-90](.5,0);
	\draw[-,thick,darkblue] (0,0.3) to [out=-90,in=180] (.3,-0.2);
	\draw[-,line width=4pt,white] (0.3,.2) to [out=180,in=90](0,-0.3);
	\draw[-,thick,darkblue] (0.3,.2) to [out=180,in=90](0,-0.3);
\end{tikzpicture}
}
=
t\:\mathord{
\begin{tikzpicture}[baseline = -0.5mm]
	\draw[<-,thick,darkblue] (0,0.6) to (0,-0.6);
\end{tikzpicture}
}\;.
\qquad\,
\qquad\qquad\qquad
\qquad\qquad\qquad\text{($\mathrm{T}$)}
\end{align*}
$$\qquad\qquad\qquad\qquad
\qquad\qquad\qquad\quad \Bubble
=
\frac{t-t^{-1}}{z}\; 1_\varnothing.
\qquad\qquad
\quad\qquad\qquad\qquad\text{($\mathrm{D}$)}
$$
\begin{align*}
\hspace{50mm}
\mathord{
\begin{tikzpicture}[baseline = -.5mm]
	\draw[thick,darkblue,->] (-0.28,-.3) to (0.28,.4);
      \node at (-0.16,-0.15) {$\color{darkblue}\scriptstyle\bullet$};
	\draw[-,line width=4pt,white] (0.28,-.3) to (-0.28,.4);
	\draw[->,thick,darkblue] (0.28,-.3) to (-0.28,.4);
\end{tikzpicture}
}&= 
\mathord{
\begin{tikzpicture}[baseline = -.5mm]
	\draw[->,thick,darkblue] (0.28,-.3) to (-0.28,.4);
	\draw[line width=4pt,white,-] (-0.28,-.3) to (0.28,.4);
	\draw[thick,darkblue,->] (-0.28,-.3) to (0.28,.4);
      \node at (0.145,0.23) {$\color{darkblue}\scriptstyle\bullet$};
\end{tikzpicture}
}\:.
\hspace{47.5mm}
(\text{A})
\end{align*}
\caption{Reidemeister-type relations}
\end{figure}
 
Now let $\k$ be some fixed commutative ground ring
and fix parameters 
$z,t \in \k^\times$.
The {\em extended oriented skein
  category} $\widehat{\OS}(z,t)$ is the quotient of the $\k$-linearization
of $\FOT$ by the $\k$-linear tensor ideal generated by the {Conway skein relation} ($\text{S}$) and the {twist
  relation}
($\text{T}$), both of which are displayed in Figure 1.
These relations imply that
\begin{equation}\label{beautiful}
(t-t^{-1})\:
\mathord{
\begin{tikzpicture}[baseline = -.5mm]
	\draw[->,thick,darkblue] (0.08,-.6) to (0.08,.6);
\end{tikzpicture}
}=
\mathord{
\begin{tikzpicture}[baseline = -0.5mm]
	\draw[<-,thick,darkblue] (0,0.6) to (0,0.3);
	\draw[-,thick,darkblue] (0.5,0) to [out=90,in=0](.3,0.2);
	\draw[-,thick,darkblue] (0,-0.3) to (0,-0.6);
	\draw[-,thick,darkblue] (0.3,-0.2) to [out=0,in=-90](.5,0);
	\draw[-,thick,darkblue] (0,0.3) to [out=-90,in=180] (.3,-0.2);
	\draw[-,line width=4pt,white] (0.3,.2) to [out=180,in=90](0,-0.3);
	\draw[-,thick,darkblue] (0.3,.2) to [out=180,in=90](0,-0.3);
\end{tikzpicture}
}
-
\mathord{
\begin{tikzpicture}[baseline = -0.5mm]
	\draw[<-,thick,darkblue] (0,0.6) to (0,0.3);
	\draw[-,thick,darkblue] (0.5,0) to [out=90,in=0](.3,0.2);
	\draw[-,thick,darkblue] (0,-0.3) to (0,-0.6);
	\draw[-,thick,darkblue] (0.3,-0.2) to [out=0,in=-90](.5,0);
	\draw[-,thick,darkblue] (0.3,.2) to [out=180,in=90](0,-0.3);
	\draw[-,line width=4pt,white] (0,0.3) to [out=-90,in=180] (.3,-0.2);
	\draw[-,thick,darkblue] (0,0.3) to [out=-90,in=180] (.3,-0.2);
\end{tikzpicture}
}
=
z \:
\mathord{
\begin{tikzpicture}[baseline = -.5mm]
	\draw[->,thick,darkblue] (0.08,-.6) to (0.08,.6);
\end{tikzpicture}
}\:\:\:
\mathord{
\begin{tikzpicture}[baseline = 1.25mm]
  \draw[<-,thick,darkblue] (0.2,0.2) to[out=90,in=0] (0,.4);
  \draw[-,thick,darkblue] (0,0.4) to[out=180,in=90] (-.2,0.2);
\draw[-,thick,darkblue] (-.2,0.2) to[out=-90,in=180] (0,0);
  \draw[-,thick,darkblue] (0,0) to[out=0,in=-90] (0.2,0.2);
\end{tikzpicture}
}\:.
\end{equation}
Usually, we will impose the additional {dimension
  relation}
($\text{D}$) from Figure 1.
We call the resulting category the (reduced) {\em oriented skein
  category}, and denote it simply by $\OS(z,t)$. This is the main
object of study in this article.

With a different normalization of crossings, the category $\OS(z,t)$ was introduced in
\cite[$\S$5.2]{Turaev3}, where it is called the {\em Hecke category}.
In \cite[Definition 2.1]{QS} it is called the {\em quantized oriented Brauer category}.
Others call $\OS(z,t)$ the {\em framed HOMFLY-PT skein category}.

\vspace{2mm}

\noindent{\bf 1.2.}
Let us state some foundational results about $\OS(z,t)$.
The first one gives an efficient monoidal presentation.
It is a corollary of a more general result of Turaev \cite[Lemma
I.3.3]{Turaev1} which
gives a presentation for the category $\FOT$.

\begin{theorem}\label{first}
The oriented skein category $\OS(z,t)$ is isomorphic to the
strict $\k$-linear monoidal
category generated by objects $E$ and $F$ and morphisms
$$
S:E\otimes E \rightarrow E\otimes E,\quad
T:F \otimes E \rightarrow E \otimes F,\quad
C:\unit \rightarrow F \otimes E,\quad
D:E \otimes F \rightarrow \unit,
$$
subject to the following relations:
\begin{itemize}
\item[(1)] 
$S^2 =z S + 1_E \otimes 1_E$;
\item[(2)] $(S \otimes 1_E) \circ (1_E \otimes S) \circ (S \otimes
  1_E) = (1_E \otimes S) \circ (S \otimes 1_E) \circ (1_E \otimes S)$;
\item[(3)] $(D \otimes 1_E) \circ (1_E \otimes C) = 1_E$,
 $(1_F \otimes D) \circ (C \otimes 1_F) = 1_F$;
\item[(4)]
$T^{-1} = (1_F \otimes 1_E  \otimes D) \circ (1_F \otimes
  S  \otimes 1_F) \circ (C \otimes 1_E \otimes
  1_F)$ (two-sided inverse);
\item[(5)] $t D \circ T \circ C = \frac{t-t^{-1}}{z} 1_\unit$.
\end{itemize}
\end{theorem}
 
An explicit functor giving an isomorphism from the monoidal category with this presentation to
$\OS(z,t)$ sends
$E \mapsto \up$, $F \mapsto \down$, and the
generating morphisms $S, T, C$ and $D$ to 
$\mathord{
\begin{tikzpicture}[baseline = -.5mm]
	\draw[->,thick,darkblue] (0.2,-.2) to (-0.2,.3);
	\draw[line width=4pt,white,-] (-0.2,-.2) to (0.2,.3);
	\draw[thick,darkblue,->] (-0.2,-.2) to (0.2,.3);
\end{tikzpicture}
}$, 
$\mathord{
\begin{tikzpicture}[baseline = -.5mm]
	\draw[->,thick,darkblue] (0.2,-.2) to (-0.2,.3);
	\draw[line width=4pt,white,-] (-0.2,-.2) to (0.2,.3);
	\draw[thick,darkblue,<-] (-0.2,-.2) to (0.2,.3);
\end{tikzpicture}
}$,
$\mathord{
\begin{tikzpicture}[baseline = 1mm]
	\draw[<-,thick,darkblue] (0.3,0.35) to[out=-90, in=0] (0.05,0);
	\draw[-,thick,darkblue] (0.05,0) to[out = 180, in = -90] (-0.2,0.35);
\end{tikzpicture}
}
$ and $\mathord{
\begin{tikzpicture}[baseline = 1mm]
	\draw[<-,thick,darkblue] (0.3,0) to[out=90, in=0] (0.05,0.35);
	\draw[-,thick,darkblue] (0.05,0.35) to[out = 180, in = 90] (-0.2,0);
\end{tikzpicture}
}\:$,  respectively.
We strongly encourage the reader to verify that the relations (1)--(5)
from  the theorem
all hold in $\OS(z,t)$ by drawing the appropriate pictures!

\vspace{2mm}
\noindent{\bf 1.3.}
The next theorem gives bases for morphism spaces. Again, this is due to 
Turaev \cite[Theorems 5.1 and 5.2.3]{Turaev3}; Turaev notes that it
was also proved independently by Morton and Traczyk.
To formulate it, 
given $\a = \a_m\cdots \a_1, \b = \b_n\cdots \b_1\in\words$,
an {\em $(\a,\b)$-matching} means a bijection
\begin{multline*}
\left\{\big({\textstyle\frac{m+1-i}{m+1}},0,0\big)\:\Big|\:\a_i=\up\right\}\cup\left\{
\big({\textstyle\frac{n+1-j}{n+1}},1,0\big)\:\Big|\: \b_{j} = \down\right\}
\\\stackrel{\sim}{\longrightarrow}
\left\{\big({\textstyle\frac{m+1-i}{m+1}},0,0\big)\:\Big|\:\a_i=\down\right\}\cup\left\{\big({\textstyle\frac{n+1-j}{n+1}},1,0\big)\:\Big|\: \b_{j} = \up\right\}.
\end{multline*}
There are no such bijections unless the domain and codomain have the
same size $d$, in which case there are $d!$ possibilities.
An $(\a,\b)$-ribbon is a {\em lift} of a given
$(\a,\b)$-matching if 
the boundary of each strand in the ribbon 
consists of a pair of points which correspond under the matching;
in particular, it contains no ``floating bubbles.''
An $(\a,\b)$-ribbon is {\em reduced} if
no strand crosses itself and no two
strands cross more than once.

\begin{theorem}\label{bt}
The morphism space $\Hom_{\OS(z,t)}(\a,\b)$ is free as a
$\k$-module with
basis given by any set consisting of a reduced lift for each of the
$(\a,\b)$-matchings.
The same is true in $\widehat{\OS}(z,t)$ with one exception:
if $\a = \b = \varnothing$ then the morphism space
$\Hom_{\widehat\OS(q,t)}(\varnothing, \varnothing)$ is free of rank
two with basis $\left\{\,1_\varnothing,\Bubble\:\right\}$.
\end{theorem}

The algebra $\Hom_{\widehat\OS(z,t)}(\varnothing, \varnothing)$
appearing in Theorem~\ref{bt}
is
known in the literature as the {\em Conway skein module}
\cite{Turaev0}
or the {\em framed HOMFLY-PT skein module} \cite[Definition 2.1]{MS}
of the manifold $\mathbb{R}^3$. 
The basis $\left\{\:1_\varnothing,\Bubble\:\right\}$ for 
it described in the theorem is implicit already in \cite{HOMFLY,
  PT}, indeed, the existence
of the HOMFLY-PT polynomial for oriented links constructed in those papers follows 
easily from this result. To explain this briefly, let $L$ be an oriented link diagram, 
and define $\operatorname{writhe}(L)$ as usual to be the number of
positive crossings minus the number of negative crossings. 
Viewing $L$ as a $(\varnothing,\varnothing)$-ribbon,
there is a 
unique scalar
$H(L) \in \k$ such that
$$
t^{-\operatorname{writhe}(L)} L = H(L) \:\Bubble
$$
in $\End_{\widehat\OS(q,t)}(\varnothing)$.
The scalar $H(L)$ is invariant under the Reidemeister
moves 
($\mathrm{R}$I), ($\mathrm{R}$II) and
($\mathrm{R}$III); for all but
($\mathrm{R}$I), this is automatic from the defining relations in
$\FOT$, while
($\mathrm{R}$I) follows from ($\mathrm{T}$) and
($\mathrm{FR}$I).
The relation ($\mathrm{S}$) implies
that
$$
t H(L_+) - t^{-1} H(L_-) = z H(L_0)
$$
for oriented link diagrams $L_+, L_-$ and $L_0$
which agree except in one place, which is a positive crossing in
$L_+$, a negative crossing in $L_-$, and the crossing is resolved
in $L_0$.
This is exactly the skein relation defining the HOMFLY-PT polynomial.
Finally, observe that $H(L) = 1$ in case $L$ is the
unknot. Hence, taking $\k := \ZZ[z,z^{-1},t,t^{-1}]$,
the scalar $H(L)$ is exactly the HOMFLY-PT polynomial of $L$.

\vspace{2mm}
\noindent{\bf 1.4.}
Let $H_r$ be the {\em Iwahori-Hecke algebra} of the
symmetric group
$\Sym_r$ with quadratic relation $S^2 = zS + 1$; if $z = q-q^{-1}$
this can be written equivalently as $(S-q)(S+q^{-1}) = 0$.
There is a homomorphism
\begin{equation}\label{randwick}
\i_r:H_r \rightarrow \End_{\OS(z, t)}(\up^r)
\end{equation}
sending the generator for $H_r$ that corresponds to the $i$th basic
transposition 
to the positive crossing
$\begin{tikzpicture}[baseline = -.5mm]
	\draw[->,thick,darkblue] (0.2,-.2) to (-0.2,.3);
	\draw[line width=4pt,white,-] (-0.2,-.2) to (0.2,.3);
	\draw[thick,darkblue,->] (-0.2,-.2) to (0.2,.3);
\end{tikzpicture}$
of the $i$th and $(i+1)$th strand, numbering strands by $1,\dots,r$ from {right}
to {left}. The main step in Turaev's proof of Theorem~\ref{bt} is to show that $\i_r$
is an isomorphism. This is deduced ultimately from
Jimbo's {\em quantized Schur-Weyl
  reciprocity} from \cite{J}, which connects $H_r$ to the quantized enveloping algebra
$U_q(\mathfrak{gl}_n)$.

In fact, quantized Schur-Weyl reciprocity can be upgraded to the
following well-known result.
For a $\k$-linear category $\mathcal C$, we write $\dot{\mathcal C}$
for its {\em additive Karoubi envelope}, that is, the idempotent
completion of its additive envelope;
in case $\mathcal C$ is monoidal, $\dot{\mathcal C}$ is monoidal too.

\begin{theorem}\label{webs}
Assume that
 $\k$ is a field of characteristic zero, $q \in \k^\times$ is not a root of
unity, $z = q-q^{-1}$, and $t = q^{\eps n}$ for $n \in \NN$ and $\eps
\in \{\pm\}$.
There is a full $\k$-linear monoidal functor
$\Psi:\OS(z,t) \rightarrow 
\Rep U_q(\mathfrak{gl}_{n})$
sending $\up$ to the natural $U_q(\mathfrak{gl}_{n})$-module
and $\down$ to its dual.
It induces a monoidal equivalence 
\begin{equation}\label{tiltingequiv}
\bar\Psi:\dot{\OS}(z,t) / \mathcal N
\stackrel{\approx}{\longrightarrow}
\Rep U_q(\mathfrak{gl}_{n}),
\end{equation}
where
$\mathcal N$ is the tensor ideal
of $\dot \OS(z,t)$ consisting of negligible morphisms (see
\cite[$\S$6.1]{D}).
As an additive
$\k$-linear tensor ideal,
$\mathcal N$
is generated 
by $\i_{n+1}(e)$
where $e
 \in H_{n+1}$
is the Young symmetrizer 
associated to 
the sign representation if $\eps = +$ or the trivial representation if
$\eps = -$.
\end{theorem}

The evident ribbon structure on $\OS(z,t)$ induces a ribbon structure on
$\Rep U_q(\mathfrak{gl}_{n})$ 
so that $\bar\Psi$ is an equivalence of ribbon categories.
This induced ribbon structure 
depends on the sign $\eps$; we denote the resulting ribbon category by $\Rep
U_q(\mathfrak{gl}_{\eps n})$. 
When $\k = \CC$, $q$ is not a root of unity, and $\delta$ is any
complex number,
the category 
\begin{equation}
\REP U_q(\mathfrak{gl}_\delta) := \dot\OS(q-q^{-1}, q^\delta)
\end{equation}
is the $q$-analog of the Deligne category $\REP GL_\delta$
introduced in \cite{DM} (see also \cite[$\S$10]{D}) and
studied recently in \cite{CW, EHS}.
In $\REP U_q(\mathfrak{gl}_\delta)$, 
relation (D) implies that the objects $\up$ and $\down$ have categorical dimension
$$
[\delta]_q := \frac{q^\delta - q^{-\delta}}{q-q^{-1}}.
$$
For $n \in \ZZ$, $[n]_q$ is the usual quantum integer, and
Theorem~\ref{webs} shows that the ribbon category $\Rep U_q(\mathfrak{gl}_{n})$
is a quotient of $\REP U_q(\mathfrak{gl}_n)$.
Thus, the categories $\REP U_q(\mathfrak{gl}_\delta)$
for $\delta \in \CC$
interpolate between the categories $\Rep U_q(\mathfrak{gl}_{n})$.
It is also known that the category $\REP U_q(\mathfrak{gl}_\delta)$ is semisimple 
when $\delta \notin \ZZ$; we will say more about this shortly.

Most of the recent literature on diagrammatic approaches
to $\Rep U_q(\mathfrak{gl}_n)$
focuses
instead on variants of the ``$SL_n$-spider''
of Cautis, Kamnitzer and Morrison from
\cite{CKM}. In \cite[Definition 6.4]{QS}, this $\CC$-linear monoidal category is upgraded to
a 
ribbon category $\mathsf{Sp}(\delta)$ depending on $q \in
\CC^\times$ (not a root of unity) and a parameter $\delta \in \CC$; 
we prefer to denote $\mathsf{Sp}(\delta)$ by
$\mathcal{W}eb(\delta)$.
According to \cite[Proposition 6.7]{QS}, $\mathcal{W}eb(\delta)$ is a thickening (i.e.,
a partial idempotent completion) of $\OS(q-q^{-1},q^\delta)$,
so that
the Deligne category $\REP
U_q(\mathfrak{gl}_\delta)$
may also be realized as the
additive Karoubi envelope of $\mathcal{W}eb(\delta)$. 
Subsequent developments in the literature
have revolved around
2-categorifications 
related to Khovanov-Rozansky homology; e.g., see \cite{MW}.

\vspace{2mm}
\noindent{\bf 1.5.}
We are interested here instead in the decategorification of $\OS(z,t)$. There are two
basic ways to understand this: either by taking the trace, or by
passing to the Grothendieck ring. Let us briefly recall these definitions.

By the {\em trace} of a $\k$-linear category $\mathcal C$, we mean the $\k$-module
$$
\Tr(\mathcal C) := 
\bigoplus_{X \in \ob \mathcal C} \Hom_{\mathcal C}(X,X) \Big/ \Big\langle f \circ g -
g \circ f\:\Big|\:
\begin{array}{l}
\text{for all }X,Y \in \ob
\mathcal C\text{ and}\\
f:X \rightarrow Y, g:Y \rightarrow X
\end{array}
\Big\rangle.
$$
One can represent
the image $[f] \in \Tr(\mathcal C)$ of $f\in \Hom_{\mathcal C}(X,X)$
diagrammatically by drawing $f$ in an annulus:
\begin{equation}
\label{bomber}
\mathord{
\begin{tikzpicture}[baseline = -.5mm]
\draw[-,thick,double,darkblue] (0,-.5) to [out=0,in=-90] (.5,0);
\draw[-,thick,double,darkblue] (.5,0) to [out=90,in=0] (0,.5);
\draw[-,thick,double,darkblue] (0,.5) to [out=180,in=85] (-.48,0.14);
\draw[-,thick,double,darkblue] (-.48,-.14) to [out=-85,in=180] (0,-.5);
\draw[darkblue,thick,yshift=-4pt,xshift=-4pt] (-.48,0) rectangle ++(8pt,8pt);
\node at (-.48,0) {$\scriptstyle f$};
\filldraw (0,0) circle (4pt);
\end{tikzpicture}
}\:.
\end{equation} 
If $\mathcal C$ is a monoidal category, then $\Tr(\mathcal C)$ is a
$\k$-algebra with $[f] [g] := [f \otimes g]$.
Note also that $\Tr(\mathcal C)$ and
$\Tr(\dot{\mathcal C})$ may be identified; see 
\cite[Proposition 3.2]{BGHL}. 

The {\em Grothendieck group}
$K_0(\dot{\mathcal C})$ 
is the $\ZZ$-module generated by isomorphism
classes $[X]$ of objects $X$ in $\dot{\mathcal C}$ modulo the
relations $[X\oplus Y] = [X] + [Y]$.
Also, a {\em $\mathcal C$-module} means a $\k$-linear functor from
$\mathcal C^{\op}$ to the category of $\k$-modules;
we write $\Mod\mathcal C$ for the category of all such modules. 
The Yoneda embedding induces an equivalence between $\dot{\mathcal C}$ 
and the full subcategory $\pMod\mathcal C$ of $\Mod\mathcal C$ consisting 
of finitely generated
projective $\mathcal C$-modules.
So any finitely generated projective $\mathcal C$-module $M$
also defines a class $[M] \in K_0(\dot{\mathcal C})$.

The notions of trace and Grothendieck group are related by
the  {\em character map}
\begin{equation}
h:K_0(\dot{\mathcal C}) \otimes_\ZZ \k \rightarrow
\Tr(\mathcal C), \qquad
[X] \mapsto [1_X].
\end{equation}
Typically, this map is injective, e.g., it is so if 
$\k$ is an algebraically closed field
and all of the morphism
spaces of $\mathcal C$ are finite-dimensional; see \cite[Proposition
2.4]{BHLW}.
If in addition $\dot{\mathcal C}$ is semisimple then $h$ is an isomorphism; 
see \cite[Proposition 2.5]{BHLW} for a more general
statement here.
In case $\mathcal C$ is monoidal, $K_0(\dot{\mathcal C})$ is a ring
with $[X][Y] := [X \otimes Y]$,
and $h$ is a ring homomorphism.

The trace of $\OS(z,t)$ was computed originally by Turaev
\cite[Theorem 2]{Turaev0}, albeit from a rather different point
of view: 
it is exactly the Conway skein module of the solid torus, as follows
by contemplating the picture (\ref{bomber}).
Turaev's result can be formulated as follows.

\begin{theorem}\label{TR}
The algebra $\Tr(\OS(z,t))$ is the free polynomial algebra 
$\k[u_n, v_n\:|\:n \geq 1]$
generated by the trace classes $u_n$ and $v_n$ of the following ``cycles'' for all $n \geq 1$:
\begin{align*}
u_n&\leftrightarrow\underbrace{\mathord{
\begin{tikzpicture}[baseline =-1.5mm]
	\draw[->,thick,darkblue] (.5,-.5) to (.2,.5);
	\draw[->,thick,darkblue] (.2,-.5) to (-.1,.5);
	\draw[->,thick,darkblue] (-.4,-.5) to (-.7,.5);
	\draw[-,line width=4pt,white] (-.7,-.5) to (.5,.5);
	\draw[->,thick,darkblue] (-.7,-.5) to (.5,.5);
      \node at (-.1,-0.45) {$\scriptstyle\cdots$};
      \node at (-.35,0.45) {$\scriptstyle\cdots$};
\end{tikzpicture}
}}_{\text{$n$ strands}},&
v_n&\leftrightarrow\underbrace{\mathord{
\begin{tikzpicture}[baseline =-1.5mm]
	\draw[<-,thick,darkblue] (.5,-.5) to (.2,.5);
	\draw[<-,thick,darkblue] (.2,-.5) to (-.1,.5);
	\draw[<-,thick,darkblue] (-.4,-.5) to (-.7,.5);
	\draw[-,line width=4pt,white] (-.7,-.5) to (.5,.5);
	\draw[<-,thick,darkblue] (-.7,-.5) to (.5,.5);
      \node at (-.1,-0.45) {$\scriptstyle\cdots$};
      \node at (-.35,0.45) {$\scriptstyle\cdots$};
\end{tikzpicture}
}}_{\text{$n$ strands}}.
\end{align*}
\end{theorem}

 Theorem~\ref{bt} implies that the algebra
$B_{r,s} := \End_{\OS(z,t)}(\down^s\:\up^r)$
is free as a $\k$-module of rank $(r+s)!$.
This is 
the {\em quantized walled Brauer algebra} introduced originally by
Kosuda and Murakami in \cite{KM1, KM2}. 
%
As pointed out by Morton \cite{M2}, any 
$[f]\in\operatorname{Tr}(\OS(z,t))$ 
defines a central element
$$
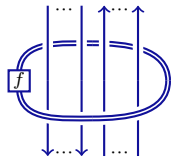
\mathord{
\begin{tikzpicture}[baseline = -.5mm]
\draw[-,thick,double,darkblue] (1,0) to [out=90,in=0] (0,.5);
\draw[-,thick,double,darkblue] (0,.5) to [out=180,in=85] (-.98,0.14);
\draw[-,thick,darkblue] (.6,-1) to (.6,0);
\draw[-,thick,darkblue] (.15,-1) to (.15,0);
\draw[<-,thick,darkblue] (-.15,-1) to (-.15,0);
\draw[<-,thick,darkblue] (-.6,-1) to (-.6,0);
      \node at (-.36,-0.95) {$\scriptstyle\cdots$};
      \node at (.38,-0.95) {$\scriptstyle\cdots$};
      \node at (-.36,0.95) {$\scriptstyle\cdots$};
      \node at (.38,.95) {$\scriptstyle\cdots$};
\draw[-,line width=5pt,white] (0,-.5) to [out=0,in=-90] (1,0);
\draw[-,line width=5pt,white] (-.98,-.14) to [out=-85,in=180] (0,-.5);
\draw[-,thick,double,darkblue] (0,-.5) to [out=0,in=-90] (1,0);
\draw[-,thick,double,darkblue] (-.98,-.14) to [out=-85,in=180] (0,-.5);
\draw[darkblue,thick,yshift=-4pt,xshift=-4pt] (-.98,0) rectangle ++(8pt,8pt);
\node at (-.98,0) {$\scriptstyle f$};
\draw[-,line width=4pt,white] (.6,0) to (.6,1);
\draw[-,line width=4pt,white] (.15,0) to (.15,1);
\draw[-,line width=4pt,white] (-.15,0) to (-.15,1);
\draw[-,line width=4pt,white] (-.6,0) to (-.6,1);
\draw[->,thick,darkblue] (.6,0) to (.6,1);
\draw[->,thick,darkblue] (.15,0) to (.15,1);
\draw[-,thick,darkblue] (-.15,0) to (-.15,1);
\draw[-,thick,darkblue] (-.6,0) to (-.6,1);
\end{tikzpicture}
}
$$
in $B_{r,s}$.
Morton conjectured that these elements generate the entire center
$Z(B_{r,s})$.
Morton's conjecture has recently been proved in \cite{JK} assuming
$\k$ is a field of characteristic zero and $z,t$ are generic. 
In fact, Jung
and Kim show that $Z(B_{r,s})$ is generated already by the
supersymmetric power sums $p_n(X_1,\dots,X_r|Y_{1},\dots,Y_{s}) = X_1^n+\cdots+X_r^n-Y_{1}^n - \cdots -
Y_{s}^n$
in the {\em Jucys-Murphy elements}
\begin{equation}\label{salsa}
X_i := 
\mathord{
\begin{tikzpicture}[baseline =-1.5mm]
	\draw[<-,thick,darkblue] (.2,.5) to[out=-90,in=90] (1.2,0);
	\draw[<-,thick,darkblue] (-1.1,-.5) to (-1.1,.5);
	\draw[<-,thick,darkblue] (-.6,-.5) to (-.6,.5);
	\draw[->,thick,darkblue] (-.3,-.5) to (-.3,.5);
	\draw[-,line width=4pt,white] (.5,-.5) to (.5,.5);
	\draw[-,line width=4pt,white] (1,-.5) to (1,.5);
	\draw[->,thick,darkblue] (.5,-.5) to (.5,.5);
	\draw[->,thick,darkblue] (1,-.5) to (1,.5);
      \node at (-.83,-0.45) {$\scriptstyle\cdots$};
      \node at (.77,-0.45) {$\scriptstyle\cdots$};
      \node at (-.03,-0.45) {$\scriptstyle\cdots$};
      \node at (-.83,0.45) {$\scriptstyle\cdots$};
      \node at (.77,0.45) {$\scriptstyle\cdots$};
      \node at (-.03,0.45) {$\scriptstyle\cdots$};
	\draw[-,line width=4pt,white] (.2,-.5) to[out=90,in=-90] (1.2,0);
	\draw[-,thick,darkblue] (.2,-.5) to[out=90,in=-90] (1.2,0);
\end{tikzpicture}
}
\:,
\qquad
Y_j :=
t^{-2}\:
\mathord{
\begin{tikzpicture}[baseline =-1.5mm]
	\draw[<-, thick,darkblue] (-.6,-.5) to [out=90,in=-90] (1.2,0);
	\draw[<-,thick,darkblue] (-1.1,-.5) to (-1.1,.5);
      \node at (-.83,-0.45) {$\scriptstyle\cdots$};
      \node at (.78,-0.45) {$\scriptstyle\cdots$};
      \node at (-.03,-0.45) {$\scriptstyle\cdots$};
      \node at (-.83,0.45) {$\scriptstyle\cdots$};
      \node at (.78,0.45) {$\scriptstyle\cdots$};
      \node at (-.03,0.45) {$\scriptstyle\cdots$};
	\draw[-, line width=4pt,white] (-.3,-.5) to (-.3,.5);
	\draw[-, line width=4pt,white] (.2,-.5) to (.2,.5);
	\draw[-, line width=4pt,white] (.5,-.5) to (.5,.5);
	\draw[-, line width=4pt,white] (1,-.5) to (1,.5);
	\draw[<-, thick,darkblue] (-.3,-.5) to (-.3,.5);
	\draw[<-, thick,darkblue] (.2,-.5) to (.2,.5);
	\draw[->, thick,darkblue] (.5,-.5) to (.5,.5);
	\draw[->, thick,darkblue] (1,-.5) to (1,.5);
	\draw[-, line width=4pt,white] (-.6,.5) to [out=-90,in=90] (1.2,0);
	\draw[-, thick,darkblue] (-.6,.5) to [out=-90,in=90] (1.2,0);
\end{tikzpicture}
}\:,
\end{equation}
where the interesting strand is the $i$th or
 $(r+j)$th from the right, respectively.
These elements were also introduced by Morton (extending an observation 
from
\cite{Ram} in the case of the Iwahori-Hecke algebra): up to an
obvious symmetry and rescaling they are the elements $T$ and $U$
from the proof of \cite[Theorem 1]{M2}; see \cite[Remark
6.7]{JK}.
(Jung and Kim  also prove a version of
\cite[Conjecture 7.4]{SS} in the degenerate case.)
Later in the article, we will give a more conceptual interpretation of
Jucys-Murphy elements 
based on another monoidal category, the {\em affine oriented skein
  category} $\AOS(z,t)$, which is of independent interest.

\vspace{2mm}

\noindent{\bf 1.7.}
{\em For the remainder of the introduction}, we assume that $\k$ is a field
and $z = q-q^{-1}$ for $q \in \k^\times\setminus\{\pm 1\}$.
The next theorem describes $K_0(\dot\OS(z,t))$ in all semisimple cases.
It is also possible to
compute the irreducible characters 
$h(\chi_\LA) \in \k[u_n,v_n\:|\:n \geq 1]$ by an algorithm involving
Starkey's rule \cite{Geck}.

To state the theorem,
let $\Par = \bigsqcup_{r,s \geq 0} \Par_{r,s}$ where $\Par_{r,s}$
consists of {\em bipartitions} $\LA = (\la^\up, \la^\down)$
for $\la^\up \vdash r$ and $\la^\down \vdash s$.
Let $\SYM$ be the ring of symmetric functions and
denote the Schur function associated to a partition $\lambda$ by
$\s_\lambda$.
The structure constants for this basis of $\SYM$ are the Littlewood-Richardson
coefficients:
$\s_\mu \s_\nu = \sum_\lambda LR^\lambda_{\mu,\nu} \s_\lambda$.
Let $\lambda^\trans$ denote the conjugate partition to $\lambda$.

\begin{theorem}\label{sscase}
The category $\dot\OS(z,t)$ is semisimple if and only if $q$ is not a
root of unity and $t \notin \{\pm q^n\:|\:n \in \ZZ\}$.
Assuming this is the case, the isomorphism classes of indecomposable objects in
$\dot\OS(z,t)$
are parametrized in a canonical way by $\Par$.
Moreover, the rings
$K_0(\dot\OS(z,t))$ 
and
$\SYM \otimes_{\ZZ} \SYM$
may be identified 
so that the isomorphism class of the indecomposable indexed by $\LA
\in \Par_{r,s}$ identifies with 
\begin{equation}
\s_\LA := 
\sum_{\substack{0 \leq d \leq \min(r,s)\\\MU \in \Par_{r-d,s-d}}}
\!\!\!N^\LA_\MU \,\s_{\mu^\up} \otimes \s_{\mu^\down}
\quad\text{where}\quad
N^\LA_\MU
:=
(-1)^d \sum_{\nu \vdash d}
LR_{\mu^\up, \nu}^{\la^\up} LR_{\mu^\down, \nu^\trans}^{\la^\down}.
\label{nlamu}\end{equation}
\end{theorem}

The standard technique to prove Theorem~\ref{sscase} 
is to deduce it from Theorem~\ref{webs}
by similar arguments to \cite[Proposition 10.6]{D}; see also \cite[Theorems 4.8.1 and
7.1.1]{CW}.
In other words, one uses Schur-Weyl duality and well-known properties
of  $\Rep U_q(\mathfrak{gl}_n)$ for sufficiently large $n$.
We will take a completely different approach to the proof of
Theorem~\ref{sscase} and the representation theory of $\OS(z,t)$ in
general 
(even in positive characteristic or at roots of unity)
based on the simple observation that it has a
{\em triangular decomposition}. This allows us to adapt the usual arguments
of highest weight theory in a way that is reminiscent of the general framework of \cite{HN, BT}.

In this triangular decomposition, the ``Cartan subalgebra'' $\OS^\circ(z,t)$ 
is the monoidal subcategory consisting of all objects, and morphisms 
spanned by diagrams containing neither caps
nor cups in which all upward propagating strands pass underneath downward
propagating strands. The ``positive Borel subalgebra''
$\OS^\sharp(z,t)$ is defined similarly, allowing also cups but no caps.
Inflation from $\OS^\circ(z,t)$
to $\OS^\sharp(z,t)$ followed by induction from there to $\OS(z,t)$
defines an exact
{\em standardization functor} 
\begin{equation}
\Delta: 
\Mod \OS^\circ(z,t) \rightarrow \Mod \OS(z,t).
\end{equation}
Moreover, there is an obvious equivalence of categories
\begin{equation}\label{morita}
\Mod \OS^\circ(z,t) \approx \prod_{r,s
  \geq 0} \Mod H_r \otimes H_s.
\end{equation}
Since the Hecke algebra is semisimple when $q$ is not a root of unity,
the semisimplicity part of Theorem~\ref{sscase} is a consequence of the following
more general result.

\begin{theorem}\label{selection}
If $t \notin \{\pm q^n\:|\:n \in \ZZ\}$ then
$\Delta$ is an equivalence of
categories.
\end{theorem} 

The other basic observation used to compute $K_0(\dot\OS(z,t))$ as a
ring is:

\begin{theorem}\label{delection}
The inclusion
$\OS^\circ(z,t) \rightarrow \OS(z,t)$ induces a
{ring isomorphism}
$$K_0(\dot\OS{^\circ}(z,t)) \stackrel{\sim}{\rightarrow}
K_0(\dot\OS(z,t)).
$$
\end{theorem}

Now we describe
$K_0(\dot\OS(z,t))$ for all choices of $q$ and $t$.
There are four cases.
\begin{itemize}
\item
Suppose first that $q$ is not a root of unity. 
Up to isomorphism, the irreducible representations of the
(semisimple) Hecke algebra $H_r$ are the {\em Specht modules} parametrized by partitions of $r$.
Using the Morita equivalence (\ref{morita}), we deduce that the irreducible $\OS^\circ(z,t)$-modules
are parametrized by bipartitions; we denote them 
$\{\SS(\LA)\:|\:\LA \in \Par\}.$
Their standardizations give us a family of
$\OS(z,t)$-modules
$\{\Delta(\LA)\:|\:\LA \in \Par\}.$
\begin{itemize}
\item
When $t \notin \{\pm q^n\:|\:n \in \ZZ\}$ (so that $\dot\OS(z,t)$ is semisimple), the modules $\Delta(\LA)$
give
a full set of
pairwise inequivalent indecomposable $\OS(z,t)$-modules.
This is the labelling from Theorem~\ref{sscase}:
in the identification of $K_0(\dot\OS(z,t))$ with
$\SYM\otimes_\ZZ\SYM$ we have that
\begin{equation}\label{k0}
\hspace{15mm}[\Delta(\LA)] \leftrightarrow \chi_\LA.
\end{equation}
\item
When $t = \pm q^n$ for $n \in \ZZ$,
we will show that $\Mod \OS(z,t)$ has the structure of an {\em
  upper-finite highest weight category} with standard modules
$\{\Delta(\LA) \:|\:\LA \in \Par\}$. This is a slight generalization of the usual notion of highest weight
category; e.g., see \cite[$\S$6.1.2]{EL}.
Each standard module $\Delta(\LA)$ has a unique irreducible quotient
$\L(\LA)$,
and these give a full set of pairwise inequivalent irreducible $\OS(z,t)$-modules. Moreover, 
the projective cover $\P(\LA)$ of $\Delta(\LA)$ has a finite
$\Delta$-flag with multiplicities satisfying BGG reciprocity; 
however, unlike for usual highest weight categories, 
standard modules have infinite length.
In this situation, the Grothendieck ring
$K_0(\dot\OS(z,t))$ is identified with the {\em same} ring $\SYM \otimes_{\ZZ} \SYM$
as for generic $t$ so that
\begin{equation}\label{k0a}
\hspace{15mm}
[\P(\MU)] \leftrightarrow \sum_{\substack{0 \leq d \leq \min(r,s)\\\LA \in \Par_{r-d,s-d}}} [\Delta(\LA):\L(\MU)] \chi_\LA
\end{equation}
for $\MU \in \Par_{r,s}$.
It remains to compute the numbers $[\Delta(\LA):\L(\MU)]$.
This turns out to be quite straightforward: they are all either $0$ or
$1$ and can be computed using the cup diagrams of
\cite{BS}. The combinatorics is discussed in detail elsewhere; e.g.,
see \cite{CW, EHS} (with Theorem~\ref{finalt} in mind).
\end{itemize}
\item
Now suppose that $q^2$ is a primitive $e$th root of unity for $e > 1$.
Then the situation is more complicated as the Hecke algebras are
no longer semisimple.
Let $\RPar = \bigsqcup_{r,s \geq 0} \RPar_{r,s}$ be the set of {\em
  $e$-restricted bipartitions}. By \cite{DJ} and (\ref{morita}),
the ``Specht module'' 
$\SS(\LA)$ has irreducible head $\D(\LA)$ if $\LA$ is $e$-restricted, and
the
modules $\{\D(\LA)\:|\:\LA \in \RPar\}$ give a full set of pairwise
inequivalent irreducible
$\OS^\circ(z,t)$-modules.
Also let $\Y(\LA)$ be a projective cover of $\D(\LA)$.
Applying the standardization functor to 
$\D(\LA)$ and $\Y(\LA)$ gives us 
$\OS(z,t)$-modules denoted 
$\bar\Delta(\LA)$ and $\Delta(\LA)$, respectively.
\begin{itemize}
\item
When $t \notin\{\pm q^n\:|\:n \in \ZZ\}$, 
the modules $\left\{\Delta(\LA)\:|\:\LA \in \RPar\right\}$ give a full set of pairwise inequivalent
indecomposable projective $\OS(z,t)$-modules, and $K_0(\dot\OS(z,t))$
is identified with a
proper subring of $\SYM \otimes_\ZZ \SYM$ so that
\begin{equation}\label{k0b}
\hspace{15mm}
[\Delta(\LA)] \leftrightarrow \sum_{\KAPPA \in \Par_{r,s}} [\SS(\KAPPA):\D(\LA)] \chi_\KAPPA
\end{equation}
for $\LA \in \RPar_{r,s}$ and $r,s \geq 0$.
The decomposition numbers
$[\SS(\KAPPA):\D(\LA)]$ 
are known providing $\k$ is of characteristic zero, 
since they are products of the decomposition numbers of Hecke algebras 
determined by Ariki \cite{Ariki}.
\item
When $t = \pm q^n$, the 
 category of $\OS(z,t)$-modules is an {\em upper-finite standardly
   stratified category} with standard modules
$\{\Delta(\LA) \:|\:\LA \in \RPar\}$
and proper standard modules
$\{\bar\Delta(\LA) \:|\:\LA \in \RPar\}$; see \cite[$\S$2]{LW}
and \cite[$\S$6.2.1]{EL}.
The proper standard module $\bar\Delta(\LA)$ has irreducible head
$\L(\LA)$, and these modules give a full set of pairwise inequivalent
irreducible $\OS(z,t)$-modules. 
The projective cover $\P(\MU)$ of $\L(\MU)$ has a finite $\Delta$-flag
with multiplicities satisfying $(\P(\MU):\Delta(\LA)) = [\bar\Delta(\LA):\L(\MU)]$.
Then 
$K_0(\dot\OS(z,t))$
is identified with the {\em same} subring of $\SYM\otimes_\ZZ \SYM$ as
for generic $t$ so that
\begin{equation}\label{k0c}
\hspace{15mm}[\P(\MU)] \leftrightarrow \sum_{\substack{0 \leq d \leq \min(r,s)\\\LA \in
    \RPar_{r-d,s-d}\\\KAPPA \in \Par_{r-d,s-d}}} [\bar\Delta(\LA):\L(\MU)] 
[\SS(\KAPPA):\D(\LA)] \chi_\KAPPA
\end{equation}
for $\MU \in \RPar_{r,s}$ and $r,s\geq 0$.
It means that as well as the decomposition numbers for Hecke
algebras, one also wants to determine the composition multiplicities
$[\bar\Delta(\LA):\L(\MU)]$. This is still an open problem
even when $\k$ is of characteristic zero; we will make some
further comments at the end of the next subsection.
\end{itemize}
\end{itemize}

\vspace{2mm}
\noindent{\bf 1.8.}
Suppose either that $q$ is not a root of unity and $e = 0$, or $q^2$
is a primitive $e$th root of unity
for $e > 1$.
Let $I := \{q^{2n}, t^{-2} q^{-2n}\:|\:n \in \ZZ\} \subset \k$ and
$\mathfrak{g}$ be the (complex) Kac-Moody algebra 
with Cartan matrix $(c_{i,j})_{i,j \in I}$
defined from
\begin{equation}\label{cartan}
c_{i,j} := 
\left\{
\begin{array}{rl}
2&\text{if $i=j$,}\\
-1&\text{if $i = q^2 j$ or $i=q^{-2} j$ but not both,}\\
-2&\text{if $i=q^2j=q^{-2} j$ (which is possible only if $e=2$),}\\
0&\text{otherwise.}
\end{array}\right.
\end{equation}
There are four cases paralleling the discussion of $K_0$ in the previous subsection: 
when $e = 0$
then 
$\g \cong \mathfrak{sl}_\infty\oplus \mathfrak{sl}_\infty$ 
if $t \notin \{\pm q^n\:|\:n \in \ZZ\}$
and $\g\cong\mathfrak{sl}_\infty$ otherwise;
when $e > 0$
then $\g\cong\widehat{\mathfrak{sl}}_e\oplus \widehat{\mathfrak{sl}}_e$ 
if $t \notin \{\pm q^n\:|\:n \in \ZZ\}$
and $\g \cong \widehat{\mathfrak{sl}}_e$ otherwise.
We denote the weight lattice of $\g$ by $P$ and its fundamental dominant weights by $\Lambda_i\:(i
\in I)$.
Let $V(-\Lambda_1|\Lambda_{t^{-2}})$ be the tensor product
of the integrable lowest weight module of lowest weight $-\Lambda_1$
and the integrable highest weight module of highest weight
$\Lambda_{t^{-2}}$.
This is an irreducible $\g$-module if and only if
$t \notin \{\pm q^n\:|\:n \in \ZZ\}$.

\begin{theorem}\label{tpctheorem}
The category of  $\OS(z, t)$-modules admits the structure of a tensor product
  categorification of the $\g$-module
$V(-\Lambda_1|\Lambda_{t^{-2}})$ 
in the general sense of Losev and Webster \cite{LW}.
\end{theorem}

This means in particular that $\dot\OS(z,t)$ is a $2$-representation of
the {\em Kac-Moody 2-category} $\mathfrak{U}(\g)$ of Khovanov, Lauda and
Rouquier \cite{KL3, Rou}: there is a strict $\k$-linear
2-functor from $\mathfrak{U}(\g)$
to the $2$-category of $\k$-linear categories taking
objects to blocks of $\dot\OS(z,t)$, 
1-morphisms to functors between these blocks, and 2-morphisms to
natural transformations between these functors.
The functors $E_i\:(i \in I)$ arise from the summands
of the endofunctor $\up
\otimes ?$ defined by the generalized $i$-eigenspaces of
Jucys-Murphy elements.
For more background on 2-representations, we refer to \cite{BD}, whose notation and
conventions we follow closely. Diagrams representing 2-morphisms in
$\mathfrak{U}(\g)$ will be drawn in red to distinguish them from
diagrams in $\OS(z,t)$.

For any weight $\Lambda \in P$,
there is a universal 2-representation 
$\mathcal R(\Lambda)$ of $\mathfrak{U}(\g)$ with weight subcategories
$\mathcal R(\Lambda)_\omega :=\mathcal{H}om_{\mathfrak{U}(\g)}(\Lambda,
\omega)$
for each $\omega \in P$; 
see \cite[$\S$5.1.2]{Rou} and also \cite[$\S$4.2]{BD}.
Now we set
$\Lambda := \Lambda_{t^{-2}}-\Lambda_1$
and
let $\mathcal I$ be the invariant ideal (``full sub-2-representation'') of 
$\mathcal{R}(\La)$ generated by the 2-morphisms
\begin{equation}\label{ideal}
\mathord{
\begin{tikzpicture}[baseline = 0]
	\draw[->,darkred,thick] (0.08,-.3) to (0.08,.4);
      \node at (0.08,0) {$\color{darkred}\bullet$};
   \node at (0.1,-.42) {$\scriptstyle{i}$};
   \node at (-0.3,0) {$\color{darkred}\scriptstyle{\delta_{i,1}}$};
\end{tikzpicture}
}
{\scriptstyle\La}\:,
\qquad
\mathord{
\begin{tikzpicture}[baseline = 2]
	\draw[<-,darkred,thick] (0.08,-.3) to (0.08,.4);
      \node at (0.08,0.1) {$\color{darkred}\bullet$};
   \node at (0.1,.54) {$\scriptstyle{i}$};
   \node at (-0.38,.1) {$\color{darkred}\scriptstyle{\delta_{i,t^{-2}}}$};
\end{tikzpicture}
}
{\scriptstyle\La}\:,
\qquad
\mathord{
\begin{tikzpicture}[baseline = 1mm]
  \draw[<-,thick,darkred] (0,0.4) to[out=180,in=90] (-.2,0.2);
  \draw[-,thick,darkred] (0.2,0.2) to[out=90,in=0] (0,.4);
 \draw[-,thick,darkred] (-.2,0.2) to[out=-90,in=180] (0,0);
  \draw[-,thick,darkred] (0,0) to[out=0,in=-90] (0.2,0.2);
 \node at (0,-.12) {$\scriptstyle{1}$};
   \node at (0.45,0.2) {$\scriptstyle{\La}$};
\end{tikzpicture}
}
\end{equation}
for all $i \in  I$ (the last generator is needed only in case $t=\pm 1$).
The quotient $2$-representation
\begin{equation}\label{similarly}
\VV(-\Lambda_1|\Lambda_{t^{-2}})
:=
\mathcal{R}(\La) / \mathcal I
\end{equation}
is a special one of Webster's {\em generalized cyclotomic
  quotients} of $\mathfrak{U}(\g)$ introduced in \cite[Proposition
5.6]{Web}; see also \cite[Construction 4.13]{BD} where 
it is denoted
${\mathcal L}_{\min}(-\Lambda_1|\Lambda_{t^{-2}})$.
It is a $\k$-linear category which is not monoidal in any obvious way.

\begin{theorem}\label{evalthm}
Evaluation on the unit  object
defines a full
strongly equivariant functor (``morphism of $2$-representations'')
$\Theta:\mathcal{R}(\Lambda_{t^{-2}}-\Lambda_1)
\rightarrow \dot\OS(z,t)$.
This factors through $\VV(-\Lambda_1|\Lambda_{t^{-2}})$ to induce a strongly equivariant
equivalence
$$\bar\Theta:
\dot{\VV}(-\Lambda_1|\Lambda_{t^{-2}})
\stackrel{\approx}{\longrightarrow} \dot\OS(z,t).
$$
\end{theorem}

This is significant because the finite-dimensional category
$\VV(-\Lambda_1|\Lambda_{t^{-2}})$ possesses a natural
$\ZZ$-grading.
When the ground field is of characteristic zero, 
this 
grading is
known to be {\em mixed}
in the sense of \cite[Definition
1.11]{Web}
in three of the four cases discussed above: it is trivial in the semisimple case;
 it may be deduced in an elementary way from the
Koszulity of the Khovanov arc algebra $K^\infty_\infty$ studied in
\cite{BS} 
when $e=0$ and $t \in \{\pm q^n\:|\:n \in \ZZ\}$;
and it follows from \cite{VV} 
when $e > 0$ and $t \notin \{\pm q^n\:|\:n \in \ZZ\}$.
We conjecture that it is also mixed in the fourth case.
As discussed in \cite[$\S$8]{Web}, the truth of this conjecture
implies
that the classes $[\P(\LA)]$ coincide with Lusztig's canonical basis for
$V(-\Lambda_1|\Lambda_{t^{-2}})$.

\vspace{2mm}
\noindent{\bf 1.9.}
There is a parallel story in the {\em degenerate case} $z=0$.
In this case, relation ($\mathrm{S}$) says simply that the
positive and negative crossings are {\em equal}; it is natural
to denote them both by the same ``singular'' crossing
$\mathord{
\begin{tikzpicture}[baseline = -.5mm]
	\draw[->,thick,darkblue] (0.2,-.2) to (-0.2,.3);
	\draw[thick,darkblue,->] (-0.2,-.2) to (0.2,.3);
\end{tikzpicture}
}$. The relation (T) forces $t^2=1$;
we assume actually that $t = 1$ since the other possibility $t=-1$ produces an
isomorphic object.
In place of the relation ($\mathrm{D}$) (which no longer makes any sense)
we impose that
$$
\Bubble
=
\delta\, 1_\varnothing
$$
for some $\delta \in \k$.
The resulting category is the {\em oriented
  Brauer category} $\OB(\delta)$
from \cite{BCNR}, which is the free $\k$-linear symmetric monoidal
category generated by the dual pair of objects $\up$ and $\down$ of
dimension $\delta$.
Like in (\ref{randwick}), 
there is a homomorphism
\begin{equation}\label{randwick2}
\i_r:\k\Sym_r \rightarrow \End_{\OS(\delta)}(\up^r)
\end{equation}
sending the transposition $(i\:\:i\!+\!1)$
to the crossing of the $i$th and $(i+1)$th strands. 
The degenerate analogs of Theorems~\ref{first} and \ref{bt} are
discussed in \cite{BCNR}; the latter shows in particular that $\i_r$ is an isomorphism.

In this paragraph suppose that $\k = \CC$. The category
\begin{equation}
\REP GL_\delta := \dot\OB(\delta)
\end{equation}
is the Deligne category mentioned before.
As in Theorem~\ref{webs}, for $n \in \NN$ and 
any sign $\eps \in \{\pm\}$, the category $\Rep
GL_{n}$ of (finite-dimensional) 
rational representations of 
$GL_{n}$ over $\k$ is monoidally equivalent to the quotient of 
the Deligne category $\REP GL_{\eps n}$ by the tensor ideal of negligible morphisms. This is proved in
\cite[Th\'eor\`eme 10.4]{D}; the induced symmetric monoidal structure
on $\Rep GL_{n}$ is the usual one when $\eps = +$, and comes
from super vector spaces when $\eps = -$.
The following extends this result to include fields of positive
characteristic.

\begin{theorem}\label{websup}
For
$n \in \NN$ and $\eps \in \{\pm\}$, 
there is a full $\k$-linear monoidal functor
$\Psi:\OB(\eps n) \rightarrow 
\Rep GL_{n}$ 
sending $\up$ and $\down$ to the natural $GL_{n}$-module $V$
and its dual $V^*$, respectively,
the crossing $\mathord{
\begin{tikzpicture}[baseline = -.5mm]
	\draw[->,thick,darkblue] (0.2,-.2) to (-0.2,.3);
	\draw[thick,darkblue,->] (-0.2,-.2) to (0.2,.3);
\end{tikzpicture}
}$
to the homomorphism $V \otimes V
\rightarrow V \otimes V, v \otimes w \mapsto \eps w \otimes v$,
and the cap
$\mathord{
\begin{tikzpicture}[baseline = -3mm]
	\draw[<-,thick,darkblue] (0.3,-0.35) to[out=90, in=0] (0.05,0);
	\draw[-,thick,darkblue] (0.05,0) to[out = 180, in = 90] (-0.2,-0.35);
\end{tikzpicture}
}
$ to 
$V \otimes V^* \rightarrow \k, v \otimes f \mapsto \eps f(v)$.
It induces a monoidal equivalence 
\begin{equation}\label{tiltingequiv}
\bar\Psi:\dot{\OB}(\eps n) / \mathcal N
\stackrel{\approx}{\longrightarrow}
\Tilt' GL_{n},
\end{equation}
where
$\mathcal N$ is the 
additive $\k$-linear tensor ideal of $\dot{\OB}(\eps n)$
generated 
by
$$
x := \left\{
\begin{array}{ll}
\sum_{g \in \Sym_{n+1}} \operatorname{sgn}(g) \i_{n+1}(g)
&\text{if $\eps = +$},\\
\sum_{g \in \Sym_{n+1}} \i_{n+1}(g)&\text{if $\eps = -$,}
\end{array}\right.
$$
and
$\Tilt' GL_n$ is the full subcategory of $\Rep GL_n$ consisting of
modules that are isomorphic to direct sums of summands of tensor
products of $V$ and $V^*$.
\end{theorem}

The other results discussed above can also be adapted quite easily to $\OB(\delta)$.
For example, 
the degenerate analog of Theorem~\ref{sscase} 
gives that $\dot\OB(\delta)$ is semisimple if and only if $\k$ is of
characteristic zero and $\delta \notin \ZZ$. This is proved in
\cite{D}; non-semisimplicity in positive characteristic is clear from
(\ref{randwick}).
The analog of Theorem~\ref{selection} needs $\delta \notin \ZZ\cdot 1_\k$.
Then there is a description of $K_0(\dot\OB(\delta))$ with four cases similar to the
above: replace the representation theory of Hecke algebras with that
of symmetric groups.

Let us discuss the degenerate analogs of Theorem~\ref{tpctheorem}--\ref{evalthm} in a little
more detail. Assume now that $\k$ is a field of characteristic $p \geq 0$.
Let $I := \{n, -n-\delta\:|\:n \in \ZZ\}
\subseteq \k$, 
and $\mathfrak{g}$ be the (complex) Kac-Moody algebra 
with Cartan matrix $(c_{i,j})_{i,j \in I}$
defined from
\begin{equation}\label{cartan2}
c_{i,j} := 
\left\{
\begin{array}{rl}
2&\text{if $i=j$,}\\
-1&\text{if $i = j+1$ or $i=j-1$ but not both,}\\
-2&\text{if $i=j+1=j-1$ (which is possible only if $p=2$),}\\
0&\text{otherwise.}
\end{array}\right.
\end{equation}
As before, there are four possibilities depending on
$p$ and $\delta$: $\g \cong \mathfrak{sl}_\infty\oplus
\mathfrak{sl}_\infty,
\mathfrak{sl}_\infty,
\widehat{\mathfrak{sl}}_p\oplus\widehat{\mathfrak{sl}}_p$ or $\widehat{\mathfrak{sl}}_p$.
The degenerate analog of Theorem~\ref{tpctheorem} shows that
$\dot\OB(\delta)$ admits the structure of a tensor product
categorification of the $\g$-module $V(-\Lambda_0|\Lambda_{-\delta})$;
see also \cite[Theorem 10.2.1]{A} for a closely related (actually,
Ringel dual) statement in the case $p=0$.
The following is the degenerate analog of Theorem~\ref{evalthm};
it was conjectured originally in discussions
with Stroppel and Webster.

\begin{theorem}\label{aisha2}
Evaluation on the unit induces
a strongly equivariant equivalence
\begin{equation}\label{thappy}
\bar\Theta:\dot\VV(-\Lambda_0|\Lambda_{-\delta})
\stackrel{\approx}{\longrightarrow}
\dot\OB(\delta)
\end{equation}
where
$\VV(-\Lambda_0|\Lambda_{-\delta})$
is the quotient of the universal 2-representation
$\mathcal R(\Lambda_{-\delta}-\Lambda_0)$
of $\mathfrak{U}(\g)$ by the invariant ideal generated by the
2-morphisms
\begin{equation}
\mathord{
\begin{tikzpicture}[baseline = 0]
	\draw[->,darkred,thick] (0.08,-.3) to (0.08,.4);
      \node at (0.08,0) {$\color{darkred}\bullet$};
   \node at (0.1,-.42) {$\scriptstyle{i}$};
   \node at (-0.3,0) {$\color{darkred}\scriptstyle{\delta_{i,0}}$};
\end{tikzpicture}
}
{\scriptstyle\Lambda_{-\delta}-\Lambda_0}\:,
\qquad
\mathord{
\begin{tikzpicture}[baseline = 2]
	\draw[<-,darkred,thick] (0.08,-.3) to (0.08,.4);
      \node at (0.08,0.1) {$\color{darkred}\bullet$};
   \node at (0.1,.54) {$\scriptstyle{i}$};
   \node at (-0.38,.1) {$\color{darkred}\scriptstyle{\delta_{i,-\delta}}$};
\end{tikzpicture}
}
{\scriptstyle\Lambda_{-\delta}-\Lambda_0}\:,
\qquad
\mathord{
\begin{tikzpicture}[baseline = 1mm]
  \draw[<-,thick,darkred] (0,0.4) to[out=180,in=90] (-.2,0.2);
  \draw[-,thick,darkred] (0.2,0.2) to[out=90,in=0] (0,.4);
 \draw[-,thick,darkred] (-.2,0.2) to[out=-90,in=180] (0,0);
  \draw[-,thick,darkred] (0,0) to[out=0,in=-90] (0.2,0.2);
 \node at (0,-.12) {$\scriptstyle{0}$};
   \node at (0.85,0.2) {$\scriptstyle{\Lambda_{-\delta}-\Lambda_0}$};
\end{tikzpicture}
}
\end{equation}
for all $i \in  I$ (the last generator is needed only in case
$\delta=0$).
\end{theorem}

\begin{corollary}\label{finalt}
Assume that $\k = \CC$, $q$ is not a root of unity, and $\delta \in
\CC$ is arbitrary.
Then there is a $\CC$-linear equivalence of categories
$\REP U_q(\mathfrak{gl}_\delta)\stackrel{\approx}{\longrightarrow} \REP GL_\delta$.
\end{corollary}

\begin{proof}
When $e=p=0$, we have that $\mathfrak{g} \cong \mathfrak{sl}_\infty \oplus
\mathfrak{sl}_\infty$ if $\delta \notin \ZZ$ or $\mathfrak{sl}_\infty$
if $\delta \in \ZZ$.
Now observe that the category $\VV(-\Lambda_0|\Lambda_{-\delta})$
in Theorem~\ref{aisha2} is isomorphic to the category $\VV(-\Lambda_1|\Lambda_{t^{-2}})$
in Theorem~\ref{evalthm} by a relabelling of the Dynkin diagram.
\end{proof}

The equivalence constructed in Corollary~\ref{finalt} is {\em not} monoidal, but
it is
a strongly equivariant equivalence of 2-representations. Hence,
it is compatible with the endofunctors $\up \otimes -$ 
and $\down \otimes -$, and it induces a ring isomorphism between the
Grothendieck rings preserving the labellings of isomorphism 
classes of indecomposable objects.
Etingof has suggested that such an equivalence could also be
constructed using KZ equations in the spirit of the Drinfeld-Kohno
theorem.

For our final corollary, we assume $\k$ is a field of positive
characteristic $p$ and take $\delta$ to be the image in $\k$ of some
$n\in \NN$, 
so that  $\mathfrak{g} \cong \widehat{\mathfrak{sl}}_p$.
There is a well-known categorical action making the category $\Rep GL_n$ of rational
representations of $GL_n$ over $\k$ into a 2-representation of
$\mathfrak{U}(\g)$;
see \cite[$\S$7.5.1]{CR} and \cite[$\S$6]{RW}.
The full subcategory $\Tilt GL_n$ of $\Rep GL_n$ consisting of all 
tilting modules (e.g., see \cite{Donkin})
is a Karoubian sub-2-representation.
As explained in detail in \cite[Proposition 6.5]{RW},
there is a $\g$-module isomorphism
\begin{equation}\label{thippy}
\CC \otimes_{\ZZ} K_0(\Tilt GL_n)
\cong {\bigwedge}^n \operatorname{Nat}_p,
\end{equation}
where $\operatorname{Nat}_p$ is the level zero
representation of $\mathfrak{g}$ with basis $\{m_r\:|\:r \in \ZZ\}$
on which the Chevalley generators of $\g$ act via
\begin{align*}
e_i m_r &= \left\{\begin{array}{ll}
m_{r+1}&\text{if $i \equiv r \pmod{p}$,}\\
0&\text{otherwise;}
\end{array}\right.
&
f_i m_{r+1} &= \left\{\begin{array}{ll}
m_{r}&\text{if $i \equiv r \pmod{p}$,}\\
0&\text{otherwise.}
\end{array}\right.
\end{align*}
Using the defining relations for the cyclic module
$V(-\Lambda_0|\Lambda_{-n})$ (e.g., see
\cite[(3.6)]{BD}), it is easy to check that there is a 
$\g$-module homomorphism
$$
V(-\Lambda_0|\Lambda_{-n}) \rightarrow
{\bigwedge}^n \operatorname{Nat}_p
$$
sending the generator of $V(-\Lambda_0|\Lambda_{-n})$ ($=$ the class of the
unit object under (\ref{thappy}))
to $m_0\wedge m_{-1}\wedge\cdots\wedge m_{1-n}$
($=$ the class of the trivial module under
(\ref{thippy})).
This map is surjective, i.e.,
$m_0\wedge m_{-1}\wedge\cdots\wedge m_{1-n}$
generates $\bigwedge^n \operatorname{Nat}_p$ as a $\g$-module, 
if and only if $p > n$.
This is also exactly the requirement on $p$ needed to ensure that 
the subcategory
$\Tilt' GL_n$ from Theorem~\ref{websup} is all of $\Tilt GL_n$.
Indeed, when $p > n$
all of the exterior powers $V, \bigwedge^2 V, \dots, \det = \bigwedge^n V$ plus
$\det^{-1} = \bigwedge^n V^*$ are summands of corresponding
tensor powers of $V$ or $V^*$ so they 
lie in $\Tilt' GL_n$.
Every indecomposable 
tilting module arises as a summand of some tensor product of these fundamental
tilting modules by highest weight considerations.

\begin{corollary}\label{lastcor}
Suppose that $n \in \NN$ and
$\k$ is a field of characteristic $p > n$.
There is a strongly equivariant equivalence
$$
\Phi:\dot\VV(-\La_0|\La_{- n}) / \mathcal J
\stackrel{\approx}{\longrightarrow}
\Tilt GL_n
$$
where
$\mathcal J$ is the invariant ideal of 
$\dot\VV(-\La_0|\La_{- n})$
generated by
$$
y := 
\mathord{
\begin{tikzpicture}[baseline = 0]
	\draw[->,darkred,thick] (0.48,-.2) to (0.48,.3);
   \node at (0.5,-.32) {$\scriptstyle{0}$};
	\draw[->,darkred,thick] (-0.02,-.2) to (-0.02,.3);
   \node at (-0.05,-.32) {$\scriptstyle{-1}$};
   \node at (-0.5,.05) {$\scriptstyle{\cdots}$};
	\draw[->,darkred,thick] (-1.02,-.2) to (-1.02,.3);
	\draw[->,darkred,thick] (-1.52,-.2) to (-1.52,.3);
   \node at (-1,-.32) {$\scriptstyle{1-n}$};
   \node at (-1.6,-.32) {$\scriptstyle{-n}$};
   \node at (-1.52,.05) {$\color{darkred}\scriptstyle{\bullet}$};
   \node at (-1.98,.05) {$\color{darkred}\scriptstyle{\delta_{p,n+1}}$};
\end{tikzpicture}
}
{\scriptstyle\Lambda_{-n}-\Lambda_0}.
$$
\end{corollary}

\begin{proof}
Let $\Psi$ be the functor from Theorem~\ref{websup} taking $\eps = +$.
By the definitions of the categorical actions, it is strongly equivariant.
The tensor ideal $\mathcal N$ in Theorem~\ref{websup}
is generated by the quasi-idempotent $x= 
\sum_{g \in \Sym_{n+1}} \operatorname{sgn}(g) \i_{n+1}(g)$.
Since $\dot{\OB}(n)$ is symmetric monoidal, it is actually
generated by $x$ just as a {\em left} tensor ideal.

Let $\bar\Theta^{-1}$ be quasi-inverse to the strongly equivariant equivalence $\bar\Theta$ from Theorem~\ref{aisha2}.
We claim that it maps the generator $x$ of 
$\mathcal N$ (as a left tensor ideal) 
to a non-zero multiple of the generator $y$ of $\mathcal J$ (as an
invariant ideal).
To prove this, the definition of $\bar\Theta$ from the proof of
Theorem~\ref{aisha2} means that on endomorphisms of $\up^{n+1}$ 
the map induced by
$\bar\Theta^{-1}$ arises
from the isomorphism between the group algebra $\k \Sym_{n+1}$
and the corresponding cyclotomic quiver Hecke algebra constructed in \cite{BK}.
So the claim follows from \cite[Proposition 6.7]{HM} applied to the standard tableau $\mathfrak{s}$
that is a single row of length $(n+1)$. (This argument works when $p
\leq n$ too producing a slightly more complicated formula for $y$.)

From the claim and Theorem~\ref{websup},
it follows that $\bar\Theta$ induces a strongly equivariant equivalence
$
\dot{\VV}(-\Lambda_0|\Lambda_{-n}) / \mathcal J
\stackrel{\approx}{\longrightarrow}
\dot{\OB}(n) / \mathcal N.
$
To get $\Phi$, it just remains to
compose this with the strongly equivariant equivalence $\bar \Psi$
from Theorem~\ref{websup}, noting that $\Tilt' GL_n = \Tilt GL_n$ 
due to the assumption $p > n$.
\end{proof}

The category $\dot\VV(-\La_0|\La_{-n}) / \mathcal J$ appearing in
Corollary~\ref{lastcor}
has a natural $\ZZ$-grading.
So we have constructed a graded lift of 
$\Tilt GL_n$. In \cite{RW}, 
Riche and Williamson have constructed graded lifts of all regular
blocks of $\Tilt GL_n$ via the diagrammatic Hecke category of
\cite{EW}; see also \cite{EL}.
We expect that the graded lifts of such blocks arising 
from Corollary~\ref{lastcor} are equivalent to the
ones of {\em loc. cit.}. 

We leave it to the reader to formulate the $q$-analog of Corollary~\ref{lastcor};
see Remark~\ref{websrem}.

\vspace{2mm}
\noindent{\bf 1.10.}
The remainder of the article is organized as follows.
Sections 2 and 3 are expository in nature and contain proofs of
Theorems~\ref{first},\ref{bt} and \ref{webs}, thereby making
the connection to $U_q(\mathfrak{gl}_n)$.
For Theorem~\ref{TR}, which is not needed in the remainder of the
article,
we refer the reader to
Turaev's original article \cite{Turaev0}.
Then Section~\ref{aos} discusses the affine oriented skein category
$\AOS(z,t)$ and the resulting Jucys-Murphy elements.
The triangular decomposition of $\OS(z,t)$ is introduced in
section~\ref{sswt}, and the highest weight approach to representations is 
developed there. Another noteworthy result in this section is
Theorem~\ref{koike1},
which is used both to prove Theorem~\ref{delection} and to 
obtain the description of $K_0(\dot\OS(z,t))$ (Theorem~\ref{thegg}).
Then in section~\ref{schars} we study certain induction and restriction
functors $E_i$ and $F_i$ which give rise to the categorical action on
$\OS(z,t)$-modules.
A novel result here is Theorem~\ref{characters}, which uses these induction and
restriction functors to compute the formal characters of the
standardizations of Specht modules. 
This is used to prove
Theorem \ref{selection}, also completing the proof of Theorem~\ref{sscase}.
In section~\ref{tpc}, 
Theorem~\ref{characters} is used again to prove 
a
linkage principle for the decomposition numbers
$[\bar\Delta(\LA):\L(\MU)]$ (Theorem~\ref{linkage1}).
In Theorem~\ref{sstrat}, we introduce the highest weight/standardly stratified structure on
$\OS(z,t)$-modules. Then we prove Theorems~\ref{tpctheorem}--\ref{evalthm}.
Finally, in section~\ref{sdeg}, we discuss the degenerate case.
Theorem~\ref{websup} is proved by the same argument as Theorem~\ref{webs}.
After that, we just highlighting the main differences in the
degenerate case compared to the quantum case, which
arise because the Jucys-Murphy elements need different
treatment. 
Theorem~\ref{aisha2} then follows.

\vspace{2mm}
\noindent{\em Acknowledgements.}
This article is based in part on the PhD thesis of Andrew Reynolds \cite{R}, who developed the
representation theory of the oriented Brauer category using the highest
weight approach. In particular, Reynolds proved the degenerate analog of Theorem~\ref{tpctheorem}.
Many of the arguments were inspired by 
Ben Webster's work \cite{Wbook} and the ideas of \cite{LW}.
I also thank Stephen Donkin, Stephen Doty, Michael Ehrig,
Inna Entova-Aizenbud,
Pavel Etingof, Catharina Stroppel and Geordie Williamson for helpful
questions and answers.

\section{Generators and relations}\label{sgr}

The monoidal category $\FOT$ from the introduction is the category $\mathcal{R}ib_{\mathcal
  V}$ from \cite[$\S$I.2.3]{Turaev1}, taking $\mathcal V$ to be the trivial monoidal category with
just one object $*$ and one morphism $1_*:* \rightarrow *$.
Our generating objects $\up$ and $\down$ are Turaev's $(*,-)$ and
$(*,+)$. As we explained in the introduction, objects in $\FOT$ are
words $\a, \b, \dots$ in the letters $\up, \down$, i.e., elements of the
free monoid $\words$ generated by these symbols. Morphisms
$\a \rightarrow \b$ are isotopy classes of $(\a,\b)$-ribbons.

We say that an $(\a,\b)$-ribbon is {\em generic} if all of its critical
points ($=$ points of slope zero) are local maxima and minima, and all
crossings occur away from the critical points. 
Thus, a generic $(\a,\b)$-ribbon can only involve ``identity lines'' of non-zero
slope, two sorts of cup (left/right), 
two sorts of cap (left/right), and eight sorts of crossing
(up/right/down/left and positive/negative).
Any $(\a,\b)$-ribbon is isotopic to a generic one. 
Moreover, isotopy of generic $(\a,\b)$-ribbons is generated by
{\em rectilinear isotopy}, i.e., planar isotopy that fixes the
boundary and preserves genericity,
plus the oriented Reidemeister moves $(\text{R}0)$, 
$(\text{FR}\text{I})$, 
$(\text{R}\text{II})$ and $(\text{R}\text{III})$ from
Figure 1.
This is justified carefully in \cite[$\S$I.4.6]{Turaev1}.

The following theorem giving an explicit monoidal presentation for
$\FOT$ follows from this discussion; see also
\cite[Theorem 3.2]{Turaev3} for the analogous result without framing,
and
\cite[$\S$I.4.2]{Turaev1} for more background about generators and relations for strict monoidal
categories.

\begin{lemma}\label{title}
The category $\FOT$ is the free strict monoidal category generated
by the objects $\up$ and $\down$ and the morphisms
$\mathord{

}$,
 subject only to the relations $(\text{\em R}0)$,
$(\text{\em FRI})$,
$(\text{\em RII})$ and $(\text{\em RIII})$.
\end{lemma}

There are many redundancies in the presentation for $\FOT$ just given. 
In fact, the morphisms 
$\mathord{
%
}$ suffice to generate all other morphisms, since 
the other crossings can be expressed in terms of them:
\begin{align}
\mathord{
%
}\,.
\end{align}
Alternatively,
as observed in \cite[Lemma I.3.1.1]{Turaev1}, the morphisms
$\mathord{
%
}
\end{equation}
give a system of generators. The leftward cap and leftward cup can then be
recovered:
\begin{align}\label{ready}
\mathord{
%
}\,.
\end{align}
The following theorem due to Turaev gives a more efficient set of relations
between this set of generators.

\begin{theorem}\label{key}
The category $\FOT$ is generated by the objects $\up, \down$
and the morphisms 
$\mathord{
%
}$
subject only to the following relations:
\begin{itemize}
\item[(i)]
$\mathord{
%
}
$. 
\end{itemize}
\end{theorem}

\begin{proof}
This is \cite[Lemma I.3.3]{Turaev1} except that we have rotated Turaev's
generators and relations by $180^\circ$ around a horizontal axis.
\end{proof}

The category $\FOT$ is a ribbon category in the sense of
\cite[$\S$3.3.2]{Turaev2}: it is braided and pivotal,
and $(\text{FR}\text{I})$ ensures that the right and left twists are equal.
Like in
\cite[$\S$XII.2.2]{Turaev1}, the braiding $\tau_{\a,\b}:\a \otimes \b \rightarrow \b
\otimes \a$ is defined by the first of the following diagrams;
the right dual of $\a = \a_n \cdots \a_1\in\words$ is $\a^* := \a_1^* \cdots \a_n^*$
where $\up^* = \down$ and $\down^* = \up$ with structure maps
$\a\otimes \a^* \rightarrow \unit \rightarrow \a^* \otimes \a$ defined
from the second of the following diagrams;
the left dual ${^*}\a$ is the same object as the right dual
with structure maps 
${^*\!}\a \otimes \a  \rightarrow \unit \rightarrow \a \otimes {^*\!}\a$
defined by the third diagram.
$$
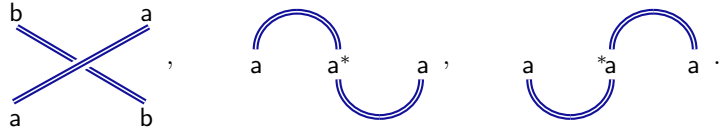
\mathord{
\begin{tikzpicture}[baseline = 0]
\node at (0,.7) {$\b\:\qquad\qquad \: \a$};
\draw[-,thick,double,darkblue] (-.85,.5) to (.85,-.5);
\draw[-,line width=5pt,white] (-.9,-.5) to (.9,.5);
\draw[-,thick,double,darkblue] (-.9,-.5) to (.9,.5);
\node at (0,-.7) {$\a\:\qquad  \qquad \:\b$};
\end{tikzpicture}}\,,\qquad
\mathord{
\begin{tikzpicture}[baseline = 0]
\node at (0,0) {$\,\:\a\:\:\qquad  \a^*\:\qquad \:\a$};
\draw[-,thick,double,darkblue] (-1.05,.2) to [out=90,in=180] (-.5,.7);
\draw[-,thick,double,darkblue] (-.5,.7) to [out=0,in=90] (.05,.2);
\draw[-,thick,double,darkblue] (.6,-.7) to [out=180,in=-90] (.05,-.2);
\draw[-,thick,double,darkblue] (1.15,-.2) to [out=-90,in=0] (0.6,-.7);
\end{tikzpicture}
}\,,
\qquad
\mathord{
\begin{tikzpicture}[baseline = 0]
\node at (0,0) {$\:\a\:\:\:\qquad \!\! {^*\!}\a\:\:\:\qquad \a$};
\draw[-,thick,double,darkblue] (-1.05,-.2) to [out=-90,in=180] (-.5,-.7);
\draw[-,thick,double,darkblue] (-.5,-.7) to [out=0,in=-90] (.05,-.2);
\draw[-,thick,double,darkblue] (1.15,.2) to [out=90,in=0] (.6,.7);
\draw[-,thick,double,darkblue] (.6,.7) to [out=180,in=90] (.05,.2);
\end{tikzpicture}
}\,.
$$
(In these diagrams, the double lines labelled $\a$ denote parallel thin lines oriented
in order from left to right according to the letters of the word $\a$.)
The right and left duality functors are both defined on diagrams by rotating
the $xy$-plane through $180^\circ$, hence, they are equal, and we have
equipped $\FOT$ with a strictly pivotal structure.

The ribbon structure on $\FOT$ induces a ribbon structure on the
oriented skein category $\OS(z,t)$ too.
In particular, it also possesses a strictly pivotal structure.

\begin{proof}[Proof of Theorem~\ref{first}]
Let $\mathcal C$ be the strict monoidal category defined by the
generators and relations (1)--(5) from Theorem~\ref{first}.
We first define a strict monoidal functor $\Phi:{\mathcal C}
\rightarrow \OS(z,t)$
sending $E \rightarrow \up, F \rightarrow \down$, and
$S, T, C$ and $D$ to 
$\mathord{
\begin{tikzpicture}[baseline = -.5mm]
	\draw[->,thick,darkblue] (0.28,-.3) to (-0.28,.4);
	\draw[line width=4pt,white,-] (-0.28,-.3) to (0.28,.4);
	\draw[thick,darkblue,->] (-0.28,-.3) to (0.28,.4);
\end{tikzpicture}
}$, 
$\mathord{
\begin{tikzpicture}[baseline = -.5mm]
	\draw[->,thick,darkblue] (0.28,-.3) to (-0.28,.4);
	\draw[line width=4pt,white,-] (-0.28,-.3) to (0.28,.4);
	\draw[thick,darkblue,<-] (-0.28,-.3) to (0.28,.4);
\end{tikzpicture}
}$,
$\mathord{
\begin{tikzpicture}[baseline = 1mm]
	\draw[<-,thick,darkblue] (0.4,0.4) to[out=-90, in=0] (0.1,0);
	\draw[-,thick,darkblue] (0.1,0) to[out = 180, in = -90] (-0.2,0.4);
\end{tikzpicture}
}$ and $\mathord{
\begin{tikzpicture}[baseline = 1mm]
	\draw[<-,thick,darkblue] (0.4,0) to[out=90, in=0] (0.1,0.4);
	\draw[-,thick,darkblue] (0.1,0.4) to[out = 180, in = 90] (-0.2,0);
\end{tikzpicture}
}\,$.
To check this is well defined, one needs to verify that the
relations from Theorem~\ref{first}(1)--(5) all hold in
$\OS(z,t)$.
We already set this as an exercise for the reader in the introduction,
and are not about to spoil the fun here!

Next we construct a strict monoidal functor
$\Psi:\OS(z,t) \rightarrow \mathcal C$ in the other direction. 
For this we use the presentation for the strict
$\k$-linear monoidal category
$\OS(z,t)$ arising from Theorem~\ref{key}, with
eight generating morphisms and relations
(i)--(viii) from the theorem, plus the relations
($\text{S}$), ($\text{T}$) and ($\text{D}$).
Then (\ref{ready}) becomes the {\em definition} of the leftward cap and
leftward cup. 
We set
$C' := t  T \circ C$ and $D' := t D \circ T$, then
define $\Psi$
by sending
$\up \mapsto E, \down \mapsto F$ and
the eight generating morphisms in the order listed to 
$D, C, S, S - z 1_E \otimes 1_E, T, T + z C' \circ D'$, $t 1_E$ and $t^{-1}
1_E$, respectively.
To verify that this is well defined, we must check the images of the eleven relations
hold in $\mathcal C$. The relations (i) and ($\text{S}$)
follow from (1). Relations (ii), (iii), (iv), (vi) and (vii) are
equally easy using (2), (3) and the definitions.
For relation (D), $\Psi$ maps $\Bubble$ to $D \circ C' = D' \circ C = t D \circ T
\circ C = \frac{t-t^{-1}}{z} 1_\unit$ by (5). 
Consider relation (v). 
From (4), we see that the image of the negative rightward
crossing
is
$(T^{-1} - z C \circ D)$,
and we must check that this has two-sided inverse $(T + z C'
\circ D')$.
This follows easily using the identities established so far
plus
$T^{-1} \circ C' = t C, D' \circ T^{-1} = t D$.
For (viii), 
we note that $(1_F \otimes S) \circ (C \otimes 1_E) = (T^{-1}
\otimes 1_E) \circ (1_E \otimes C)$.
So
$\Psi$ maps the $\&$ symbol to 
\begin{align*}
((D \circ T& + z D \circ C' \circ D') \otimes 1_E) \circ 
(T^{-1}
\otimes 1_E) \circ (1_E \otimes C)\\
&=
((t^{-1} D' + (t-t^{-1}) D') \otimes 1_E) \circ 
(T^{-1}
\otimes 1_E) \circ (1_E \otimes C)\\
&= t((D' \circ T^{-1}) \otimes 1_E) \circ (1_E \otimes C)
=t^2(D \otimes 1_E) \circ (1_E \otimes C) = t^2 1_E,
\end{align*}
as required.
Finally, for (T), the image of the positive right curl is $$
(D \otimes
1_E) \circ (1_E \otimes T^{-1}) \circ (1_E \otimes C')
= t (D \otimes 1_E) \circ (1_E \otimes C) = t 1_E.
$$

To complete the proof of Theorem~\ref{first}, it remains to observe that the functors $\Phi$ and $\Psi$ are two-sided
inverses. This depends on (\ref{diamond})--(\ref{ready}).
\end{proof}

To conclude the section, we briefly list some further symmetries of the monoidal categories
$\OS(z,t)$.
There is an isomorphism 
\begin{equation}\label{tau}
\tau:\OS(z,t) \stackrel{\sim}{\rightarrow}
\OS(z,t)^{\operatorname{op}}
\end{equation}
which fixes objects,
and rotates diagrams for morphisms though $180^\circ$ around a horizontal axis then
reverses all orientations. Thus, the vertical crossings are
fixed and leftward crossings are switched with rightward crossings (preserving whether
they are positive or negative), while rightward and leftward caps are switched with
leftward and rightward cups, respectively.
Composing $\tau$ with duality, we obtain an
  isomorphism 
\begin{equation}
\phi:\OS(z,t) \stackrel{\sim}{\rightarrow}
  \OS(z,t)^{\operatorname{rev}}.
\end{equation}
This fixes sideways crossings, cups and caps, but switches upward
crossings with downward crossings (preserving whether they are positive or
negative).
Finally, there are isomorphisms
\begin{align}
\rho&:\OS(z,t) \stackrel{\sim}{\rightarrow} \OS(z,t),\label{rho}\\
\sigma&:\OS(z,t) \stackrel{\sim}{\rightarrow} \OS(-z,-t),\label{sigma}\\
\omega&:\OS(z,t)\stackrel{\sim}{\rightarrow} \OS(-z,t^{-1}),\label{omega}\\
\pi&:\OS(z,t)
\stackrel{\sim}{\rightarrow} \OS(z,-t).\label{pi}
\end{align}
These reverse all orientations, scale by $(-1)^{\#\text{crossings}}$,
switch all positive crossings with negative crossings, and scale by
$(-1)^{\#\text{leftward cups}+\#\text{leftward caps}}$, respectively.
Let
\begin{equation}\label{signaut}
\#:\OS(z,t)\stackrel{\sim}{\rightarrow} \OS(z,t^{-1})
\end{equation}
denote $\sigma\circ\omega\circ\pi$.

\section{Connection to $\Rep U_q(\mathfrak{gl}_n)$}\label{sbt}

In this section, we assume until the final proof that $\k$ is a field
of characteristic 0, $q \in \k^\times$ is not a root of unity,
and
$z = q-q^{-1}$.
Fix $n \in \NN$ and let $U_q(\mathfrak{gl}_n)$ be the usual
quantized enveloping algebra
over $\k$; we include
the possibility that $n=0$ by interpreting $U_q(\mathfrak{gl}_0)$ as $\k$.
We denote the standard generators 
of $U_q(\mathfrak{gl}_n)$ by $\left\{e_i, f_i, d_j^{\pm}\:|\:1 \leq i <
  n, 1 \leq j \leq n\right\}$.
This is a well-known object, so the reader should have no trouble surmising the relations on being
told that the usual diagonal generator $k_i$ of $U_q(\mathfrak{sl}_n)$ is $d_i d_{i+1}^{-1}$.
We have the natural $U_q(\mathfrak{gl}_n)$-module $V^+$ on basis
$\left\{v_i^+\:|\:1 \leq i \leq n\right\}$ and the dual natural module 
$V^-$ on basis $\left\{v_i^-\:|\:1 \leq i\leq  n\right\}$. The actions of
the generators on these bases are given by the following formulae:
\begin{align*}
f_i v^{+}_j &= \delta_{i,j} v^{+}_{i+1},
&e_i v^{+}_j &= \delta_{i+1,j} v^{+}_i,
&
d_i  v^+_j &= q^{\delta_{i,j}}  v^+_j,
\\
f_i v^{-}_j &= \delta_{i+1,j} v^{-}_i,
&
e_i v^{-}_j &= \delta_{i,j} v^{-}_{i+1},
&d_i  v^-_j &= q^{- \delta_{i,j}} v^-_j.
\end{align*}
We use the comultiplication
$\Delta:U_q(\mathfrak{gl}_n) \rightarrow U_q(\mathfrak{gl}_n) \otimes U_q(\mathfrak{gl}_n)$ defined from
$$
\Delta(f_i) =
1 \otimes f_i  + f_i \otimes d_id_{i+1}^{-1},
\quad
\Delta(e_i) =
d_i^{-1} d_{i+1} \otimes e_i  + e_i \otimes 1,
\quad
\Delta(d_i) = d_i \otimes d_i.
$$
The corresponding antipode is given by $\mathrm{S}(e_i) = -d_i d_{i+1}^{-1} e_i, \mathrm{S}(f_i) = 
-f_i d_i^{-1} d_{i+1}$ and $\mathrm{S}(d_i) = d_i^{-1}$; for the user of \cite{Lubook}
we note that Lusztig's $v$ and $K_i$ are our
$q^{-1}$ and $k_i^{-1} = d_i^{-1} d_{i+1}$.

Let $\Rep U_q(\mathfrak{gl}_n)$ be the category
of finite-dimensional
$U_q(\mathfrak{gl}_n)$-modules that are isomorphic to finite direct sums of summands of
the modules obtained by taking tensor products of $V^+$ and
$V^-$; in the trivial case $n=0$, we mean the category of
finite-dimensional vector spaces. 
In general, $\Rep U_q(\mathfrak{gl}_n)$ is the usual category of finite-dimensional
representations of $U_q(\mathfrak{gl}_n)$ that are semisimple of type
{\bf 1} over its diagonal subalgebra.
It is well known that $\Rep U_q(\mathfrak{gl}_n)$ is a ribbon category, but we do not
want to fix a ribbon structure yet. Instead, we are going to use
Theorem~\ref{first} classify monoidal functors
$\Phi:\OS(z,t)\rightarrow\Rep U_q(\mathfrak{gl}_n)$ that take $\up$ to
$V^+$ and $\down$ to $V^-$.

There is a unique (up to scalars) non-degenerate bilinear pairing
$$\langle\cdot,\cdot\rangle:V^+ \times V^- \rightarrow \k$$ 
satisfying $\langle uv^+,v^- \rangle = \langle v^+,\mathrm{S}(u)
v^-\rangle$. Since there is freedom to rescale
the basis vectors $v_j^-$ by a global scalar, we may assume this is given
explicitly by the formula
$\langle v_i^+, v_j^- \rangle := (-1)^i q^{-i} \delta_{i,j}.
$
The associated evaluation and coevaluation maps
will be denoted 
\begin{align}\label{chip1}
\ev&:V^+ \otimes V^- \rightarrow \k, &&v_i^+ \otimes v_j^-
\mapsto (-1)^i q^{-i} \delta_{i,j},\\
\coev&:\k \rightarrow V^- \otimes V^+,
&&1 \mapsto  \sum_{j=1}^n (-1)^j q^j
v_j^- \otimes v_j^+.\label{chip2}
\end{align}
Then if we define $\Phi(C)
:= \coev$ and $\Phi(D) := \ev$, 
where $C = \mathord{
\begin{tikzpicture}[baseline = 1mm]
	\draw[<-,thick,darkblue] (0.3,0.3) to[out=-90, in=0] (0.1,0);
	\draw[-,thick,darkblue] (0.1,0) to[out = 180, in = -90] (-0.1,0.3);
\end{tikzpicture}
}$ and $D = \mathord{
\begin{tikzpicture}[baseline = 1mm]
	\draw[<-,thick,darkblue] (0.3,0) to[out=90, in=0] (0.1,0.3);
	\draw[-,thick,darkblue] (0.1,0.3) to[out = 180, in = 90] (-0.1,0);
\end{tikzpicture}
}\:$,
the relation (3) from Theorem~\ref{first} is satisfied.

Next we choose a candidate for the image of $S=
\begin{tikzpicture}[baseline = -.5mm]
	\draw[->,thick,darkblue] (0.2,-.2) to (-0.2,.3);
	\draw[line width=4pt,white,-] (-0.2,-.2) to (0.2,.3);
	\draw[thick,darkblue,->] (-0.2,-.2) to (0.2,.3);
\end{tikzpicture}
$. This should
be an isomorphism $V^+\otimes V^+\stackrel{\sim}{\rightarrow} V^+\otimes V^+$
satisfying the relations (1) and (2) from Theorem~\ref{first}.
One possible choice is to take $\Psi(S) := R$ where
\begin{align}\label{R}
R(v_i^+ \otimes v_j^+) &:=
\left\{
\begin{array}{ll}
v_j^+ \otimes v_i^+&\text{if $i < j$},\\
q v_j^+ \otimes v_i^+&\text{if $i = j$},\\
v_j^+ \otimes v_i^- + (q-q^{-1}) v_i^+ \otimes v_j^+&\text{if $i > j$}.
\end{array}
\right.
\end{align}
This formula is the $R$-matrix from 
\cite[$\S$32.1.4]{Lubook}.
The only other possibility for us would be to take $\Psi(S) := -R^{-1}$, but there
is no loss in generality in choosing the former, since one
can twist with the isomorphism $\#
:\OS(z,t)
\stackrel{\sim}{\rightarrow} \OS(z,t^{-1})$ from
(\ref{signaut})
which switches $S$ and $-S^{-1}$.
To see that $-R^{-1}$ is indeed the only other option, recall that
endomorphism algebra of $V^+\otimes V^+$ is two-dimensional, so any
isomorphism $V^+ \otimes V^+ \rightarrow V^+ \otimes V^+$ takes the form
$aR  + b$ for scalars $a$ and $b$. Then a
simple computation shows there are only two choices for these
scalars which satisfy the relation (1): $a = 1, b = 0$ or $a = -1, b = q-q^{-1}$.

Using the relation (4), we can determine the
image of
$T =
\begin{tikzpicture}[baseline = -.5mm]
	\draw[->,thick,darkblue] (0.2,-.2) to (-0.2,.3);
	\draw[line width=4pt,white,-] (-0.2,-.2) to (0.2,.3);
	\draw[thick,darkblue,<-] (-0.2,-.2) to (0.2,.3);
\end{tikzpicture}
$, as follows. We want to have $\Psi(T^{-1}) = (1_{V^-} \otimes 1_{V^+} \otimes \ev) \circ (1_{V^-} \otimes R
\otimes 1_{V^-}) \circ (\coev \otimes 1_{V^+} \otimes
1_{V^-})$. Computing the right hand side explicitly gives that
$$
\Psi(T^{-1})(v^+_i \otimes v^-_j) =
\left\{
\begin{array}{ll}
v^-_j \otimes v^+_i&\text{if $i \neq j$},\\
\displaystyle q v^-_i \otimes v^+_i + (q-q^{-1})
\sum_{k=j+1}^{n} (-q)^{k-i} v^-_{k}\otimes v^+_{k}&\text{if $i = j$};
\end{array}
\right.
$$
Inverting this map then gives us $\Psi(T)$:
$$
\Psi(T)(v^-_i \otimes v^+_j) =
\left\{
\begin{array}{ll}
v^+_j \otimes v^-_i&\text{if $i \neq j$},\\
\displaystyle q^{-1} v^+_i \otimes v^-_i - (q-q^{-1})
\sum_{k=i+1}^{n} (-q)^{i-k} v^+_{k}\otimes v^-_{k}\!&\text{if $i = j$}.
\end{array}
\right.
$$
Now the relation (4) holds.

The images under $\Psi$ of 
$C' = \mathord{
\begin{tikzpicture}[baseline = 1mm]
	\draw[-,thick,darkblue] (0.3,0.3) to[out=-90, in=0] (0.1,0);
	\draw[->,thick,darkblue] (0.1,0) to[out = 180, in = -90] (-0.1,0.3);
\end{tikzpicture}
}$ and $D'=\mathord{
\begin{tikzpicture}[baseline = 1mm]
	\draw[-,thick,darkblue] (0.3,0) to[out=90, in=0] (0.1,0.3);
	\draw[->,thick,darkblue] (0.1,0.3) to[out = 180, in = 90] (-0.1,0);
\end{tikzpicture}
}\:$
must come from another
non-degenerate pairing
$$\langle\cdot,\cdot\rangle':V^-\times V^+ \rightarrow \k$$ 
such that $\langle u v^-,v^+ \rangle' = \langle v^-,\mathrm{S}(u)
v^+\rangle'$.
There is a unique (up to scalars) such pairing, 
namely,
$
\langle v_i^-, v_j^+\rangle' := \eps (-1)^{i} q^{i-n-1} \delta_{i,j}
$
for $\eps \in \k^\times$.
We denote the corresponding evaluation and coevaluation maps by
\begin{align}\label{chip3}
\ev'&:V^-\otimes V^+ \rightarrow \k, &&v_i^-\otimes v_j^+ \mapsto \eps
(-1)^i q^{i-n-1}\delta_{i,j},\\
\coev'&:\k \rightarrow
V^+ \otimes V^-,&
&1\mapsto \eps^{-1}\sum_{j=1}^n (-1)^j q^{n+1-j} v_j^+\otimes v_j^-.\label{chip4}
\end{align}
Then, for some choice of $\eps$, we have that
$\Phi(C') = \coev'$ and $\Phi(D') = \ev'$.
To determine the possibilities for $\eps$,
from the first equation in (\ref{ready}), we know that $\ev' = t \ev\circ \Psi(T)$.
Applying this equation to the vector $v_n^- \otimes v_n^+$ quickly produces the
equation
$
\eps (-1)^n q^{-1} = t q^{-1} (-1)^n q^{-n},
$
hence, $t = \eps q^n$. 
From the second equation in (\ref{ready}), we know that $\coev' =
t \Psi(T) \circ \coev$. Looking at the $v_1^+ \otimes
v_1^-$-coefficient of the image of $1$ under the two sides of this
equation gives
$
-\eps^{-1} q^n = - t,
$
hence, $t = \eps^{-1} q^n$. We deduce that $\eps = \eps^{-1}$, i.e.,
$\eps = \pm 1$. 
Since we can twist with the isomorphism
$\pi:\OS(z,t)\stackrel{\sim}{\rightarrow}\OS(z,-t)$ 
from (\ref{pi}) which
takes $C'$ to $-C'$ and $D'$ to $-D'$, we are reduced without loss of generality to the case that $\eps  =
+1$ and $t = q^n$. 
It can then be checked that (\ref{ready}) holds fully.

Finally, we have that $\coev \circ
\ev' = [n]_q$, so that relation (5) from Theorem~\ref{first} holds too, and the theorem
implies that the functor $\Psi$ is well defined.
We have proved the following lemma, a version of which was used already in \cite{Turaev3}.

\begin{lemma}\label{heartbeat}
Assume $\k$ is of characteristic zero,  $z=q-q^{-1}$ for
generic $q \in
\k^\times$, and
$t = q^n$ for $n \in \NN$.
There is a $\k$-linear monoidal functor
$\Psi:\OS(z, t)
\rightarrow \Rep U_q(\mathfrak{gl}_n)$
sending $\up \mapsto V^+, \down \mapsto V^-$, 
the positive upward crossing
to the $R$-matrix from (\ref{R}),
the rightward cap and rightward cup to the maps $\ev$ and $\coev$
from (\ref{chip1})--(\ref{chip2}), and the leftward cap and leftward cup to the maps $\ev'$ and
$\coev'$ from (\ref{chip3})--(\ref{chip4}) taking $\eps = +1$.
\end{lemma}

\begin{remark}
Pre-composing the functor $\Psi$ with one or both of the isomorphisms
$\#$ and $\pi$ 
from (\ref{pi})--(\ref{signaut}) gives three more
such functors with $t=q^n$ replaced by $q^{-n}, -q^n$ or $-q^{-n}$.
The arguments above actually show that these four functors constitute 
essentially all possible $\k$-linear 
monoidal functors $\OS(z,t)\rightarrow\Rep U_q(\mathfrak{gl}_n)$ taking $\up\mapsto V^+$ and
$\down\mapsto V^-$. 
Each of the four choices for this functor
corresponds to a ribbon structure on the monoidal
category
$\operatorname{Rep} U_q(\mathfrak{gl}_n)$.
The standard ribbon structure on $\operatorname{Rep}
U_q(\mathfrak{gl}_n)$
is the one coming from $\Psi$ itself, i.e., $t = q^n$, and we will only
use this from now it.
Taking $\Psi\circ\#$, i.e., $t = q^{-n}$,
gives a non-standard ribbon structure on $\Rep U_q(\mathfrak{gl}_n)$
with positive upward crossing $-R^{-1}$,
rightward cap and cup being $\ev$ and $\coev$, and leftward cap
and cup being $\ev'$ and $\coev'$ with $\eps = -1$;
this is the ribbon category denoted
$\Rep U_q(\mathfrak{gl}_{-n})$ in the introduction.
\end{remark}

In order to prove Theorems~\ref{bt} and \ref{webs} from the
introduction, we also need 
the {\em Iwahori-Hecke algebra}
$H_r$ associated to the symmetric group $\Sym_r$. This is
the associative $\k$-algebra with generators $S_1,\dots,S_{r-1}$
subject to the relations
\begin{align}\label{hecke}
(S_i-q)(S_i+q^{-1}) &= 0,
&S_i S_j &= S_jS_i \text{ if $|i-j|>1$,}&
S_i S_{i+1} S_i &= S_{i+1} S_i S_{i+1}.
\end{align}
As is well known, $H_r$ has dimension $r!$, with basis
$\{S_w\:|\:w \in \Sym_r\}$ defined as usual by letting $S_w$ be the word in
the generators $S_i$ arising from any reduced
expression for $w$.
It is obvious from the defining relations that there is a
homomorphism 
\begin{align}\label{iotap1}
\i_r
:
H_r&\rightarrow \End_{\OS(z,t)}(\up^r)
\end{align}
sending $S_i$ to the positive crossing 
$\begin{tikzpicture}[baseline = -.5mm]
	\draw[->,thick,darkblue] (0.2,-.2) to (-0.2,.3);
	\draw[line width=4pt,white,-] (-0.2,-.2) to (0.2,.3);
	\draw[thick,darkblue,->] (-0.2,-.2) to (0.2,.3);
\end{tikzpicture}
$ 
of the $i$th and $(i+1)$th strands, numbering strands
in increasing order from right to left.
Also for $\lambda \vdash r$ we let
\begin{equation}\label{symmetrizers}
x_\lambda := \sum_{w \in \Sym_\lambda} q^{\ell(w)} S_w,\qquad
y_\lambda := \sum_{w \in \Sym_\lambda} (-q)^{-\ell(w)} S_w,
\end{equation}
where $\Sym_\lambda$ denotes the usual parabolic subgroup
$\Sym_{\lambda_1}\times\Sym_{\lambda_2}\times\cdots$ 
of $\Sym_r$.
Assuming $q$ is not a root of unity, there is a unique (up to sign) idempotent
\begin{equation}\label{youngsymmetrizer}
e_\lambda \in y_{\lambda^\trans} H_r x_\lambda.
\end{equation}
This is the {\em Young symmetrizer}.
For example,
\begin{align}
e_{(n)} &= 
\frac{q^{-\frac{1}{2}n(n-1)}}{[n]_q^!}
\sum_{w \in \Sym_{n}} q^{\ell(w)} h_w,&
e_{(1^{n})} &= 
\frac{q^{\frac{1}{2}n(n-1)}}{[n]_q^!}
\sum_{w \in \Sym_{n}} (-q)^{-\ell(w)} h_w,
\end{align}
which correspond to the trivial and the sign representations of $H_r$, respectively.
The algebra involution
\begin{equation}\label{sign}
\#:H_r \rightarrow H_r, \qquad
S_i \mapsto -S_i^{-1}
\end{equation}
interchanges $e_{(n)}$ and $e_{(1^n)}$. 
Given an $H_r$-module $M$, the $H_r$-module $M^\#$
obtained from $M$ by twisting the action by $\#$
gives the $q$-analog
of ``tensoring with sign.''

\begin{proof}[Proof of Theorem~\ref{webs}]
Since we can compose with $\#$, which switches
$t = q^n$ with $t = q^{-n}$
and $e_{(n)}$ with $e_{(1^n)}$, we are reduced just to proving the theorem
in the case that $\eps = +$.
Then the appropriate monoidal functor $\Psi$ is the one constructed in Lemma~\ref{heartbeat}.

In this paragraph, we show that $\Psi$ is full.
Take 
$\a,\b\in\words$ such that $x$ (resp. $x'$)
letters of $\a$ and $y$ (resp. $y'$) letters of $\b$ are equal to $\down$
(resp. $\up$).
The space $\Hom_{\OS(z, t)}(\a,\b)$ is zero
unless $r:=x'+y=x+y'$, so we may assume that is the case. 
Let $a:\up^{x'} \rightarrow
\a \up^{x}$ be the unique morphism that consists of $x$ nested
rightward cups on top of $x'$ vertical upward strands. Let $b:\up^{y} \b
\rightarrow \up^{y'}$ be the unique morphism that consists of
$y$ nested rightward caps on top of $y'$ vertical strands.
The linear map
\begin{align}\label{exactly}
\theta:\Hom_{\OS(z,t)}(\a,\b) &\rightarrow
\Hom_{\OS(z,t)}(\up^r, \up^r),\\\notag
f &\mapsto (b \otimes 1_{\up^{x}}) \circ
(1_{\up^{y}}\otimes f \otimes  1_{\up^{x}}) \circ
(1_{\up^{y}}\otimes a)
\end{align} 
has an obvious two-sided inverse, hence, it is a vector space isomorphism. For
example, taking $\a = \down\up\down\up$ and $\b = \up\down$, the map $\theta$ sends
$$
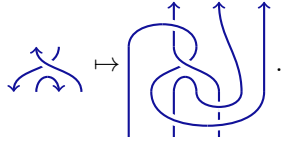
\mathord{
\begin{tikzpicture}[baseline = -.25mm]
  \draw[<-,thick,darkblue] (-.6,-.3) to[out=90,in=-90] (0,.3);
  \draw[-,line width=4pt,white] (0.3,-.3) to[out=90,in=-90] (-0.3,.3);
  \draw[->,thick,darkblue] (0.3,-.3) to[out=90,in=-90] (-0.3,.3);
  \draw[-,thick,darkblue] (-0.3,-.3) to[out=90,in=180] (-0.15,-.1);
  \draw[->,thick,darkblue] (-0.15,-.1) to[out=0,in=90] (0,-.3);
\end{tikzpicture}
}
\mapsto
\mathord{
\begin{tikzpicture}[baseline = -.25mm]
  \draw[-,thick,darkblue] (-.6,-.3) to[out=90,in=-90] (0,.3);
  \draw[<-,thick,darkblue] (-.3,.9) to (-.3,.3);
  \draw[-,line width=4pt,white] (0.3,-.3) to[out=90,in=-90] (-0.3,.3);
  \draw[-,thick,darkblue] (0.3,-.3) to[out=90,in=-90] (-0.3,.3);
  \draw[-,line width=4pt,white] (0,.3) to[out=90,in=0] (-0.45,.6);
  \draw[-,thick,darkblue] (0,.3) to[out=90,in=0] (-0.45,.6);
  \draw[-,thick,darkblue] (-.9,-.9) to (-0.9,.3);
  \draw[-,thick,darkblue] (-0.45,.6) to[out=180,in=90] (-0.9,.3);
  \draw[-,thick,darkblue] (0.3,-.9) to (0.3,-.3);
  \draw[-,thick,darkblue] (-0.3,-.3) to[out=90,in=180] (-0.15,-.1);
  \draw[-,thick,darkblue] (-0.3,-.3) to (-0.3,-.9);
  \draw[-,thick,darkblue] (-0.15,-.1) to[out=0,in=90] (0,-.3);
  \draw[-,line width=4pt,white] (-0.6,-.3) to[out=-90,in=180] (0.15,-.75);
  \draw[-,thick,darkblue] (-0.6,-.3) to[out=-90,in=180] (0.15,-.75);
  \draw[-,line width=4pt,white] (0.15,-.75) to[out=0,in=-90] (.9,-.3);
  \draw[-,thick,darkblue] (0.15,-.75) to[out=0,in=-90] (.9,-.3);
  \draw[-,line width=4pt,white] (0,-.3) to[out=-90,in=180] (.3,-.5);
  \draw[-,thick,darkblue] (0,-.3) to[out=-90,in=180] (.3,-.5);
  \draw[-,line width=4pt,white] (.3,-.5) to[out=0,in=-90] (.6,-.3);
  \draw[-,thick,darkblue] (.3,-.5) to[out=0,in=-90] (.6,-.3);
  \draw[->,thick,darkblue] (0.6,-.3) to[out=90,in=-90] (0.3,.9);
  \draw[->,thick,darkblue] (0.9,-.3) to (0.9,.9);
\end{tikzpicture}
}\:.
$$
Since $\Psi$ is a monoidal functor, there is an isomorphism
\begin{align*}
\phi:
\Hom_{U_q(\mathfrak{gl}_n)}(\Psi(\a),\Psi(\b)) &\stackrel{\sim}{\rightarrow}
\Hom_{U_q(\mathfrak{gl}_n)}((V^+)^{\otimes r}, (V^+)^{\otimes r}),\\
g &\mapsto (\Psi(b) \otimes 1_{V^+}^{\otimes x}) \circ
(1_{V^+}^{\otimes y}\otimes g \otimes  1_{V^+}^{\otimes x}) \circ
(1_{V^+}^{\otimes y}\otimes \Psi(a))
\end{align*}
making the following diagram commute:
\begin{equation}\label{queenie}
\begin{CD}
\Hom_{\OS(z,t)}(\a,\b) &@>\sim>\theta>&
\Hom_{\OS(z,t)}(\up^r, \up^r) &@<\i_r<<&H_r\\
@V\Psi VV&&@VV\Psi V\\
\Hom_{U_q(\mathfrak{gl}_n)}(\Psi(\a),\Psi(\b)) &@>\sim >\phi>&
\Hom_{U_q(\mathfrak{gl}_n)}((V^+)^{\otimes r}, (V^+)^{\otimes r}).
\end{CD}
\end{equation}
The composition
$\j_r:H_r \rightarrow \End_{U_q(\mathfrak{gl}_n)}((V^+)^{\otimes
  r})$ of $\i_r$ and the right hand $\Psi$
is a homomorphism
studied in \cite{J} in the context of ``quantized Schur-Weyl
reciprocity.'' It is shown there that $\j_r$ is surjective. Hence, the right hand $\Psi$ is surjective. 
The commutativity of the diagram then implies the
analogous statement for the left hand $\Psi$. 

As $\Rep U_q(\mathfrak{gl}_n)$ is additive Karoubian, the functor
$\Psi$ extends to a 
full functor $\dot\Psi:\dot\OS(z,t) \rightarrow \Rep
U_q(\mathfrak{gl}_n)$. 
Let $\mathcal N$ be the tensor ideal of $\dot\OS(z,t)$ of negligible morphisms.
Since $\Rep U_q(\mathfrak{gl}_n)$ is absolutely
semisimple, $\dot\Psi$ induces a fully faithful functor
$\bar\Psi:\dot\OS(z,t) / \mathcal N \rightarrow \Rep
U_q(\mathfrak{gl}_n)$
by the argument from the proof of \cite[Th\'eor\`eme 6.2]{D}.
This functor is also dense
since every object of $\Rep U_q(\mathfrak{gl}_n)$ is a summand of some
tensor product of the modules $V^+$ and $V^-$.
So it is a monoidal equivalence.

To complete the proof, we need to show that 
$\mathcal N$ is generated as
an additive $\k$-linear tensor ideal of $\dot\OS(z,t)$ by the
morphism $\i_{n+1}(e_{(1^{n+1})})$.
It suffices to show that the kernel\footnote{
We mean the 
tensor ideal of $\OS(z,t)$ defined by the kernels of the maps
$\Psi:\Hom_{\OS(z,t)}(\a,\b)\twoheadrightarrow 
\Hom_{U_q(\mathfrak{gl}_n)}(\Psi(\a),\Psi(\b))$ for all $\a,\b\in\words$.}
of the original functor $\Psi$
is the $\k$-linear tensor ideal $\mathcal I$ of $\OS(z,t)$ generated by
$\i_{n+1}(e_{(1^{n+1})})$.
Jimbo's results show that the homomorphism $\j_r$ introduced above
is injective when $r \leq n$, and that $\ker \j_r$ is the ideal of $H_r$ generated
by $e_{(1^{n+1})}$ (viewed as an element of $H_r$ via the natural
embedding
$H_{n+1} \hookrightarrow H_r$)
when $r > n$.
In particular, taking $r=n+1$, this shows that $\Psi(\i_{n+1}(e_{1^{n+1}})) = 0$,
so $\Psi$ induces a monoidal functor $\tilde\Psi:\OS(q, q^n) / \mathcal I
\rightarrow \Rep U_q(\mathfrak{gl}_n)$.
The commuting diagram (\ref{queenie}) becomes
$$
\begin{CD}
\Hom_{\OS(z,t) / \mathcal I}(\a,\b) &@>\sim>\tilde\theta>&
\Hom_{\OS(z,t) / \mathcal I}(\up^r, \up^r)
&@<\tilde\i_r<<&H_r / I_r\\
@V\tilde\Psi VV&&@VV\tilde\Psi V\\
\Hom_{U_q(\mathfrak{gl}_n)}(\Psi(\a),\Psi(\b)) &@>\sim >\phi>&
\Hom_{U_q(\mathfrak{gl}_n)}((V^+)^{\otimes r}, (V^+)^{\otimes r}),
\end{CD}
$$
where $I_r := \{0\}$ if $r \leq n$ and $I_r :=  \langle e_{(1^{n+1})}
\rangle$ if $r > n$.
The isomorphism $\tilde\theta$ in this diagram is defined in the same way as
$\theta$, indeed, it is induced by $\theta$ in an obvious way.
Also when $r > n$ the map $\i_r$ takes $e_{(1^{n+1})} \in H_r$ to $\up^{r-n-1}
\i_{n+1}(e_{(1^{n+1})})$. This morphism lies in $\mathcal I$,
showing that $\i_r$ induces the homomorphism $\tilde\i_r$ 
indicated in the diagram.
Now the composition of $\tilde\i_r$ and
the right hand $\tilde\Psi$ is
an isomorphism.
Also it is obvious from the definition of $\OS(z, t)$ 
that $\i_r$, hence, $\tilde\i_r$ is surjective.
We deduce that the right hand $\tilde\Psi$ is an
isomorphism, hence, the left hand one is too. 
This shows that $\mathcal I$ is indeed the kernel of $\Psi$.
\end{proof}

\begin{remark}
Let notation be as in Theorem~\ref{webs}, taking $\eps = +$.
If $\a,\b\in\words$ 
are objects such that 
$x$ (resp. $x'$)
letters of $\a$ and $y$ (resp. $y'$) letters of $\b$ are equal to $\down$
(resp. $\up$), and
$r := x'+y=x+y'$ satisfies $r \leq n$, then $\Psi$ is injective on
$\Hom_{\OS(z,t)}(\a,\b)$ and
\begin{equation}\label{swr}
\dim_{\k} 
\Hom_{\OS(z,t)}(\a,\b)= \dim H_r = r!.
\end{equation}
These assertions follow from the proof just explained: when $r \leq n$ the map
$\j_r$ is an isomorphism so all of the vertical maps in (\ref{queenie}) are
isomorphisms too.
(Theorem~\ref{bt} implies that the formula
 (\ref{swr}) holds without the restriction $r \leq n$, but 
we will use this special case in its proof.)
\end{remark}

\begin{proof}[Proof of Theorem~\ref{bt}]
In this proof, we are going to allow $\k$ to vary, so may add an
additional subscript, denoting $\OS(z,t)$ and $\widehat{\OS}(z,t)$ 
instead by $\OS(z,t)_\k$ and $\widehat{\OS}(z,t)_\k$, respectively.
We first establish the result for the morphism spaces of $\OS$.
Take $\a, \b\in\words$ such that
$x$ (resp. $x'$)
letters of $\a$ and $y$ (resp. $y'$) letters of $\b$ are equal to $\down$
(resp. $\up$), and
$r := x'+y=x+y'$.
Let $B(\a,\b)$ be some set of reduced lifts of the
$(\a,\b)$-matchings, so that $|B(\a,\b)| = r!$.
It is straightforward to see 
for any $\k, z$ and $t$ 
that $B(\a,\b)$ spans $\Hom_{\OS(z,t)_\k}(\a,\b)$. We
need to show that it is also linearly independent.

Consider first the
case that $\k = \ZZ[z,z^{-1}, t, t^{-1}]$. 
Take a linear relation $$
\sum_{b \in B(\a,\b)} c_b(z,t) b = 0
$$
for $c_b(z,t) \in \ZZ[z,z^{-1}, t, t^{-1}]$.
For any $n \geq r$, 
we can consider the obvious strict $\ZZ$-linear monoidal functor $\omega:\OS(z,t)_{\ZZ[z,z^{-1},t,t^{-1}]} \rightarrow
\OS(q-q^{-1}, q^n)_{\QQ(q)}$ sending $z \mapsto q-q^{-1}, t \mapsto
q^n$, and generating morphisms to the generating morphisms with the
same names. 
This functor maps $B(\a,\b)$ to a spanning set for
$\Hom_{\OS(q-q^{-1}, q^n)_{\QQ(q)}}(\a,\b)$. Since $|B(\a,\b)| = r!$, we
deduce from (\ref{swr}) that $\omega(B(\a,\b))$ is linearly
independent too. Hence, $c_b(q-q^{-1}, q^n) = 0$ for each $b \in
B(\a,\b)$. Since this is true for infinitely many values of $n$, it
follows that each $c_b(z, t) = 0$.

Now take an arbitrary commutative ground ring $\k$ and parameters $\bar z,\bar t \in
\k^\times$.
Viewing $\k$ as a $\ZZ[z,z^{-1}, t,t^{-1}]$-module so $z$ and $t$ act
via $\bar z$ and $\bar t$, there is an obvious strict $\k$-linear monoidal functor
$\OS(\bar z, \bar t)_\k
\rightarrow
\OS(z,t)_{\ZZ[z,z^{-1},t,t^{-1}]} \otimes_{\ZZ[z,z^{-1}, t,t^{-1}]} \k$
sending generating morphisms to the generating morphisms with the same
name tensored with $1_\k$.
This functor sends $B(\a,\b)$ to a set of morphisms which we already know is linearly
independent thanks to the previous paragraph. Hence $B(\a,\b)$ itself is
linearly independent.
This completes the proof for $\OS$.

It remains to treat the extended category $\widehat{\OS}$.
Again, it is clear that the morphisms from the statement of
Theorem~\ref{bt} span, so we just need to establish linear independence.
For all but the case $\a=\b=\varnothing$, this follows immediately since
we have already established linear independence in the quotient category
$\OS(z,t)_\k$.
Thus, we are left with showing that $1_\varnothing$ and $\Bubble$ are
linearly independent in $\Hom_{\widehat\OS(z,t)_\k}(\varnothing,\varnothing)$.
By the same arguments as in the previous two paragraphs, this follows
if we can check it in
$\Hom_{\widehat\OS(q-q^{-1},q^n)_{\QQ(q)}}(\varnothing,\varnothing)$
for infinitely many values of $n$. This is done in the final paragraph
of the proof.

So 
 assume that $\k = \QQ(q)$, $z = q-q^{-1}$ and
$t = q^n$ for $n \in \NN$. 
We define a new strict $\k$-linear monoidal
category $\mathcal C$.
Its objects are  as in $\OS(z, t)$ with the
same tensor product, and its morphisms are defined from
$$
\Hom_{\mathcal C}(\a,\b) := 
\left\{
\begin{array}{ll}
\Hom_{\OS(z, t)}(\a,\b)&\text{if $\a \neq \varnothing$ or $\b \neq
    \varnothing$,}\\
\Hom_{\OS(z,
  t)}(\varnothing,\varnothing) \oplus \k&\text{if $\a = \b = \varnothing$.}
\end{array}\right.
$$
So $\Hom_{\mathcal C}(\varnothing,\varnothing)$ is two-dimensional with basis
$(1_\varnothing,0)$ and $(0,1_{\k})$.
Horizontal and vertical composition of most of the morphisms in $\mathcal C$ is induced
by the compositions in $\OS(z, t)$ in the obvious way;
the horizontal and vertical composition of
$(0,1_{\k})$ with any morphism in $\Hom_{\mathcal C}(\a,\b)$ is zero
if  $\a \neq \varnothing$ or $\b \neq \varnothing$; the horizontal and vertical composition of
$(0,1_{\k})$ with $(a 1_\varnothing,b 1_{\k})$ is $(0,b
1_{\k})$.
Now the point is that there is a strict $\k$-linear monoidal functor
$\widehat\OS(z, t) \rightarrow \mathcal C$
sending objects and generating morphisms to their images under the
quotient functor to $\OS(z, t)$ embedded
(non-unitally) into
$\mathcal C$. Due to the relation ($\text{D}$) in 
$\OS(z, t)$, the morphism
$\Bubble\in \Hom_{\widehat\OS(z,t)}(\varnothing,\varnothing)$ maps to $([n]_q 1_\varnothing, 0)$,
while the identity element $1_\varnothing \in \Hom_{\widehat\OS(z,t)}(\varnothing,\varnothing)$ must map to the identity element
$(1_\varnothing, 1_{\k}) \in \Hom_{\mathcal C}(\varnothing,\varnothing)$.
Since $([n]_q 1_{\varnothing}, 0)$ and $(1_\varnothing, 1_{\k})$
  are linearly independent, it follows that $\Bubble$ and
  $1_\varnothing$ are linearly independent in 
$\Hom_{\widehat\OS(z,t)}(\varnothing,\varnothing)$.
\end{proof}

\begin{remark}\label{websrem}
A modified version of Theorem~\ref{webs} holds over any
field $\k$ for any $q \in \k^\times\setminus\{\pm 1\}$.
Let $\qGL_n$ be the quantum general linear group over $\k$ at
parameter $q$; its coordinate algebra $\k[\qGL_n]$ is
the localization of Manin's quantized coordinate algebra of $n
\times n$ matrices at the 
quantum determinant as in \cite{PW}.
Let $\Rep \qGL_n$ be the category of rational 
$\qGL_n$-modules ($=$finite-dimensional $\k[\qGL_n]$-comodules).
Then there is a full $\k$-linear monoidal functor
$\Psi:\OS(q-q^{-1},t^{n}) \rightarrow
\Rep \qGL_n$ sending $\up \mapsto V^+$ and $\down \mapsto V^-$
defined just like in Lemma~\ref{heartbeat}. 
It induces a monoidal equivalence
\begin{equation}
\bar\Psi:\dot\OS(q-q^{-1},q^n) / \mathcal N \stackrel{\approx}{\longrightarrow}
\Tilt' \qGL_n
\end{equation}
where $\mathcal N$ is the 
additive $\k$-linear tensor ideal 
generated by $\i_{n+1}(e_{(1^{n+1})})$,
and $\Tilt' \qGL_n$ is the full subcategory of 
$\Rep \qGL_n$ consisting of all modules isomorphic to direct
sums of summands of tensor powers of $V^+$ and $V^-$.
The proof of this is similar to the proof of Theorem~\ref{webs}, using the
generalization of Schur-Weyl duality 
from \cite[Theorem 6.2]{DPS} and \cite[Theorem 4]{H}.
\end{remark}

\section{The affine oriented skein category}\label{aos}

In this section, $\k$ is a commutative ground ring and $z,t \in
\k^\times$ are arbitrary.
The {\em affine oriented skein category} $\AOS(z,t)$ is the strict $\k$-linear monoidal category
obtained from $\OS(z,t)$ by adjoining an additional generating
morphism
$\mathord{
\begin{tikzpicture}[baseline = -1mm]
      \node at (0,0) {$\color{darkblue}\scriptstyle\bullet$};
	\draw[<-,thick,darkblue] (0,0.25) to (0,-0.25);
\end{tikzpicture}
}$ and a two-sided inverse of this morphism,
subject to the additional relation ($\text{A}$) from Figure 1.
For any $n \in \ZZ$, we write
$\mathord{
\begin{tikzpicture}[baseline = -1mm]
      \node at (0,0) {$\color{darkblue}\scriptstyle\bullet$};
      \node at (.2,0) {$\color{darkblue}\scriptstyle n$};
	\draw[<-,thick,darkblue] (0,0.25) to (0,-0.25);
\end{tikzpicture}
}$
for the $n$th power of this additional generator.
The relation ($\text{A}$) comes from
the
{\em affine Hecke algebra} $AH_r$, which is generated by the
Iwahori-Hecke algebra
$H_r$ from (\ref{hecke}) plus additional elements $X_1^{\pm 1},\dots,X_r^{\pm 1}$
subject to the relations
\begin{align}\label{AHR}
X_i X_j &= X_j X_i,
&S_i X_i S_i &= X_{i+1}
\end{align}
for all $i,j$. There is an algebra
homomorphism
\begin{equation}
\j_r:AH_r \rightarrow \End_{\AOS(z,t)}(\up^r)
\end{equation}
defined  on $S_1,\dots,S_{r-1}$ in the same way as for the homomorphism $\i_r$ from (\ref{iotap1}),
and sending $X_i$ to the dot on the $i$th strand from the right.




\begin{lemma}
In $\AOS(z,t)$, we have that 
$\mathord{
\begin{tikzpicture}[baseline = -1mm]
	\draw[->,thick,darkblue] (0.08,.4) to (0.08,-.4);
      \node at (0.08,0) {$\color{darkblue}\scriptstyle\bullet$};
\end{tikzpicture}
}:=\:\mathord{
\begin{tikzpicture}[baseline = -1mm]
  \draw[->,thick,darkblue] (0.3,0) to (0.3,-.4);
	\draw[-,thick,darkblue] (0.3,0) to[out=90, in=0] (0.1,0.4);
	\draw[-,thick,darkblue] (0.1,0.4) to[out = 180, in = 90] (-0.1,0);
	\draw[-,thick,darkblue] (-0.1,0) to[out=-90, in=0] (-0.3,-0.4);
	\draw[-,thick,darkblue] (-0.3,-0.4) to[out = 180, in =-90] (-0.5,0);
  \draw[-,thick,darkblue] (-0.5,0) to (-0.5,.4);
   \node at (-0.1,0) {$\color{darkblue}\scriptstyle\bullet$};
\end{tikzpicture}
}=\:
\mathord{
\begin{tikzpicture}[baseline = -1mm]
  \draw[-,thick,darkblue] (0.3,0) to (0.3,.4);
	\draw[-,thick,darkblue] (0.3,0) to[out=-90, in=0] (0.1,-0.4);
	\draw[-,thick,darkblue] (0.1,-0.4) to[out = 180, in = -90] (-0.1,0);
	\draw[-,thick,darkblue] (-0.1,0) to[out=90, in=0] (-0.3,0.4);
	\draw[-,thick,darkblue] (-0.3,0.4) to[out = 180, in =90] (-0.5,0);
  \draw[->,thick,darkblue] (-0.5,0) to (-0.5,-.4);
   \node at (-0.1,0) {$\color{darkblue}\scriptstyle\bullet$};
\end{tikzpicture}
}\,.$
Moreover, all of the following relations hold:
\begin{align}
\mathord{
\begin{tikzpicture}[baseline = 1mm]
	\draw[<-,thick,darkblue] (0.4,-0.1) to[out=90, in=0] (0.1,0.3);
	\draw[-,thick,darkblue] (0.1,0.3) to[out = 180, in = 90] (-0.2,-0.1);
      \node at (-0.14,0.15) {$\color{darkblue}\scriptstyle\bullet$};
\end{tikzpicture}
}
&=
\mathord{
\begin{tikzpicture}[baseline = 1mm]
      \node at (0.335,0.15) {$\color{darkblue}\scriptstyle\bullet$};
	\draw[<-,thick,darkblue] (0.4,-0.1) to[out=90, in=0] (0.1,0.3);
	\draw[-,thick,darkblue] (0.1,0.3) to[out = 180, in = 90] (-0.2,-0.1);
\end{tikzpicture}
}\:,
&
\mathord{
\begin{tikzpicture}[baseline = -1.5mm]
	\draw[<-,thick,darkblue] (0.4,0.1) to[out=-90, in=0] (0.1,-0.3);
	\draw[-,thick,darkblue] (0.1,-0.3) to[out = 180, in = -90] (-0.2,0.1);
      \node at (-0.14,-0.15) {$\color{darkblue}\scriptstyle\bullet$};
\end{tikzpicture}
}&
=
\mathord{
\begin{tikzpicture}[baseline = -1.5mm]
      \node at (0.335,-0.15) {$\color{darkblue}\scriptstyle\bullet$};
	\draw[<-,thick,darkblue] (0.4,0.1) to[out=-90, in=0] (0.1,-0.3);
	\draw[-,thick,darkblue] (0.1,-0.3) to[out = 180, in = -90] (-0.2,0.1);
\end{tikzpicture}
}\:,\label{d5}\\
\mathord{
\begin{tikzpicture}[baseline = 1mm]
      \node at (0.33,0.15) {$\color{darkblue}\scriptstyle\bullet$};
	\draw[-,thick,darkblue] (0.4,-0.1) to[out=90, in=0] (0.1,0.3);
	\draw[->,thick,darkblue] (0.1,0.3) to[out = 180, in = 90] (-0.2,-0.1);
\end{tikzpicture}
}
&=
\mathord{
\begin{tikzpicture}[baseline = 1mm]
	\draw[-,thick,darkblue] (0.4,-0.1) to[out=90, in=0] (0.1,0.3);
	\draw[->,thick,darkblue] (0.1,0.3) to[out = 180, in = 90] (-0.2,-0.1);
      \node at (-0.14,0.15) {$\color{darkblue}\scriptstyle\bullet$};
\end{tikzpicture}
}\:,
&
\mathord{
\begin{tikzpicture}[baseline = -1.5mm]
      \node at (0.335,-0.15) {$\color{darkblue}\scriptstyle\bullet$};
	\draw[-,thick,darkblue] (0.4,0.1) to[out=-90, in=0] (0.1,-0.3);
	\draw[->,thick,darkblue] (0.1,-0.3) to[out = 180, in = -90] (-0.2,0.1);
\end{tikzpicture}
}&
=
\mathord{
\begin{tikzpicture}[baseline = -1.5mm]
	\draw[-,thick,darkblue] (0.4,0.1) to[out=-90, in=0] (0.1,-0.3);
	\draw[->,thick,darkblue] (0.1,-.3) to[out = 180, in = -90] (-0.2,0.1);
      \node at (-0.14,-0.15) {$\color{darkblue}\scriptstyle\bullet$};
\end{tikzpicture}
}\:,\label{d4}\\
\label{d0}
\mathord{
\begin{tikzpicture}[baseline = -.5mm]
	\draw[->,thick,darkblue] (0.28,-.3) to (-0.28,.4);
      \node at (0.165,-0.15) {$\color{darkblue}\scriptstyle\bullet$};
	\draw[line width=4pt,white,-] (-0.28,-.3) to (0.28,.4);
	\draw[thick,darkblue,->] (-0.28,-.3) to (0.28,.4);
\end{tikzpicture}
}&=\mathord{
\begin{tikzpicture}[baseline = -.5mm]
	\draw[thick,darkblue,->] (-0.28,-.3) to (0.28,.4);
	\draw[-,line width=4pt,white] (0.28,-.3) to (-0.28,.4);
	\draw[->,thick,darkblue] (0.28,-.3) to (-0.28,.4);
      \node at (-0.14,0.23) {$\color{darkblue}\scriptstyle\bullet$};
\end{tikzpicture}
}\
\:,
&\qquad
\mathord{
\begin{tikzpicture}[baseline = -.5mm]
	\draw[thick,darkblue,->] (-0.28,-.3) to (0.28,.4);
      \node at (-0.16,-0.15) {$\color{darkblue}\scriptstyle\bullet$};
	\draw[-,line width=4pt,white] (0.28,-.3) to (-0.28,.4);
	\draw[->,thick,darkblue] (0.28,-.3) to (-0.28,.4);
\end{tikzpicture}
}&= 
\mathord{
\begin{tikzpicture}[baseline = -.5mm]
	\draw[->,thick,darkblue] (0.28,-.3) to (-0.28,.4);
	\draw[line width=4pt,white,-] (-0.28,-.3) to (0.28,.4);
	\draw[thick,darkblue,->] (-0.28,-.3) to (0.28,.4);
      \node at (0.145,0.23) {$\color{darkblue}\scriptstyle\bullet$};
\end{tikzpicture}
}\:,\\
\mathord{
\begin{tikzpicture}[baseline = -.5mm]
	\draw[<-,thick,darkblue] (0.28,-.3) to (-0.28,.4);
	\draw[line width=4pt,white,-] (-0.28,-.3) to (0.28,.4);
	\draw[thick,darkblue,->] (-0.28,-.3) to (0.28,.4);
      \node at (-0.15,0.24) {$\color{darkblue}\scriptstyle\bullet$};
\end{tikzpicture}
}&=\mathord{
\begin{tikzpicture}[baseline = -.5mm]
	\draw[thick,darkblue,->] (-0.28,-.3) to (0.28,.4);
	\draw[-,line width=4pt,white] (0.28,-.3) to (-0.28,.4);
	\draw[<-,thick,darkblue] (0.28,-.3) to (-0.28,.4);
      \node at (0.14,-0.13) {$\color{darkblue}\scriptstyle\bullet$};
\end{tikzpicture}
}\
\:,
&\qquad
\mathord{
\begin{tikzpicture}[baseline = -.5mm]
	\draw[thick,darkblue,->] (-0.28,-.3) to (0.28,.4);
      \node at (-0.16,-0.15) {$\color{darkblue}\scriptstyle\bullet$};
	\draw[-,line width=4pt,white] (0.28,-.3) to (-0.28,.4);
	\draw[<-,thick,darkblue] (0.28,-.3) to (-0.28,.4);
\end{tikzpicture}
}&= 
\mathord{
\begin{tikzpicture}[baseline = -.5mm]
	\draw[<-,thick,darkblue] (0.28,-.3) to (-0.28,.4);
	\draw[line width=4pt,white,-] (-0.28,-.3) to (0.28,.4);
	\draw[thick,darkblue,->] (-0.28,-.3) to (0.28,.4);
      \node at (0.145,0.23) {$\color{darkblue}\scriptstyle\bullet$};
\end{tikzpicture}
}\:,\label{d3}\\
\mathord{
\begin{tikzpicture}[baseline = -.5mm]
	\draw[<-,thick,darkblue] (0.28,-.3) to (-0.28,.4);
	\draw[line width=4pt,white,-] (-0.28,-.3) to (0.28,.4);
	\draw[thick,darkblue,<-] (-0.28,-.3) to (0.28,.4);
      \node at (-0.15,0.24) {$\color{darkblue}\scriptstyle\bullet$};
\end{tikzpicture}
}&=\mathord{
\begin{tikzpicture}[baseline = -.5mm]
	\draw[thick,darkblue,<-] (-0.28,-.3) to (0.28,.4);
	\draw[-,line width=4pt,white] (0.28,-.3) to (-0.28,.4);
	\draw[<-,thick,darkblue] (0.28,-.3) to (-0.28,.4);
      \node at (0.14,-0.13) {$\color{darkblue}\scriptstyle\bullet$};
\end{tikzpicture}
}\
\:,
&\qquad
\mathord{
\begin{tikzpicture}[baseline = -.5mm]
	\draw[thick,darkblue,<-] (-0.28,-.3) to (0.28,.4);
	\draw[-,line width=4pt,white] (0.28,-.3) to (-0.28,.4);
	\draw[<-,thick,darkblue] (0.28,-.3) to (-0.28,.4);
      \node at (0.155,0.24) {$\color{darkblue}\scriptstyle\bullet$};
\end{tikzpicture}
}&= 
\mathord{
\begin{tikzpicture}[baseline = -.5mm]
	\draw[<-,thick,darkblue] (0.28,-.3) to (-0.28,.4);
	\draw[line width=4pt,white,-] (-0.28,-.3) to (0.28,.4);
	\draw[thick,darkblue,<-] (-0.28,-.3) to (0.28,.4);
      \node at (-0.14,-0.13) {$\color{darkblue}\scriptstyle\bullet$};
\end{tikzpicture}
}\:,\label{d2}\\
\mathord{
\begin{tikzpicture}[baseline = -.5mm]
	\draw[->,thick,darkblue] (0.28,-.3) to (-0.28,.4);
	\draw[line width=4pt,white,-] (-0.28,-.3) to (0.28,.4);
	\draw[thick,darkblue,<-] (-0.28,-.3) to (0.28,.4);
      \node at (0.15,-0.14) {$\color{darkblue}\scriptstyle\bullet$};
\end{tikzpicture}
}&=\mathord{
\begin{tikzpicture}[baseline = -.5mm]
	\draw[thick,darkblue,<-] (-0.28,-.3) to (0.28,.4);
	\draw[-,line width=4pt,white] (0.28,-.3) to (-0.28,.4);
	\draw[->,thick,darkblue] (0.28,-.3) to (-0.28,.4);
      \node at (-0.14,0.23) {$\color{darkblue}\scriptstyle\bullet$};
\end{tikzpicture}
}
\:,
&\qquad
\mathord{
\begin{tikzpicture}[baseline = -.5mm]
	\draw[thick,darkblue,<-] (-0.28,-.3) to (0.28,.4);
	\draw[-,line width=4pt,white] (0.28,-.3) to (-0.28,.4);
	\draw[->,thick,darkblue] (0.28,-.3) to (-0.28,.4);
      \node at (0.155,0.24) {$\color{darkblue}\scriptstyle\bullet$};
\end{tikzpicture}
}&= 
\mathord{
\begin{tikzpicture}[baseline = -.5mm]
	\draw[->,thick,darkblue] (0.28,-.3) to (-0.28,.4);
	\draw[line width=4pt,white,-] (-0.28,-.3) to (0.28,.4);
	\draw[thick,darkblue,<-] (-0.28,-.3) to (0.28,.4);
      \node at (-0.13,-0.12) {$\color{darkblue}\scriptstyle\bullet$};
\end{tikzpicture}
}\:.\label{d1}
\end{align}
\end{lemma}

\begin{proof}
Define 
$\mathord{
\begin{tikzpicture}[baseline = -1mm]
	\draw[<-,thick,darkblue] (0.08,-.25) to (0.08,.25);
      \node at (0.08,0.05) {$\color{darkblue}\scriptstyle\bullet$};
\end{tikzpicture}
}$ to be the left hand expression from the main identity that we are trying
to prove; we still need to show that it is equal to the right hand
expression.
The relations (\ref{d0}) follow from ($\text{A}$) and ($\text{R}$II).
Using the relations from ($\text{R}0$) involving rightward cups and
rightward caps, the relations (\ref{d5}), (\ref{d3}) and (\ref{d2})
are then easy to check too.
Here is the proof of the first equality from (\ref{d1}); the second
equality can be proved similarly:
\begin{align*}
\mathord{
\begin{tikzpicture}[baseline = -.5mm]
	\draw[thick,darkblue,<-] (-0.28,-.3) to (0.28,.4);
	\draw[-,line width=4pt,white] (0.28,-.3) to (-0.28,.4);
	\draw[->,thick,darkblue] (0.28,-.3) to (-0.28,.4);
      \node at (-0.14,0.23) {$\color{darkblue}\scriptstyle\bullet$};
\end{tikzpicture}
}=
\mathord{
\begin{tikzpicture}[baseline = 0]
	\draw[->,thick,darkblue] (0.28,0.4) to[out=90, in=-90] (-.28,1);
	\draw[-,line width=4pt,white] (0.28,1) to[out = -90, in = 90] (-0.28,0.4);
	\draw[-,thick,darkblue] (0.28,1) to[out = -90, in = 90] (-0.28,0.4);
	\draw[-,thick,darkblue] (-0.28,-0.2) to[out=90,in=-90] (0.28,.4);
	\draw[-,line width=4pt,white] (0.28,-0.2) to[out=90,in=-90] (-0.28,.4);
	\draw[-,thick,darkblue] (0.28,-0.2) to[out=90,in=-90] (-0.28,.4);
	\draw[<-,thick,darkblue] (-0.28,-.8) to[out=90,in=-90] (0.28,-0.2);
	\draw[-,line width=4pt,white] (0.28,-.8) to[out=90,in=-90] (-0.28,-0.2);
	\draw[-,thick,darkblue] (0.28,-.8) to[out=90,in=-90] (-0.28,-0.2);
      \node at (-0.28,-0.2) {$\color{darkblue}\scriptstyle\bullet$};
\end{tikzpicture}
}
=
\mathord{
\begin{tikzpicture}[baseline = 0]
	\draw[->,thick,darkblue] (0.28,0.4) to[out=90, in=-90] (-.28,1);
	\draw[-,line width=4pt,white] (0.28,1) to[out = -90, in = 90] (-0.28,0.4);
	\draw[-,thick,darkblue] (0.28,1) to[out = -90, in = 90] (-0.28,0.4);
	\draw[-,thick,darkblue] (0.28,-0.2) to[out=90,in=-90] (-0.28,.4);
	\draw[-,line width=4pt,white] (-0.28,-0.2) to[out=90,in=-90] (0.28,.4);
	\draw[-,thick,darkblue] (-0.28,-0.2) to[out=90,in=-90] (0.28,.4);
	\draw[<-,thick,darkblue] (-0.28,-.8) to[out=90,in=-90] (0.28,-0.2);
	\draw[-,line width=4pt,white] (0.28,-.8) to[out=90,in=-90] (-0.28,-0.2);
	\draw[-,thick,darkblue] (0.28,-.8) to[out=90,in=-90] (-0.28,-0.2);
      \node at (0.28,0.4) {$\color{darkblue}\scriptstyle\bullet$};
\end{tikzpicture}
}
=
\mathord{
\begin{tikzpicture}[baseline = -.5mm]
	\draw[->,thick,darkblue] (0.28,-.3) to (-0.28,.4);
	\draw[line width=4pt,white,-] (-0.28,-.3) to (0.28,.4);
	\draw[thick,darkblue,<-] (-0.28,-.3) to (0.28,.4);
      \node at (0.15,-0.14) {$\color{darkblue}\scriptstyle\bullet$};
\end{tikzpicture}
}\:.
\end{align*}
Then we use these identities plus (\ref{ready}) to check the first equality from (\ref{d4}):
\begin{align*}
\mathord{
\begin{tikzpicture}[baseline = 1mm]
      \node at (0.335,0.15) {$\color{darkblue}\scriptstyle\bullet$};
	\draw[-,thick,darkblue] (0.4,-0.1) to[out=90, in=0] (0.1,0.3);
	\draw[->,thick,darkblue] (0.1,0.3) to[out = 180, in = 90] (-0.2,-0.1);
\end{tikzpicture}
}&=
t\mathord{
\begin{tikzpicture}[baseline = -3mm]
	\draw[-,thick,darkblue] (0.28,-.6) to[out=120,in=-90] (-0.28,0);
	\draw[-,thick,darkblue] (0.28,0) to[out=90, in=0] (0,0.2);
	\draw[-,thick,darkblue] (0,0.2) to[out = 180, in = 90] (-0.28,0);
	\draw[-,line width=4pt,white] (-0.28,-.6) to[out=60,in=-90] (0.28,0);
	\draw[<-,thick,darkblue] (-0.28,-.6) to[out=60,in=-90] (0.28,0);
     \node at (0.2,-0.5) {$\color{darkblue}\scriptstyle\bullet$};
\end{tikzpicture}
}=
t\mathord{
\begin{tikzpicture}[baseline = -3mm]
	\draw[-,thick,darkblue] (0.28,0) to[out=90, in=0] (0,0.2);
	\draw[-,thick,darkblue] (0,0.2) to[out = 180, in = 90] (-0.28,0);
	\draw[<-,thick,darkblue] (-0.28,-.6) to[out=60,in=-90] (0.28,0);
	\draw[-,line width=4pt,white] (0.28,-.6) to[out=120,in=-90] (-0.28,0);
	\draw[-,thick,darkblue] (0.28,-.6) to[out=120,in=-90] (-0.28,0);
     \node at (-0.28,0) {$\color{darkblue}\scriptstyle\bullet$};
\end{tikzpicture}
}=
t\mathord{
\begin{tikzpicture}[baseline = -3mm]
	\draw[-,thick,darkblue] (0.28,0) to[out=90, in=0] (0,0.2);
	\draw[-,thick,darkblue] (0,0.2) to[out = 180, in = 90] (-0.28,0);
	\draw[<-,thick,darkblue] (-0.28,-.6) to[out=60,in=-90] (0.28,0);
	\draw[-,line width=4pt,white] (0.28,-.6) to[out=120,in=-90] (-0.28,0);
	\draw[-,thick,darkblue] (0.28,-.6) to[out=120,in=-90] (-0.28,0);
     \node at (0.28,0) {$\color{darkblue}\scriptstyle\bullet$};
\end{tikzpicture}
}=
t\mathord{
\begin{tikzpicture}[baseline = -3mm]
	\draw[-,thick,darkblue] (0.28,0) to[out=90, in=0] (0,0.2);
	\draw[-,thick,darkblue] (0,0.2) to[out = 180, in = 90] (-0.28,0);
	\draw[-,thick,darkblue] (0.28,-.6) to[out=120,in=-90] (-0.28,0);
	\draw[<-,line width=4pt,white] (-0.28,-.6) to[out=60,in=-90] (0.28,0);
	\draw[<-,thick,darkblue] (-0.28,-.6) to[out=60,in=-90] (0.28,0);
     \node at (-0.15,-.45) {$\color{darkblue}\scriptstyle\bullet$};
\end{tikzpicture}
}=
\mathord{
\begin{tikzpicture}[baseline = 1mm]
	\draw[-,thick,darkblue] (0.4,-0.1) to[out=90, in=0] (0.1,0.3);
	\draw[->,thick,darkblue] (0.1,0.3) to[out = 180, in = 90] (-0.2,-0.1);
      \node at (-0.14,0.15) {$\color{darkblue}\scriptstyle\bullet$};
\end{tikzpicture}
}
\:.
\end{align*}
The remaining equality from (\ref{d4}), and the equality of $
\mathord{
\begin{tikzpicture}[baseline = -1mm]
	\draw[<-,thick,darkblue] (0.08,-.25) to (0.08,.25);
      \node at (0.08,0.05) {$\color{darkblue}\scriptstyle\bullet$};
\end{tikzpicture}
}$
with the
second expression from the main identity, now follow
easily using ($\text{R}0$) once more.
\end{proof}

There is an obvious monoidal functor 
\begin{equation*}
\alpha:\OS(z,t) \rightarrow
\AOS(z,t)
\end{equation*}
taking objects and morphisms to the same things in $\AOS(z,t)$.
The following lemma shows that $\alpha$ is
faithful (and also that the functor $\beta$ from the lemma is full).

\begin{lemma}\label{notmon}
There is a unique $\k$-linear (but not monoidal!) functor $$\beta:\AOS(z,t) \rightarrow
\OS(z,t)$$
such that $\beta \circ \alpha =
\operatorname{Id}_{\OS(z,t)}$ and $\beta \left(1_\a \otimes \mathord{
\begin{tikzpicture}[baseline = -1mm]
      \node at (0,0) {$\color{darkblue}\scriptstyle\bullet$};
	\draw[<-,thick,darkblue] (0,0.25) to (0,-0.25);
\end{tikzpicture}
}\right) = 1_\a
\otimes \up$ for all $\a \in\words$.
The effect of $\beta$ on dots on other strands
is given by
\begin{align}
\mathord{
\begin{tikzpicture}[baseline =-1.5mm]
	\draw[-,thick,double,darkblue] (-.3,-.5) to (-.3,.5);
	\draw[-,thick,double,darkblue] (.3,-.5) to (.3,.5);
	\draw[->,thick,darkblue] (0,-.5) to (0,.5);
      \node at (0,0) {$\color{darkblue}\scriptstyle\bullet$};
      \node at (-.3,-.68) {$\a$};
      \node at (.3,-.65) {$\b$};
\end{tikzpicture}
}
&\mapsto
\mathord{
\begin{tikzpicture}[baseline =-1.5mm]
	\draw[-,thick,double,darkblue] (-.3,-.5) to (-.3,.5);
	\draw[->,thick,darkblue] (.6,0) to[out=90,in=-90] (0,.5);
	\draw[-,line width=4pt,white] (.3,-.5) to (.3,.5);
	\draw[-,thick,double,darkblue] (.3,-.5) to (.3,.5);
	\draw[-,line width=4pt,white] (0,-.5) to[out=90,in=-90] (.6,0);
	\draw[-,thick,darkblue] (0,-.5) to[out=90,in=-90] (.6,0);
      \node at (-.3,-.68) {$\a$};
      \node at (.3,-.65) {$\b$};
\end{tikzpicture}
}
\:,&
\mathord{
\begin{tikzpicture}[baseline =-1.5mm]
	\draw[-,thick,double,darkblue] (-.3,-.5) to (-.3,.5);
	\draw[-,thick,double,darkblue] (.3,-.5) to (.3,.5);
	\draw[<-,thick,darkblue] (0,-.5) to (0,.5);
      \node at (0,0) {$\color{darkblue}\scriptstyle\bullet$};
      \node at (-.3,-.68) {$\a$};
      \node at (.3,-.65) {$\b$};
\end{tikzpicture}
}
&\mapsto
t^{-2}\mathord{
\begin{tikzpicture}[baseline =-1.5mm]
	\draw[-,thick,double,darkblue] (-.3,-.5) to (-.3,.5);
	\draw[<-,thick,darkblue] (0,-.5) to[out=90,in=-90] (.6,0);
	\draw[-,line width=4pt,white] (.3,-.5) to (.3,.5);
	\draw[-,thick,double,darkblue] (.3,-.5) to (.3,.5);
	\draw[-,line width=4pt,white] (.6,0) to[out=90,in=-90] (0,.5);
	\draw[-,thick,darkblue] (.6,0) to[out=90,in=-90] (0,.5);
      \node at (-.3,-.68) {$\a$};
      \node at (.3,-.65) {$\b$};
\end{tikzpicture}
}\:,\label{arun}
\end{align}
for any $\a,\b \in \words$.
\end{lemma}

\begin{proof}
We already have a presentation for $\AOS(z,t)$ as a $\k$-linear
monoidal category,
 with
generators and relations
coming from Theorem~\ref{first} plus the additional generator
$O := \mathord{
\begin{tikzpicture}[baseline = -1mm]
      \node at (0,0) {$\color{darkblue}\scriptstyle\bullet$};
	\draw[<-,thick,darkblue] (0,0.25) to (0,-0.25);
\end{tikzpicture}
}$ and its two-sided inverse $O^{-1}$ subject to $(\text{A})$.
Since we are trying to construct a $\k$-linear functor that is not
monoidal, we need to convert this into a presentation for
$\AOS(z,t)$ as
a $\k$-linear category, as explained in \cite[$\S$2.6]{BCNR}.
The generators are all morphisms of the form $1_\a
\otimes S \otimes 1_\b, 1_\a \otimes T \otimes 1_\b, 1_\a \otimes C
\otimes 1_\b$, $1_\a \otimes D \otimes 1_\b$ and $1_\a \otimes O \otimes
1_\b$
for all $\a, \b \in \words$.
The relations are derived from the monoidal relations by tensoring
with $1_\a$ and $1_\b$ in a similar way, plus we also need {\em
  commuting relations} replacing the interchange law.

Using this new presentation, we can establish the existence of
$\beta$: 
it is the identity on objects, and must send the generating morphisms
 $1_\a
\otimes S \otimes 1_\b, 1_\a \otimes T \otimes 1_\b, 1_\a \otimes C
\otimes 1_\b$ and $1_\a \otimes D \otimes 1_\b$ to the same morphisms in
$\OS(z,t)$ since we want $\beta\circ\alpha = \operatorname{Id_{\OS(z,t)}}$.
It sends $1_\a \otimes O \otimes 1_\b$ and its two-sided inverse
$1_\a \otimes O^{-1} \otimes 1_\b$ to
$$
\mathord{

},
$$
respectively.
Again, there is no choice here, since in $\AOS(z,t)$ we have that
$$
\mathord{
%
}\:.
$$
Now we check the relations.
All the ones that do not involve $O$ hold automatically.
For the rest, the following checks
($\text{A}$):
$$
\mathord{
%
}\:.
$$
The commuting relations involving $O$ and any other generating
morphism are equally straightforward.

So now we have constructed $\beta$ and proved its uniqueness. The composition $\beta \circ
\alpha$ is $\operatorname{Id}_{\OS(z,t)}$ since it is the identity on
objects and generating morphisms.
It remains to see that the second equality in (\ref{arun}) holds:
\begin{align*}
\mathord{
%
}\:.\end{align*}
\end{proof}

\begin{remark}
The kernel of $\beta$
is the left tensor ideal of $\AOS(z,t)$ generated by the morphism
$\mathord{
%
}$\:.
This follows because the quotient of $\AOS(z,t)$ by this left tensor
ideal is a $\k$-linear category which maps surjectively onto
$\OS(z,t)$
via the functor induced by $\beta$, and also $\OS(z,t)$ maps
surjectively onto it via the functor induced by $\alpha$.
This identifies $\OS(z,t)$ with the simplest ``level one'' example of a
{\em cyclotomic quotient} of $\AOS(z,t)$.
\end{remark}

The images of the morphisms $1_\a \otimes 
\mathord{
%
}\otimes 1_\b$ under $\beta$
are the
{\em Jucys-Murphy elements} of $\OS(z,t)$.
The elements (\ref{salsa}) in the introduction are a special case.
In the sequel, we will often need these elements in the 
special case that $\a = \varnothing$, adopting the following notation:
\begin{align}
X(\up \b) &:=
\beta\left(\mathord{
%
}\otimes 1_\b\right).\label{JM}
\end{align}
We can use (\ref{arun}) to obtain a recursive formula:
first $X(\up) := 1_\up$,
$X(\down) := t^{-2} 1_\down$;
then for any $\b \in \words$ 
we have that
\begin{align}\label{JMa}
X(\up \up \b) &:= 
\mathord{
%
}\:.\label{JMb}
\end{align}

The following lemma will not be needed in this article. It suggests 
$\End_{\AOS(z,t)}(\varnothing)$ consists of two copies of the ring of symmetric functions.
Define the positive and negative bubbles with $n$ dots by
\begin{align}
\mathord{
%
}$
due to the relation (D).

\begin{lemma}
  For any $n \in \ZZ$, we have that
\begin{align*}
\sum_{\substack{r,s \geq 0\\r+s = n}}
\mathord{
%
}&= -\delta_{n,0} z^{-2}.
\end{align*}
\end{lemma}

\begin{proof}
Consider the relation involving positive bubbles. It is immediate in
case $n\leq 0$.
If $n > 0$, we need the following relation which may be proved by
induction using the relations (A) and (S):
\begin{equation}
\mathord{
%
}\:.
\end{equation}
Then we calculate:
\begin{align*}
-
\mathord{
%
}.
\end{align*}
The relation involving for negative bubbles may now be deduced by applying
the automorphism $\rho:\AOS(z,t) \rightarrow \AOS(z,t)$ which is
defined on generators in the same way as (\ref{rho}) plus it maps
$\mathord{
%
}$.
\end{proof}

The category $\AOS(z,t)$ is studied further in \cite{B2}: it
is the $k=0$ case of the $q$-Heisenberg category $\H_k(z,t)$
introduced there.

\section{Shortest word theory}\label{sswt}

For the next few sections, we assume that $\k$ is a field of
characteristic $p \geq 0$ and $z = q-q^{-1}$ for $q \in
\k^\times\setminus\{\pm 1\}$.
Since we are going to be discussing linear representations rather than tensor
ones, it is natural to replace the category $\OS(z,t)$ with the associative
algebra
$$
OS = \bigoplus_{\a,\b \in \words} 1_\a OS 1_\b
\qquad\text{where}\qquad
1_\a OS 1_\b := \Hom_{\OS(z,t)}(\b,\a),
$$
multiplication being induced by composition in $\OS(z,t)$.
This algebra is {\em locally unital} rather than unital,
with a distinguished system of local units given by the mutually
orthogonal idempotents 
$\{1_\a\:|\:\a \in \words\}$.
The functor 
category $\Mod \OS(z,t)$ of $\OS(z,t)$-modules as defined in the introduction
may be identified with the usual algebraic category $\Mod OS$ of all {\em right
  $OS$-modules} $M$ which are locally unital in the sense that
$M = \bigoplus_{\a \in \words} M 1_\a$. 
The additive Karoubi envelope $\dot\OS(z,t)$ is
equivalent to the full subcategory $\pMod OS$ of $\Mod OS$ consisting of {\em finitely generated
  projective modules},
i.e., modules isomorphic to finite direct sums of right ideals
$e OS$ for idempotents $e \in OS$.
This means that we may identify $K_0(\dot\OS(z,t))$ with the split
Grothendieck group $K_0(\pMod OS)$;
the resulting ring structure on $K_0(\pMod OS)$ is determined by
$$
[e OS] [f OS] = [(e\otimes f) OS]
$$
for idempotents $e,f \in OS$.

Each $1_\a OS 1_\b$ is finite-dimensional by Theorem~\ref{bt}, hence,
$OS$ is {\em locally finite-dimensional}, and $\Mod OS$ is a
{\em locally Schurian category} in the sense of \cite[$\S$2]{BD}.
We will freely use the general language and results about such categories explained
there, most of which boil down to the observation that 
$\Mod OS$ is a Grothendieck category.
In particular, we let $\lfdMod OS$ be the subcategory of $\Mod OS$ consisting of all
{\em locally finite-dimensional modules}, i.e., the $OS$-modules $M$
such that 
$\dim_\k M 1_\a < \infty$ for all $\a \in \words$.
These are exactly the modules that have finite composition
multiplicities.
Any finitely generated
$OS$-module $M$ is locally finite-dimensional, so that $\pMod OS$ is a
subcategory of $\lfdMod OS$.

In this section, we are going to classify the irreducible $OS$-modules.
The key to our approach is that the algebra $OS$ has a {\em
  triangular decomposition}.
Any ribbon has three sorts of component: 
\begin{itemize}
\item cups
whose boundary is on $y=1$;
\item
caps whose boundary is on $y=0$;
\item
propagating strands whose boundary intersects both $y=0$ and $y=1$.
\end{itemize}
The cups and caps carry an overall orientation that is either leftward or rightward, while the
propagating strands are either upward strands or downward strands.
Introduce the following locally unital subalgebras of $OS$:
\begin{itemize}
\item[$OS^\circ_{r,s}$:] 
The $\k$-span of all ribbons that have 
$r$ propagating upward strands and $s$ propagating downward strands,
no components that are cups or caps,
and in 
which propagating upward strands only cross underneath propagating
downward strands.
\item[$OS^\circ$:] $\bigoplus_{r,s \geq 0} OS^\circ_{r,s}$.
\item[$OS^+$:] The $\k$-span of all ribbons that have no components that are caps and in which no
  two propagating strands cross.
\item[$OS^\sharp$:] The $\k$-span of all ribbons with no components
  that are caps 
and in 
which propagating upward strands only cross underneath propagating
downward strands.
\item[$OS^-$:] The $\k$-span of all ribbons that have no components that are cups and in which no
  two propagating strands cross.
\item[$OS^\flat$:] The $\k$-span of all ribbons with no components
  that are cups 
and in 
which propagating upward strands only cross underneath propagating
downward strands.
\end{itemize}
It is obvious that these are all locally unital subalgebras of $OS$. Also, $OS$ is
graded as $OS = \bigoplus_{d \in \ZZ} OS[d]$ with $OS[d]$ spanned by
all ribbons such that the total number of cups minus the total number
of caps equals $d$. This induces a grading on each of the subalgebras
above. Moreover, $OS^\sharp[0] = OS^\flat[0] = OS^\circ$.
We have already introduced the right $OS$-module categories $\Mod
OS, \lfdMod OS$ and $\pMod OS$. We adopt analogous notation for all of
these other locally unital algebras.
Also let $$\fdMod OS^\circ = \coprod_{r,s\geq 0} \fdMod OS_{r,s}^\circ
$$ 
be the category of (globally) finite-dimensional right
$OS^\circ$-modules.

In the following lemma, we take tensor products of locally unital modules over the
locally unital algebra $\K
:= \bigoplus_{\a \in \words} \k 1_\a < OS$. In terms of the usual tensor
product $\otimes$ over the ground field $\k$, we have that
$V \otimes_\K W = \bigoplus_{\a \in \words} V 1_\a \otimes 1_\a W$.

\begin{lemma}\label{td}
Multiplication define a  vector space isomorphism
\begin{align*}
OS^+ \otimes_\K OS^\circ \otimes_\K OS^- &\stackrel{\sim}{\rightarrow}
OS.
\end{align*}
Similarly, there are isomorphisms
$OS^+ \otimes_\K OS^\circ\stackrel{\sim}{\rightarrow} OS^\sharp$ and 
$OS^\circ \otimes_\K OS^- \stackrel{\sim}{\rightarrow} OS^\flat$.
\end{lemma}

\begin{proof}
We apply Theorem~\ref{bt} to pick a basis for $OS^+$
consisting of cap-free generic reduced lifts of matchings.
Similarly, we pick a basis for $OS^-$.
Finally, we pick a basis for $OS^\circ$ consisting of cup- and cap-free generic ribbons, all
of whose rightward crossings are positive and leftward crossings are
negative.
To prove the lemma, it remains to observe that
the non-zero images of the pure tensors in these basis elements under
the multiplication
$OS^+ \otimes_\K OS^\circ \otimes_\K OS^-\rightarrow OS$ give a basis for $OS$
itself: it is just another of the bases from Theorem~\ref{bt}
consisting
of generic reduced lifts with all cups at the top, all crossings
of propagating strands in the middle, and all caps at
the bottom of the picture.
\end{proof}

Recall the Iwahori-Hecke algebras $H_r$ 
from (\ref{hecke}) and the homomorphism $\i_r$ from (\ref{iotap1}).
Theorem~\ref{bt} shows that this is an isomorphism.
More generally, let
\begin{equation}\label{belt}
H_{r,s} := H_r \otimes H_s
\end{equation}
for any $r,s \geq 0$. Then there is an injective (but no longer 
surjective) homomorphism
\begin{align}\label{iotap2}
\i_{r,s}
:
H_{r,s}
&\rightarrow 1_{\down^s\up^r} OS 1_{\down^s \up^r}
\end{align}
sending $S_i \otimes 1$ to the positive crossing 
$\begin{tikzpicture}[baseline = -.5mm]
	\draw[->,thick,darkblue] (0.2,-.2) to (-0.2,.3);
	\draw[line width=4pt,white,-] (-0.2,-.2) to (0.2,.3);
	\draw[thick,darkblue,->] (-0.2,-.2) to (0.2,.3);
\end{tikzpicture}
$ of the $i$th and $(i+1)$th strands
and $1 \otimes S_j$ to the positive crossing
$\begin{tikzpicture}[baseline = -.5mm]
	\draw[<-,thick,darkblue] (0.2,-.2) to (-0.2,.3);
	\draw[line width=4pt,white,-] (-0.2,-.2) to (0.2,.3);
	\draw[thick,darkblue,<-] (-0.2,-.2) to (0.2,.3);
\end{tikzpicture}
$ 
of the $(r+j)$th and $(r+j+1)$th
strands, again numbering strands from right to left.

\begin{lemma}\label{cartanlem}
The map $\i_{r,s}$ is an algebra isomorphism
$H_{r,s} \stackrel{\sim}{\rightarrow} 1_{\down^s \up^r} OS_{r,s}^\circ
1_{\down^s \up^r}.$
Moreover, $OS_{r,s}^\circ$ is isomorphic to the matrix algebra
$\Mat_{\binom{r+s}{r}}(H_{r,s})$.
Hence,
the functor $$
\Upsilon_{r,s} := 
? \otimes_{H_{r,s}} 1_{\down^s \up^r}
OS^\circ_{r,s}:\Mod{H_{r,s}} \rightarrow \Mod{OS^\circ_{r,s}}.
$$
is an equivalence of categories,
with quasi-inverse defined by right multiplication by the idempotent
$1_{\down^s \up^r}$.
\end{lemma}

\begin{proof}
The first statement follows from Theorem~\ref{bt}.
To deduce the second statement, let $\words_{r,s}$ denote the
$\binom{r+s}{r}$ different words which have $r$ letters $\up$ and $s$
letters $\down$. Note that $OS^\circ_{r,s} = \bigoplus_{\a,\b \in
  \words_{r,s}} 1_\a OS^\circ_{r,s} 1_\b$.
For each $\a,\b \in \words_{r,s}$, let $e_{\a,\b} \in 1_\a OS_{r,s}^\circ 1_\b$
be the unique (up to planar isotopy) reduced $(\b,\a)$-ribbon which only
involves positive rightward crossings and negative leftward crossings, and has no
cups, caps or upward/downward crossings. For example, $e_{\a,\a} = 1_\a$ for
each $\a$.
We have that $e_{\a,\b} e_{\b',\c} = \delta_{\b,\b'} e_{\a,\c}$.
Hence, the map
$$
H_{r,s} \rightarrow
1_\a OS^\circ_{r,s} 1_\b, \qquad
f \mapsto e_{\a,{\down^s \up^r}} \i_{r,s}(f) e_{{\down^s \up^r},\b}
$$
is a bijection, and
$OS^\circ_{r,s} = \bigoplus_{\a, \b \in \words_{r,s}} H_{r,s} e_{\a,\b}$ is the matrix algebra as claimed.
\end{proof}

Next we are going to mimic the usual arguments of highest weight theory for
semisimple Lie algebras, with the role of ``Borel subalgebra'' being
played by $OS^\sharp$, and the role of ``Cartan subalgebra'' being
played by $OS^\circ$.
To parametrize the isomorphism classes of
irreducible representations of $OS^\circ$, we need some
facts about the representation
theory of the Iwahori-Hecke algebra $H_r$; e.g. see \cite{DJ}. 

The algebra $H_r$ is semisimple
if $q$ is not a root of unity, with irreducible representations 
being the Specht modules $\SS_\lambda$
parametrized by partitions $\lambda \vdash r$.
To construct $\SS_\lambda$ explicitly, recall the elements $x_\lambda$
and $y_\lambda$
from (\ref{symmetrizers}).
The right ideals $x_\lambda H_r$ and $y_\lambda H_r$ are the {\em permutation module}
$M_\lambda$ and the {\em signed permutation module} $N_\lambda$,
respectively.
By \cite[Theorem 3.3]{DJ}, the space $\Hom_{H_r}(M_\lambda,
N_{\lambda^\trans})$ is one-dimensional. Then the {\em Specht module}
$\SS_\lambda$ is the image of any non-zero homomorphism in this space.

The definition just given also makes sense when $q$ is a root of unity (remembering
$q^2 \neq 1$); the resulting Specht module $\SS_\lambda$ 
is related to the module $\SS^\lambda$ constructed in \cite{DJ} by $\SS_\lambda \cong (\SS^{\lambda^\trans})^\#$.
In general, we define $e$ to be the smallest positive integer such
that $q^{2e}=1$,
setting $e := 0$ if $q$ is not a root of unity.
A partition $\lambda$ is {\em $e$-restricted} if either $e=0$, or $e >
0$ and $\lambda_i - \lambda_{i+1} < e$ for each $i=1,2,\dots$.
For $e$-restricted $\lambda\vdash r$, the Specht module $\SS_\lambda$ has
irreducible head $\D_\lambda$, and these modules for all $e$-restricted
$\lambda\vdash r$
give a complete set of pairwise inequivalent irreducible right
$H_r$-modules.

Let $\Y_\lambda$ be a projective cover of $\D_\lambda$; since $H_r$
is a symmetric algebra, this is
also an injective hull.
If $e=0$ we have simply that
$\D_\lambda=\SS_\lambda=\Y_\lambda$, and they are all equal to the right ideal
$e_\lambda H_r$ where $e_\lambda$ is the Young symmetrizer from
(\ref{youngsymmetrizer}).
In general, it is known
that $\Y_\lambda$ has a finite
filtration\footnote{This is proved by applying the ``Schur
functor'' to the analogous result for the
$q$-Schur algebra.} with sections $\SS_\mu$, each appearing
$[\SS_\mu:\D_\lambda]$ times.
Consequently,
\begin{equation*}
[\Y_\lambda] = \sum_{\mu \vdash r} [\SS_\mu:\D_\lambda] [\SS_\mu]
\end{equation*}
in the Grothendieck group $K_0(\fdMod H_r)$ of the Abelian category
$\fdMod H_r$.
We refer to this decomposition as {\em Brauer reciprocity};
it may also be proved by lifting idempotents.

Now let $\Par_{r,s} := \{\LA = (\lambda^\up,\lambda^\down)\:|\:\lambda\vdash r, \mu \vdash
s\}$ be the set of {\em bipartitions} of $(r,s)$,
and let $\RPar_{r,s} \subseteq \Par_{r,s}$ be the $e$-restricted ones.
In particular, we denote the empty bipartition
$(\varnothing,\varnothing)$ by $\NOTHING$.
From the previous paragraph, it follows that $\RPar_{r,s}$ parametrizes
the irreducible $H_{r,s}$-modules.
Applying Lemma~\ref{cartanlem}, we get from them irreducible $OS^\circ$-modules
\begin{equation}
\D(\LA):=
(\D_{\lambda^\up} \boxtimes \D_{\la^\down}) \otimes_{H_{r,s}} 1_{\down^s \up^r} OS^\circ_{r,s}.
\end{equation}
Define 
$\SS(\LA)$ 
for $\LA\in\Par_{r,s}$ and
$\Y(\LA)$ for $\LA\in\RPar_{r,s}$
in similar ways,
starting from $\SS_{\lambda^\up} \boxtimes \SS_{\la^\down}$ or $\Y_{\lambda^\up}\boxtimes
\Y_{\la^\down}$, respectively.
For $\LA\in\RPar_{r,s}$, $\Y(\LA)$ 
is a projective cover and an injective hull of $\D(\LA)$,
and it has a filtration with sections $\SS(\MU)$ for $\MU
\in \Par_{r,s}$, each appearing
$[\SS(\MU):\D(\LA)] = [\SS_{\mu^\up}:\D_{\la^\up}] [\SS_{\mu^\down}: \D_{\la^\down}]$
in the filtration.
Finally, we put these modules all together: setting
$\Par := \bigsqcup_{r,s \geq 0} \Par_{r,s}$ and
$\RPar := \bigsqcup_{r,s \geq 0} \RPar_{r,s}$,
the modules
$\{\D(\LA)\:|\:\LA \in \RPar\}$
are a complete set of pairwise inequivalent irreducible
$OS^\circ$-modules.

The projection of $OS^\sharp$ onto its degree zero component
$OS^\sharp[0]$ is a surjective homomorphism $OS^\sharp
\twoheadrightarrow OS^\circ$. Using this, we can view any $OS^\circ$-module
instead as an $OS^\sharp$-module. We denote the resulting
inflation
functor by $\infl^{\sharp}$; similarly, there is an inflation
functor
$\infl^{\flat}$ taking $OS^\circ$-modules to $OS^\flat$-modules.
Define
\begin{align}\label{eclipse1}
\bar\Delta(\LA) &:= \infl^{\sharp} \D(\LA) \otimes_{OS^\sharp} OS,
\\
\tilde\Delta(\LA) &:= \infl^\sharp \SS(\LA) \otimes_{OS^\sharp} OS,
\\
\Delta(\LA) &:= \infl^\sharp \Y(\LA) \otimes_{OS^\sharp} OS,\label{eclipse3}
\end{align}
for $\la \in \RPar, \Par$ and $\RPar$, respectively.
We call $\bar\Delta(\LA)$ the {\em proper standard module}
and $\Delta(\LA)$ the {\em standard module}
associated to $\LA \in \RPar$.
These are locally finite-dimensional but not ``globally'' finite-dimensional $OS$-modules.
Note also if $e=0$ that
$\bar\Delta(\LA)=\tilde\Delta(\LA)=\Delta(\LA)$
for each $\LA \in \Par$.

Any $OS$-module $M$ decomposes as $\bigoplus_{\a \in \words} M 1_\a$.
We refer to the direct sum of the subspaces $M 1_\a$ taken over all $\a \in \words$ of
minimal length such that $M 1_\a \neq 0$
as the {\em shortest word space} of $M$.
It is automatically an $OS^\sharp$-submodule of $M$ on which
$\bigoplus_{d > 0} OS^\sharp[d]$ acts as zero. We say that $M$ is a
{\em shortest word module} of type $\LA \in \RPar$ if $M$ is
generated as an $OS$-module by its shortest word space, and this space
is isomorphic to $\D(\LA)$ as an $OS^\circ$-module.
Then, since $\D(\LA)$ is irreducible, $M$ is actually
generated by any non-zero vector in its shortest word
space. 

\begin{theorem}\label{class}
For $\LA\in\RPar$, the proper standard module $\bar\Delta(\LA)$ has a unique
maximal submodule $\rad\,\bar\Delta(\LA)$.
Setting
$\L(\LA) := \bar\Delta(\LA) / \rad\,
\bar\Delta(\LA)$,
we obtain a complete set of pairwise inequivalent irreducible
$OS$-modules
$\{\L(\LA)\:|\:\LA\in\RPar\}$.
Moreover, each $\L(\LA)$ is absolutely irreducible.
\end{theorem}

\begin{proof}
By Lemma~\ref{td}, we have that
$\bar\Delta(\LA) = \D(\LA)
\otimes_{OS^\sharp}  OS 
=
\D(\LA) \otimes_\K OS^-$
as a right $\K$-module. Hence, it is a non-zero shortest word module
of type $\LA$.
Since any non-zero vector in $\D(\LA) \otimes OS^-[0]$ generates all of $\bar\Delta(\LA)$,
any proper submodule of $\bar\Delta(\LA)$ must be a subspace of 
$\bigoplus_{d < 0} \D(\LA) \otimes_\K OS^-[d]$.
This implies that the sum of all proper submodules
of $\bar\Delta(\LA)$ is itself proper, hence, it is the unique maximal submodule
$\rad\, \bar\Delta(\LA)$ of $\bar\Delta(\LA)$. 
Thus, the quotient modules $\L(\LA) := \bar\Delta(\LA) /
\rad\,\Delta(\LA)$
are irreducible.

Now let $L$ be any irreducible $OS$-module.
Pick an irreducible $OS^\circ$-submodule $L^\circ \cong \D(\LA)$ of its
shortest word space. Then by
Frobenius reciprocity we get an $OS$-module homomorphism
$\bar\Delta(\LA) \rightarrow L$
lifting an isomorphism $\D(\LA)\stackrel{\sim}{\rightarrow} L^\circ$. 
This map is surjective since $L$ is irreducible, hence, we get that $L \cong \L(\LA)$.

Finally, if $\L(\LA)\cong \L(\MU)$,
then their shortest word spaces must be isomorphic as $OS^\circ$-modules,
so $\LA=\MU$. Also, since $\End_{OS}(\L(\LA)) \cong \End_{OS^\circ}(\D(\LA))$, the absolute
irreducibility follows from the analogous assertion for Hecke algebras,
which is well known.
\end{proof}

\begin{example}\label{swex}
To illustrate ``shortest word theory,''
we use it to prove the existence of a non-zero homomorphism
$\bar\Delta(\MU)\rightarrow \bar\Delta(\LA)$
when $t = q^n$ for $n \in \NN$,
 $\LA := ((1^n), \varnothing)$ and
$\MU = ((1^{n+1}), (1))$.
The irreducible $OS^\circ_{n,0}$-module $\D(\LA)$
comes from the ``sign representation'' of $H_n$. So it 
is one-dimensional and is spanned by a vector
on which $1_{\up^n}$ acts as the identity and any
positive upward crossing acts as $-q^{-1}$.
Let $v$ be the generator of
$\bar\Delta(\LA) = \D(\LA)
\otimes_{OS^\sharp} OS$
defined by this vector tensored $1_{\up^n} \in OS$.
Also for $i,j=0,\dots,n$ define
$$
a_i := 
\mathord{
\begin{tikzpicture}[baseline =0mm]
	\draw[->,ultra thick,darkblue] (.7,-.3) to (.7,.5);
	\draw[->,ultra thick,darkblue] (.1,-.3) to (.1,.5);
        \draw[-,line width=4pt, white] (-.2,-.3) to [out=90,in=180] (.1,0);
        \draw[-,line width=4pt, white] (.1,0) to [out=0,in=90] (.4,-.3);
        \draw[<-,thick, darkblue] (-.2,-.3) to [out=90,in=180] (.1,0);
        \draw[-,thick,darkblue] (.1,0) to [out=0,in=90] (.4,-.3);
      \node at (.75,-.45) {$\scriptstyle n-i$};
      \node at (.1,-.45) {$\scriptstyle i$};
\end{tikzpicture}
},
\qquad\qquad
b_j := 
\mathord{
\begin{tikzpicture}[baseline =0mm]
	\draw[->,ultra thick,darkblue] (.7,-.3) to (.7,.5);
	\draw[->,ultra thick,darkblue] (.1,-.3) to (.1,.5);
        \draw[-,line width=4pt, white] (-.2,.5) to [out=-90,in=180] (.1,.1);
        \draw[-,line width=4pt, white] (.1,.1) to [out=0,in=-90] (.4,.5);
        \draw[-,thick, darkblue] (-.2,.5) to [out=-90,in=180] (.1,.1);
        \draw[->,thick,darkblue] (.1,.1) to [out=0,in=-90] (.4,.5);
      \node at (.75,-.45) {$\scriptstyle n-j$};
      \node at (.1,-.45) {$\scriptstyle j$};
\end{tikzpicture}
}\:.
$$
A calculation with relations shows that 
$$
v a_ib_j = \left\{\begin{array}{ll}
[n]_q v&\text{if $i=j$,}\\
q^n (-q)^{i-j+1}v&\text{if $i < j$,}\\
q^{-n} (-q)^{i-j-1}v&\text{if $i > j$.}
\end{array}\right.
$$
Now consider the vector 
$$
w := \sum_{i=0}^n (-q)^i v a_i \in \bar\Delta(\LA).
$$
This is non-zero by Lemma~\ref{td}.
We claim:
\begin{itemize}
\item[(1)]
$w b_j = 0$ for each $j=0,1,\dots,n$;
\item[(2)]
$w \i_{n+1,1}(S_i) = -q^{-1} w$ for $i=1,\dots,n$.
\end{itemize}
To see (1), we have that
$$
w b_j = \left[(-q)^j [n]_q + \sum_{i=0}^{j-1} q^n (-q)^{i-j+1} (-q)^i
+ \sum_{i=j+1}^n q^{-n} (-q)^{i-j-1} (-q)^i\right] v,
$$
which is zero since $[n]_q = \sum_{i=0}^{j-1} q^{-n+2i-2j+1} + \sum_{i=j+1}^n
q^{-n+2i-2j-1}$.
We leave the check of (2) to the reader.
It means that $w$ spans a one-dimensional $H_{n+1,1}$-submodule
of $\bar\Delta(\LA)$ isomorphic to its ``sign representation,''
so the $OS^\circ$-submodule generated by $w$ is a copy of $\D(\MU)$.
In view of (1), it is actually an $OS^\sharp$-submodule isomorphic to
$\infl^\sharp \D(\MU)$.
Finally,
by Frobenius reciprocity, we get
a non-zero $OS$-module homomorphism
$\bar\Delta(\MU)\rightarrow \bar\Delta(\LA)$.
\end{example}

The various flavors of standard module introduced in
(\ref{eclipse1})--(\ref{eclipse3})
are obtained by applying the {\em standardization functor}
\begin{equation}\label{sta}
\Delta := (\infl^\sharp -)\otimes_{OS^\sharp} OS:\Mod{OS^\circ} \rightarrow
\Mod{OS}
\end{equation}
to the $OS^\circ$-modules $\D(\LA), \SS(\LA)$ and $\Y(\LA)$.
By Lemma~\ref{td}, the composition of this functor followed by the forgetful
functor to vector spaces is isomorphic to $- \otimes_\K OS^-$. Hence, $\Delta$
is exact.
There is also the {\em costandardization functor}
\begin{equation}\label{costa}
\nabla:= \bigoplus_{\a \in \words} 
\Hom_{OS^\flat}(1_\a OS,
\infl^\flat -):
\Mod{OS^\circ}\rightarrow \Mod{OS},
\end{equation}
where the action of $a \in 1_\a OS 1_\b$ on $f \in
\Hom_{OS^\flat}(1_{\a'} OS, \infl^\flat M)$ is 
zero unless $\a=\a'$, in which case, 
it is the element of 
$\Hom_{OS^\flat}(1_\b OS, \infl^\flat M)$
defined from 
$(fa)(b) :=f(ab)$.
This functor is exact since 
$$
1_\a OS \cong 1_\a OS^+ \otimes_\K OS^\flat
\cong \bigoplus_{\b \in \words} \dim_\k(1_\a OS^+ 1_\b) 1_\b OS^\flat
$$ as a right
$OS^\flat$-module, which is finitely generated and projective.
We refer to the modules
\begin{equation}
\bar\nabla(\LA) := 
\nabla \D(\LA),\qquad
\nabla(\LA) := \nabla \Y(\LA)
\end{equation}
as the {\em proper costandard} and {\em costandard modules}, respectively.

There is a well-known duality functor $\circledast$ on finite-dimensional modules over the Hecke
algebra with $\D_\lambda^\circledast \cong \D_\lambda$, hence, $\Y_\lambda^\circledast \cong \Y_\lambda$. 
The corresponding duality $\circledast$ on $\fdMod OS^\circ$
takes a right module to its linear dual with the natural left action
twisted into a right action using the
antiautomorphism arising from the restriction of the isomorphism $\tau$ from (\ref{tau}).
In an entirely analogous way, $\tau$ gives rise to a duality, also
denoted $\circledast$, on the category
$\lfdMod OS$; this sends a module to the direct sum of the linear
duals of its word spaces.
Since $\D(\LA)^\circledast \cong \D(\LA)$ and $\Y(\LA)^\circledast \cong \Y(\LA)$, 
the following lemma implies that
\begin{align}\label{duality2}
\L(\LA)^\circledast &\cong \L(\LA),&
\Delta(\LA)^\circledast &\cong \nabla(\LA),&
\bar\Delta(\LA)^\circledast \cong \bar\nabla(\LA),
\end{align}
for any $\LA \in \RPar$. In particular, this means that $\L(\LA)$ can
also
be realized as the
irreducible socle of $\bar\nabla(\LA)$.

\begin{lemma}
The functors $\Delta$ and $\nabla$ send finite-dimensional
$OS^\circ$-modules to locally finite-dimensional $OS$-modules. Moreover,
the functors $\circledast \circ \Delta$ and $\nabla \circ \circledast$
are isomorphic on finite-dimensional $OS^\circ$-modules.
\end{lemma}

\begin{proof}
The first statement is easy to see from the definitions; for $\nabla$,
one needs to
know that $1_\a OS$ is a finitely generated right $OS^\flat$-module by Lemma~\ref{td}.
Then, for a finite-dimensional $OS^\circ$-module $M$, we define an $OS$-module homomorphism
$$
(\infl^\sharp M \otimes_{\OS^\sharp} OS)^\circledast\rightarrow
\bigoplus_{\a \in \words} \Hom_{OS^\flat}(1_\a OS, \infl^\flat (M^\circledast)),
\qquad
\theta \mapsto \widehat{\theta}
$$
where $\widehat{\theta}(f)(v) = \theta(v \otimes \tau(f))$
for $v \in M, f \in 1_\a OS$.
There is a two-sided inverse
$$
\bigoplus_{\a \in \words} \Hom_{OS^\flat}(1_\a OS, \infl^\flat
(M^\circledast))
\rightarrow
(\infl^\sharp M \otimes_{\OS^\sharp} OS)^\circledast,
\qquad
\psi \mapsto \widetilde{\psi}
$$
where $\widetilde{\psi}(v \otimes f) = \psi(\tau(f))(v)$.
\end{proof}

We say that an $OS$-module $M$ has a {\em finite $\Delta$-flag} if it has a
finite filtration whose sections are isomorphic to standard modules
$\Delta(\LA)$ for various $\LA \in \RPar$.
Let $\deltaMod OS$ be the full subcategory of
$\Mod OS$ consisting of all modules with a finite $\Delta$-flag.
We view it as an exact category with admissible sequences being the
ones that are exact in $\Mod OS$.
The next three lemmas are all well known in this sort of situation.

\begin{lemma}\label{rflat}
The restriction of $\Delta(\LA)$ to $OS^\flat$ is isomorphic to
$\Y(\LA) \otimes_{OS^\circ} OS^\flat$. These modules give all of the
indecomposable projective $OS^\flat$-modules (up to isomorphism).
\end{lemma}

\begin{proof}
The first statement is clear from Lemma~\ref{td}.
For the second, observe that
the $OS^\flat$-modules $\Y(\LA) \otimes_{OS^\circ} OS^\flat$ are induced
from the projective $OS^\circ$-modules, hence, they are projective and every
indecomposable projective $OS^\flat$-module is isomorphic to a summand of one of
them.
It remains to show that $\Y(\LA)\otimes_{OS^\circ} OS^\flat$ is
indecomposable. This follows because
$\End_{OS^\flat}(\Y(\LA) \otimes_{OS^\circ} OS^\flat)
\cong \End_{OS^\circ}(\Y(\LA))$, which is local as $\Y(\LA)$ is
indecomposable.
\end{proof}

In view of the following lemma, the Grothendieck group $K_0(\deltaMod
OS)$ of the exact category $\deltaMod OS$
is the free $\ZZ$-module on basis
$\{[\Delta(\LA)]\:|\:\LA \in \RPar\}$.

\begin{lemma}\label{fly}
For $\LA,\MU \in \RPar$ and $d \geq 0$, we have that
$$
\dim \operatorname{Ext}^d_{OS}(\Delta(\LA),\bar\nabla(\MU)) = 
\left\{
\begin{array}{ll}
1&\text{if $d=0$ and $\LA=\MU$,}\\
0&\text{otherwise.}
\end{array}
\right.
$$
Hence, for any $M \in \Mod OS$ with a finite $\Delta$-flag, the multiplicity
$(M:\Delta(\LA))$ of $\Delta(\LA)$ as a section of such a flag
is well defined independent of the particular choice of flag, and it satisfies
$(M:\Delta(\LA)) = \dim \Hom_{OS}(M, \bar\nabla(\LA))$.
\end{lemma}

\begin{proof}
For the first statement, we have natural isomorphisms
\begin{align*}
\operatorname{Ext}^d_{OS}(\Delta(\LA),\bar\nabla(\MU))
&\cong
\operatorname{Ext}^d_{OS}\Bigg(\Delta(\LA),
\bigoplus_{\a \in \words}
\Hom_{OS^\flat}\left(1_\a OS, \infl^\flat \D(\MU)\right)\Bigg)\\
&\cong
\operatorname{Ext}^d_{OS^\flat}(\Y(\LA) \otimes_{OS^\circ} OS^\flat,
\infl^\flat \D(\MU))\\
&\cong
\operatorname{Ext}^d_{OS^\circ}(\Y(\LA),
\D(\MU)).
\end{align*}
This is zero unless $\LA=\MU$ and $d=0$ as $\Y(\LA)$ is the projective
cover of $\D(\LA)$.
It follows that $\dim \Hom_{OS}(M,\bar\nabla(\LA))$ counts the
multiplicity of $\Delta(\LA)$ in a $\Delta$-flag of $M$, giving the
second statement.
\end{proof}

The next lemma shows that $\deltaMod OS$ is Karoubian.

\begin{lemma}\label{crit}
An $OS$-module $M$
has a finite $\Delta$-flag if and only if it is finitely generated and projective as an
  $OS^\flat$-module.
Hence, any direct summand of a module with a finite $\Delta$-flag also
has a finite $\Delta$-flag.
\end{lemma}

\begin{proof}
If $M$ has a finite $\Delta$-flag, then it is finitely generated and
projective over $OS^\flat$ thanks to Lemma~\ref{rflat}.
Conversely, suppose that $M$ is finitely generated and projective over
$OS^\flat$, so that the restriction of $M$ to $OS^\flat$
is isomorphic to a direct sum of some number $n$ of $OS^\flat$-modules of the form
$\Y(\LA)\otimes_{OS^\circ} OS^\flat$.
We show that $M$ has a finite $\Delta$-flag by induction on $n$.
The case $n=0$ is trivial.
If $n> 0$, we choose $r,s\geq 0$ with $r+s$ minimal
such that the restriction of $M$ to $OS^\flat$ has a summand $M'
\cong \Y(\LA)\otimes_{OS^\circ} OS^\flat$
for some $\LA \in \RPar_{r,s}$.
The $OS^\circ$-module homomorphism $\Y(\LA) \cong \Y(\LA) \otimes_{OS^\circ}
OS^\flat[0]
\hookrightarrow M'$ is actually an $OS^\sharp$-module
homomorphism $\infl^\sharp \Y(\LA) \hookrightarrow M$ since its image is
in the shortest word space of $M$. Hence, we get induced an $OS$-module
homomorphism $\Delta(\LA) \rightarrow M$ with image $M'$.
This shows that $M'$ is
actually an $OS$-submodule of $M$ and $M' \cong \Delta(\LA)$.
The quotient $M / M'$ is finitely generated and projective over
$OS^\flat$ with one fewer indecomposable summand.
It remains to apply the induction hypothesis to deduce that $M / M'$
has a finite $\Delta$-flag, hence, so does $M$.
\end{proof}

Now we look at projectives.
Let $\P(\LA)$ be a projective cover of $\L(\LA)$.
The classes $\{[\P(\LA)]\:|\:\LA \in \RPar\}$ give a basis for
$K_0(\pMod OS)$.
The $OS$-module 
\begin{equation}\label{qd}
\Q(\LA) := \Y(\LA) \otimes_{OS^\circ} OS
\end{equation}
described by the following theorem should be viewed as a first approximation
to $\P(\LA)$.

\begin{theorem}\label{koike1}
For $\LA \in \RPar_{r,s}$,
the module $\Q(\LA)$
has a canonical
filtration with sections indexed by $d=0,1,\dots,\min(r,s)$ appearing
in order from top to bottom, such that the $d$th section is isomorphic to
\begin{equation}\label{hmm}
\bigoplus_{\MU \in \RPar_{r-d,s-d}} 
\Delta(\MU)^{\oplus M^\LA_\MU(e,p)}
\end{equation}
where
$M^\LA_\MU(e,p) :=
\sum_{\NU \in \RPar_{d,d}}
\big[\D_{\mu^\up} \circ \D_{\nu^\up}:\D_{\lambda^\up}\big]
\big[\D_{\mu^\down} \circ \D_{\nu^\down}:\D_{\lambda^\down}\big]
[\Y_{\nu^\up}:\D_{\nu^\down}]$ (which depends on $e$ and the
characteristic $p$ of the field $\k$).
\end{theorem}

\begin{proof}
Take $0 \leq d \leq \min(r,s)$.
We always view $H_{d}$ as a subalgebra of $H_r$ or $H_s$ via the
natural embeddings.
We
also need the ``unnatural'' embeddings $H_{r-d} \hookrightarrow H_r$
and $H_{s-d}\hookrightarrow H_s$ 
which send $S_i \mapsto S_{d+i}$; the images of these embeddings
centralize $H_d$.
Let $\Hom_{H_d}(H_s,H_r)$ be the space of all right $H_d$-module
homomorphisms.
Using the unnatural embeddings for $H_{r-d}$ and $H_{s-d}$,
this is an $(H_r \otimes H_{s-d}, H_{r-d}\otimes H_s)$-bimodule.
Let $\tau\text{-\!}\Hom_{H_d}(H_s,H_r)$ be the 
the same space but with the left action of $H_{s-d}$ and right action of
$H_{s}$ twisted into right and left actions, respectively, using the
anti-automorphism $\tau$ which sends $S_w \mapsto
S_{w^{-1}}$.
This means that
$\tau\text{-\!}\Hom_{H_d}(H_s,H_r)$ is an $(H_{r,s}, H_{r-d,s-d})$-bimodule.
The space $1_{\down^s \up^r} OS^\sharp 1_{\down^{s-d} \up^{r-d}}$
is also naturally an $(H_{r,s}, H_{r-d,s-d})$-bimodule. 
We claim that these two bimodules are isomorphic.

To prove the claim, 
$H_s$ is a free right
$H_d$-module with basis $\{S_y\:|\:y \in \mathfrak{D}\}$ where $\mathfrak{D}$ is the set of
minimal length $\Sym_s / \Sym_d$-coset representatives.
So the $H_d$-module homomorphisms $\left\{f_{x,y}:H_s\rightarrow H_r\:|\:x \in \Sym_r, y \in \mathfrak{D}\right\}$
defined from
$f_{x,y}(S_z) := \delta_{y,z} S_x$
for $x\in\Sym_r, y,z \in \mathfrak{D}$
give a linear basis for $\tau\text{-\!}\Hom_{H_d}(H_s,H_r)$.
Let
$$
c := 
\mathord{
\begin{tikzpicture}[baseline = -6mm]
\node at (.2,-1.1) {$\scriptstyle r-d$};
\node at (-.25,-.05) {$\scriptstyle d$};
\node at (-.7,-1.1) {$\scriptstyle s-d$};
\draw[->,ultra thick,darkblue] (-.65,-.2) to (-.65,-1);
\draw[<-,ultra thick,darkblue] (.25,-.2) to (.25,-1);
\draw[-,line width=5pt,white] (.15,-.6) to [out=0,in=-90] (.55,-.2);
\draw[->,ultra thick,darkblue] (.15,-.6) to [out=0,in=-90] (.55,-.2);
\draw[-,ultra thick,darkblue] (-.25,-.2) to [out=-90,in=180] (.15,-.6);
\end{tikzpicture}
}
\in 1_{\down^s \up^r} OS^\sharp 
1_{\down^{s-d} \up^{r-d}},
$$
where the thick arrows labelled by a number represent that number of parallel thin ones.
We will show that
the linear map
$$
\theta:\tau\text{-\!}\Hom_{H_d}(H_s,H_r) \rightarrow
1_{\down^s \up^r} OS^\sharp 1_{\down^{s-d} \up^{r-d}},
\qquad
f_{x,y} \mapsto \i_{r,s}(S_x \otimes S_y) c
$$
is an $(H_{r,s}, H_{r-d,s-d})$-bimodule isomorphism.
To see this, it is clear from Theorem~\ref{bt} that $\theta$ is a vector
space isomorphism. We must check that it is a bimodule homomorphism.
This is straightforward for the left action of $H_r$ and the right
action of $H_{r-d}$. In the next two paragraphs, we check it for the
left action of $H_s$ and the right action of $H_{s-d}$, respectively.

To show $\theta$ is a left $H_s$-module
homomorphism, take $x \in \Sym_r, y \in \mathfrak{D}$ and $1 \leq i < s$.
By \cite[Lemma 1.1]{DJ}, exactly one of the following holds: (a) $s_i
y \in \mathfrak{D}$; (b) $s_{y^{-1}(i)} \in \Sym_d$.
In case (a), $S_i S_y = S_{s_i y}$ if $\ell(s_iy) > \ell(y)$
or $S_{s_i y} + (q-q^{-1}) S_y$ if $\ell(s_i y) < \ell(y)$.
In case (b), $S_i S_y = S_y S_{y^{-1}(i)}$
and $$
\i_{r,s}(S_x \otimes S_i S_y) c = 
\i_{r,s}(S_x \otimes S_y S_{y^{-1}(i)}) c
= \i_{r,s}(S_x S_{y^{-1}(i)} \otimes S_y) c.
$$
We deduce for $z \in \mathfrak{D}$ that
\begin{align*}
(\theta^{-1}(S_i &\theta(f_{x,y})))(S_z) = 
(\theta^{-1}( \i_{r,s}(S_x \otimes S_i S_y) c))(S_z)\\
&=
\left\{
\begin{array}{ll}
f_{x,s_i y}(S_z)&\text{if $s_i y \in \mathfrak{D}, \ell(s_i y) > \ell(y)$,}\\
(f_{x,s_i y} + (q-q^{-1}) f_{x,y})(S_z)&\text{if $s_i y \in \mathfrak{D}, \ell(s_i y) < \ell(y)$,}\\
f_{x s_{y^{-1}(i)},y}(S_z)&\text{if $s_i y \notin \mathfrak{D}, \ell(x s_{y^{-1}(i)}) > \ell(x)$,}\\
(f_{x s_{y^{-1}(i)},y}+(q-q^{-1}) f_{x ,y})(S_z)&\text{if $s_i y \notin \mathfrak{D}, \ell(x s_{y^{-1}(i)}) < \ell(x)$}
\end{array}
\right.\\
&=
\left\{
\begin{array}{ll}
\delta_{s_i y, z} S_x\hspace{36.8mm}&\text{if $s_i y \in \mathfrak{D}, \ell(s_i y) > \ell(y)$,}\\
\delta_{s_i y, z} S_x + (q-q^{-1}) \delta_{y,z} S_x&\text{if $s_i y \in \mathfrak{D}, \ell(s_i y) < \ell(y)$,}\\
\delta_{y,z} S_x S_{y^{-1}(i)}&\text{if $s_i y \notin \mathfrak{D}$}.
\end{array}
\right.
\end{align*}
We need to show this is equal to
$$
(S_i f_{x,y})(S_z) = f_{x,y}(S_i S_z) = 
\left\{
\begin{array}{ll}
\delta_{y,s_i  z} S_x&\text{if $s_i z \in \mathfrak{D}, \ell(s_i z) > \ell(z)$},\\
\delta_{y,s_i  z} S_x + (q-q^{-1}) \delta_{y,z} S_x&\text{if $s_i z \in \mathfrak{D}, \ell(s_i z) < \ell(z)$},\\
\delta_{y,z} S_x S_{y^{-1}(i)}&\text{if $s_i z \notin \mathfrak{D}$}.\\
\end{array}\right.
$$
This follows easily by considering several cases: (a) $y = z$; (b) $y = s_i
z$; (c) $s_i y = z$; (d) none of the above.

To show that $\theta$ is a right $H_{s-d}$-module homomorphism,
take $x \in \Sym_r, y, z \in \mathfrak{D}$ and $1 \leq i < s-d$.
We must show that
$(\theta^{-1}(\theta(f_{x,y}) S_i))(S_z)
= (f_{x,y} S_{d+i}) (S_z)$, i.e.,
$$
(\theta^{-1}(\i_{r,s}(S_x \otimes S_y S_{d+i})c)(S_z)
=
f_{x,y}(S_z S_{d+i}).
$$
This time, 
$y s_{d+i} \in \mathfrak{D}$ always.
So the left hand side of the identity to be proved equals
$$
\left\{\begin{array}{ll}
\delta_{y s_{d+i}, z} S_x
&\text{if $\ell(y s_{d+i}) > \ell(y)$,}\\
\delta_{y s_{d+i}, z}S_x + \delta_{y,z} (q-q^{-1})S_x&\text{if $\ell(y s_{d+i}) < \ell(y)$.}
\end{array}\right.
$$
Similarly, the right hand side is
$$
\left\{\begin{array}{ll}
\delta_{y,z s_{d+i}} S_x
&\text{if $\ell(z s_{d+i}) > \ell(z)$,}\\
\delta_{y, z s_{d+i}}S_x + \delta_{y,z} (q-q^{-1})S_x\:&\text{if $\ell(z s_{d+i}) < \ell(z)$.}
\end{array}\right.
$$
The two sides are equal by considering four cases like before.

We have now proved the claim made in the opening paragraph.
To prove the theorem, we must construct the filtration of $\Q(\LA)$.
By transitivity of induction,  
$$
\Q(\LA) = (\Y(\LA) \otimes_{OS^\circ} OS^\sharp) \otimes_{OS^\sharp} OS.
$$
Since $OS^\sharp$ is positively graded, the grading gives us a filtration of 
$\Y(\LA)
\otimes_{OS^\circ} OS^\sharp$ as an $OS^\sharp$-module
with sections $\Y(\LA) \otimes_{OS^\circ} OS^\sharp[d]$ for $d=0,1,\dots$
appearing in order from top to bottom. The $d$th section is clearly zero unless $d
\leq \min(r,s)$.
Since the functor $? \otimes_{OS^\sharp} OS$ is exact, we are thus reduced
to checking
for $0 \leq d \leq \min(r,s)$ that
$$
\Y(\LA) \otimes_{OS^\circ} OS^\sharp[d] \cong
\bigoplus_{\MU \in \RPar_{r-d,s-d}} 
\Y(\MU)^{\oplus M^\LA_\MU(e,p)}
$$
as a right $OS^\circ$-module.
Using Lemma~\ref{cartanlem}, we show equivalently that
$$
(\Y_{\la^\up} \boxtimes \Y_{\la^\down}) \otimes_{H_{r,s}}
1_{\down^s\up^r} OS^\sharp 1_{\down^{s-d}\up^{r-d}}
\cong \bigoplus_{\mu \in \RPar_{r-d,s-d}} (\Y_{\mu^\up} \otimes
\Y_{\mu^\down})^{\bigoplus M^\LA_\MU(e,p)}
$$
as a right $H_{r-d,s-d}$-module.
By the opening claim, the module on the left hand side is isomorphic
to
$$
(\Y_{\la^\up} \boxtimes \Y_{\la^\down}) \otimes_{H_{r,s}}
\tau\text{-\!}\Hom_{H_d}(H_s,H_r)
\cong 
\tau\text{-\!}\Hom_{H_d}(\Y_{\la^\down},\Y_{\la^\up}),
$$
where we have used the self-duality of $\Y_{\la^\down}$.
Then we note by Frobenius reciprocity
that
\begin{align*}
\res^{H_r}_{H_{r-d,d}} \Y_{\la^\up} &\cong
\bigoplus_{\mu^\up \vdash (r-d), \nu^\up \vdash d} 
(\Y_{\mu^\up} \boxtimes \Y_{\nu^\up})^{\oplus
[\D_{\mu^\up}\circ \D_{\nu^\up}:\D_{\la^\up}]},\\
\res^{H_s}_{H_{s-d,d}} \Y_{\la^\down} &\cong
\bigoplus_{\mu^\down \vdash (s-d), \nu^\down \vdash d} 
(\Y_{\mu^\down} \boxtimes  \Y_{\nu^\down})^{\oplus [\D_{\mu^\down}\circ \D_{\nu^\down}:\D_{\la^\down}]}.
\end{align*}
Making these substitutions in
$\tau\text{-\!}\Hom_{H_d}(\Y_{\la^\down},\Y_{\la^\up})$, using also
$\dim \Hom_{H_d}(\Y_{\nu^\down}, \Y_{\nu^\up}) = \left[\Y_{\nu^\up}:\D_{\nu^\down}\right]$
and self-duality of $\Y_{\mu^\down}$, gives the conclusion.
\end{proof}

\begin{corollary}\label{kio}
For $\LA \in \RPar_{r,s}$, the module $\Q(\LA)$ is isomorphic
$\P(\LA)$ plus a finite direct sum of projectives
$\P(\MU)$ for bipartitions $\MU \in \bigsqcup_{d = 0}^{\min(r,s)-1}
\RPar_{r-d,s-d}$.
Hence, the classes $\{[\Q(\LA)]\:|\:\LA \in \RPar\}$ give another basis for
$K_0(\pMod OS)$.
\end{corollary}

\begin{proof}
Note that $\Q(\LA)$ is projective since the
functor
$? \otimes_{OS^\circ} OS$ sends projectives to projectives.
Also the top section of the $\Delta$-flag of $\Q(\LA)$ constructed in Theorem~\ref{koike1}
is $\Delta(\LA)$, so $\Q(\LA)$ has $\P(\LA)$ as an indecomposable summand.
The other sections only involve 
$\Delta(\MU)$ for $\MU \in \bigsqcup_{d = 0}^{\min(r,s)-1} \Par_{r-d,s-d}$,
so all other summands are of the form $\P(\MU)$ for such $\MU$.
\end{proof}

\begin{corollary}\label{BGG}
The projective cover $\P(\LA)$ of $\L(\LA)$ has a finite $\Delta$-flag
such that
$$
(\P(\LA):\Delta(\MU)) =[\bar\Delta(\MU):\L(\LA)].
$$
 \end{corollary}

\begin{proof}
By Corollary~\ref{kio} and Theorem~\ref{koike1}, $\P(\LA)$ is a summand of $\Q(\LA)$, and $\Q(\LA)$ has
a finite $\Delta$-flag.
Hence, $\P(\LA)$ has one too due to Lemma~\ref{crit}.
To deduce the BGG reciprocity formula, 
we use Lemma~\ref{fly}:
$
(\P(\LA):\Delta(\MU))=
\dim \Hom_{OS}(\P(\LA), \bar\nabla(\MU)) = [\bar\nabla(\MU):\L(\LA)].
$
This equals $[\bar\Delta(\MU):\L(\LA)]$ by (\ref{duality2}).
\end{proof}

\begin{corollary}\label{exactsubcat}
By Corollary~\ref{BGG}, there is an embedding $\pMod OS \rightarrow\deltaMod OS$.
This induces an
isomorphism
$K_0(\pMod OS) \stackrel{\sim}{\rightarrow} K_0(\deltaMod OS)$.
\end{corollary}

\begin{proof}
The transition matrix arising from (\ref{hmm}) can
be inverted to express each $[\Delta(\LA)]$ as a finite linear
combination of $[\Q(\MU)]$'s.
\end{proof}

\begin{corollary}\label{BGGplus}
For $\LA \in \RPar_{r,s}$, $\P(\LA)$ has a finite filtration with sections $\tilde\Delta(\MU)$ for
$\MU \in \bigsqcup_{d=0}^{\min(r,s)}\Par_{r-d,s-d}$, 
each appearing $[\tilde\Delta(\MU):\L(\LA)]$
times.
\end{corollary}

\begin{proof}
Recall for $\LA \in \RPar_{r,s}$ 
that $\Y(\LA)$ has a finite filtration with sections $\SS(\MU)$, each
appearing $[\SS(\MU):\D(\LA)]$ times.
Applying the exact standardization functor, we deduce that
$\Delta(\LA)$ has a finite filtration with sections
$\tilde\Delta(\MU)$, each
appearing $[\SS(\MU):\D(\LA)]$ times.
Combined with Corollary~\ref{BGG}, it follows that $\P(\LA)$ has a finite
filtration with sections $\tilde\Delta(\MU)$, each appearing 
with multiplicity
$$
\sum_{\NU \in \RPar} [\bar\Delta(\NU):\L(\LA)] [\SS(\MU):\D(\NU)].
$$
Also, applying $\Delta$ to a composition series of $\SS(\MU)$, we see
that 
$\tilde\Delta(\MU)$ has a filtration with sections
$\bar\Delta(\NU)$, each appearing $[\SS(\MU):\D(\NU)]$ times.
Hence, the multiplicity just displayed is equal to
$[\tilde\Delta(\MU):\L(\LA)]$.
\end{proof}

\begin{proof}[Proof of Theorem~\ref{delection}]
The monoidal functor $\OS{^\circ}(z,t) \rightarrow \OS(z,t)$
corresponds to the induction functor
$?\otimes_{OS^\circ} OS:\pMod OS^\circ \rightarrow \pMod OS$, since the latter
sends $e OS^\circ$ to $e OS$ for any idempotent $e$.
So by the definition (\ref{qd}) it sends $\Y(\LA)$ to $\Q(\LA)$.
Theorem~\ref{delection} follows because 
the classes
$\left\{[\Q(\LA)]\:|\:\LA \in \RPar\right\}$
form a basis for $K_0(\pMod OS)$ according to Corollary~\ref{kio}.
\end{proof}

For the next lemma, we return to the situation of
Theorem~\ref{webs}.
We want to relate the labelling of irreducible $OS$-modules obtained
thus far with the usual labelling of irreducible
$U_q(\mathfrak{gl}_{n})$-modules
via their highest weights.
Take $\LA \in \Par_{r,s}$.
Since $e=0$, Theorem~\ref{koike1} tells us simply that $\Q(\LA)$ has a
filtration with sections 
\begin{equation}\label{koike3}
\bigoplus_{\MU \in \Par_{r-d,s-d}}
\Delta(\MU)^{\oplus M^\LA_\MU}
\qquad\text{where}\qquad
M^\LA_\MU:= M^\LA_\MU(0,0) =
\sum_{\nu \vdash d}
LR^{\lambda^\up}_{\mu^\up, \nu}
LR^{\lambda^\down}_{\mu^\down, \nu}
\end{equation}
for $d=0,\dots,\min(r,s)$,
and it decomposes as $\P(\LA)$ plus a direct sum of projectives
$\P(\MU)$ for various bipartitions $\MU$ obtained from $\LA$ by
removing the same number $d > 0$
of boxes from both $\la^\up$ and $\la^\down$.
Recalling 
the Young symmetrizer (\ref{youngsymmetrizer}),
we have that $\Y(\LA) = \SS(\LA) = \D(\LA) = \i_{r,s}(e_{\la^\up} \otimes
e_{\la^\down}) OS^\circ$. Hence,
\begin{equation}\label{youngsymmetrizer2}
\Q(\LA) = \i_{r,s}(e_{\la^\up} \otimes
e_{\la^\down}) OS.
\end{equation}
Let $e_\LA$ 
be the projection of $\Q(\LA)$ onto its unique summand that
is isomorphic to $\P(\LA)$. Thus,
$e_\LA$ is a primitive idempotent in the quantized walled Brauer algebra
$B_{r,s}$. The following recovers results of \cite{KM1, KM2}.

\begin{lemma}\label{koike2}
Let  notation be as in Theorem~\ref{webs} and assume $n \geq 0$.
Take $\LA \in \Par_{r,s}$ such that $h(\LA)$, its total number of non-zero
parts, is $\leq n$.
Consider the idempotent $\Psi(e_\LA)\in
\End_{U_q(\mathfrak{gl}_{n})}\left((V^-)^{\otimes s}
  \otimes (V^+)^{\otimes r}\right)$.
Its image is the irreducible $U_q(\mathfrak{gl}_{n})$-module $\V(\LA)$ labelled by
the dominant weight
\begin{equation}\label{displayed}
(\la_1^\up-\la_{n}^\down) \eps_1+(\la_2^\up-\la_{n-1}^\down)\eps_2+\cdots + (\la_{n}^\up-\la_1^\down)
\eps_{n},
\end{equation}
using standard conventions for the root system of $\mathfrak{gl}_{n}$.
\end{lemma}

\begin{proof}
We proceed by induction on $r+s$, the case $r+s=0$ being trivial.
Since $e_{\lambda^\up}$ is the Young symmetrizer, the image of
$\Psi(e_{\lambda^\up})\in \End_{U_q(\mathfrak{gl}_{n})}\left((V^+)^{\otimes
    r}\right)$
is the irreducible polynomial representation 
$\V(\la^\up)$ of $U_q(\mathfrak{gl}_{n})$
of highest weight $\lambda_1^\up\eps_1+\cdots+\lambda_{n}^\up\eps_{n}$.
Similarly, the image of $\Psi(e_{\lambda^\down})$ is the dual
irreducible polynomial representation 
$\V(\la^\down)^*$ of highest weight 
$-\la_{n}^\down \eps_1+\cdots+\la_1^\down\eps_{n}$.
Hence, the image of 
$\Psi(\i_{r,s}(e_{\la^\up} \otimes
e_{\la^\down}))$ 
is 
$\V(\la^\down)^* \otimes \V(\la^\up)$.
Using characters, it is easy to see that this tensor product has a unique irreducible
constituent $\V(\LA)$ of
highest weight (\ref{displayed}), plus a sum of irreducible
modules $\V(\MU)$ for bipartitions
$\MU \in \bigsqcup_{d > 0}\Par_{r-d,s-d}$ with $h(\MU) \leq n$.
Now using induction, we deduce that $\Psi(e_\LA)$ must be the
projection onto $\V(\LA)$.
\end{proof}

\begin{remark}
A helpful picture of the weight (\ref{displayed}) is displayed in
\cite[Figure 2]{Koike}. It is also worth noting that multiplicities
$M^\LA_\MU$
appearing in (\ref{koike3}) are the same as the 
$U_q(\mathfrak{gl}_{n})$-composition multiplicities $[\V(\la^\down)^* \otimes \V(\la^\up):\V(\MU)]$
computed in \cite[Corollary 2.3.1]{Koike}.
Given this, the same induction as used to prove Lemma~\ref{koike2}
can be used to show that $\Delta(\LA) = \P(\LA)$ when in the situation
of the lemma.
We will prove this in a different way in Corollary~\ref{morocco} below.
\end{remark}

\begin{remark}
To get the appropriate analog of Lemma~\ref{koike2} when $n \leq 0$, 
one just needs to twist by the isomorphism $\#$. 
Recalling at the level of the Hecke algebra that this is ``tensoring
with sign,''
one can show that $\#$ 
maps
the primitive idempotent 
$e_\LA$ to a conjugate of
$e_{\LA^\trans}$, where $\LA^t := \left((\la^\up)^\trans,
  (\la^\down)^\trans\right)$.
So, for negative $n$,
the $U_q(\mathfrak{gl}_n)$-module
$V(\LA)$ arises as the image of $\Psi(e_{\LA^\trans})$ (instead of $\Psi(e_\LA)$).
\end{remark}

The final result in the section justifies the description of $K_0(\dot\OS(z,t))$
made after Theorem~\ref{delection} in the introduction; the discussion
there also depends on Theorem~\ref{selection} which will be proved in the
next section, and the highest weight/standardly stratified structure
which will be explained in section~\ref{tpc}.

\begin{lemma}\label{commute}
For $\LA \in \RPar_{r,s}$ and $\NU \in \Par_{r-d,s-d}$, we have that
$$
\sum_{\MU \in \Par_{r,s}} [\SS(\MU):\D(\LA)] M^\MU_\NU
=
\sum_{\MU \in \RPar_{r-d,s-d}} M^\LA_\MU(e,p) [\SS(\NU):\D(\MU)].
$$
\end{lemma}

\begin{proof}
We have that
$$
\sum_{\MU \in \RPar_{r-d,s-d}} M^\LA_\MU(e,p) [\SS(\NU):\D(\MU)]\hspace{65mm}
$$ \vspace{-7mm}
\begin{align*}
\qquad\qquad&=
\sum_{\substack{\MU \in \RPar_{r-d,s-d}\\\KAPPA \in \RPar_{d,d}, \gamma
  \vdash d}}
\begin{array}{l}
\\
\big[\D_{\mu^\up} \circ \D_{\kappa^\up}:\D_{\la^\up}\big]
\big[\SS_{\nu^\up}:\D_{\mu^\up}\big]
\big[\SS_\gamma:\D_{\kappa^\up}\big]\times\\
\qquad\qquad\qquad\big[\D_{\mu^\down} \circ \D_{\kappa^\down}:\D_{\la^\down}\big]
\big[\SS_{\nu^\down}:\D_{\mu^\down}\big]
\big[\SS_\gamma:\D_{\kappa^\down}\big]
\end{array}
\\
&=
\sum_{\substack{\KAPPA \in \RPar_{d,d}, \gamma
  \vdash d}}\big[\SS_{\nu^\up} \circ \D_{\kappa^\up}:\D_{\la^\up}\big]
\big[\SS_\gamma:\D_{\kappa^\up}\big]
\big[\SS_{\nu^\down} \circ \D_{\kappa^\down}:\D_{\la^\down}\big]
\big[\SS_\gamma:\D_{\kappa^\down}\big]\\
&=
\qquad\sum_{\gamma
  \vdash d}\big[\SS_{\nu^\up} \circ \SS_\gamma:\D_{\la^\up}\big]
\big[\SS_{\nu^\down} \circ \SS_\gamma:\D_{\la^\down}\big]\\
&=
\:\,\sum_{\substack{\MU \in \Par_{r,s}, \gamma
  \vdash d}}\big[\SS_{\nu^\up} \circ \SS_\gamma:\SS_{\mu^\up}\big]
\big[\SS_{\mu^\up}:\D_{\lambda^\up}]
\big[\SS_{\nu^\down} \circ \SS_{\gamma}:\SS_{\mu^\down}\big]
\big[\SS_{\mu^\down}:\D_{\lambda^\down}]\\
&=
\quad\sum_{\MU \in \Par_{r,s}}
[\SS(\MU):\D(\LA)]M^\MU_\NU.
\end{align*}
\end{proof}

\begin{theorem}\label{thegg}
For any choices of $q$ and $t$, 
the ring $K_0(\pMod OS)$ may be identified with a
subring of $\SYM\otimes_\ZZ\SYM$ so that (\ref{k0b}) 
and (\ref{k0c}) hold.
\end{theorem}

\begin{proof}
When $e=0$, 
Lemma~\ref{cartanlem} and the well-known representation theory of Hecke
algebras
imply that the rings
$K_0(\pMod OS^\circ)$ and $\SYM\otimes_\ZZ \SYM$ 
may be identified so that 
$[\SS(\LA)] \leftrightarrow \chi_{\la^\up}\otimes \chi_{\la^\down}.$
For general $e$, using also Brauer reciprocity for the Hecke algebra, we may identify
 $K_0(\pMod OS^\circ)$ with a subring of
$\SYM\otimes_\ZZ\SYM$
so that $$
[\Y(\LA)] \leftrightarrow \sum_{\MU \in \Par_{r,s}} [\SS(\MU):\D(\LA)]
\chi_{\mu^\up}\otimes \chi_{\mu^\down}
$$
for $\LA \in \RPar_{r,s}$ and $r,s\geq 0$.
In view of Theorem~\ref{delection} (and its proof), we deduce that
$K_0(\pMod OS)$ is identified with a subring of $\SYM\otimes_\ZZ\SYM$ so that
$$
[\Q(\LA)] \leftrightarrow  
\sum_{\MU \in \Par_{r,s}} [\SS(\MU):\D(\LA)]
\chi_{\mu^\up}\otimes \chi_{\mu^\down}
$$
for $\LA \in \RPar_{r,s}$ and $r,s\geq 0$.
Now recall the definition (\ref{nlamu}) and (\ref{koike3}).
Setting $N^\LA_\MU = M^\LA_\MU := 0$ whenever $\LA \in \Par_{r,s}$ and 
$\MU \notin \bigsqcup_{d=0}^{\min(r,s)} \Par_{r-d,s-d}$,
the matrix 
$(N^\LA_\MU)_{\LA,\MU \in \Par}$ is inverse to the matrix
$(M^\LA_\MU)_{\LA,\MU \in \Par}$ by 
\cite[Theorem 2.3]{Koike}.
So
$$
[\Q(\LA)] \leftrightarrow  
\sum_{\substack{\MU \in \Par_{r,s}\\0 \leq d \leq \min(r,s)\\\NU \in \Par_{r-d,s-d}}} [\SS(\MU):\D(\LA)] M^\MU_\NU \chi_\NU
$$
for $\LA \in \RPar_{r,s}$ and $r,s \geq 0$.
By Lemma~\ref{commute}, this gives
$$
[\Q(\LA)] \leftrightarrow
\sum_{\substack{0 \leq d \leq \min(r,s)\\\LA \in \RPar_{r-d,s-d}\\\NU
    \in \Par_{r-d,s-d}}} 
M^\LA_\MU(e,p) [\SS(\NU):\D(\MU)] \chi_\NU.
$$
Now use Corollary~\ref{exactsubcat} to identify
$K_0(\pMod OS) = K_0(\deltaMod OS)$.
Comparing with (\ref{hmm}), we deduce that
$$
[\Delta(\LA)] \leftrightarrow
\sum_{\NU \in \Par_{r,s}}
[\SS(\NU):\D(\LA)] \chi_\NU
$$
for $\LA \in \RPar_{r,s}$.
This establishes (\ref{k0b}).
To get (\ref{k0c}) too, use
Corollary~\ref{BGG}.
\end{proof}

\section{Branching rules and characters}\label{schars}

We continue with the setup of the previous section.
In this section, we introduce a biadjoint pair of endofunctors
$E$ and $F$ of $\Mod OS$, which lift the endofunctors $\up \otimes ?$ and $\down
\otimes ?$ of $\OS(z,t)$.
We will use the Jucys-Murphy elements from section \ref{aos}
to decompose these endofunctors
into direct sums of refined functors $E_i$ and
$F_i$, which we study by comparing them 
to some well-known induction and restriction functors on $\Mod OS^\circ$.

To start with, let us recall some standard facts about induction and
restriction for the Iwahori-Hecke algebra $H_r$.
Let 
\begin{equation}
\ind_{r-1}^{r} 
:\Mod H_{r-1} \rightarrow \Mod H_{r},
\qquad
\res_{r-1}^{r} :\Mod H_{r} \rightarrow \Mod H_{r-1}
\end{equation} be the usual 
induction and restriction functors with respect to the natural embedding $H_{r-1}
\hookrightarrow H_{r}, S_i \mapsto S_i$.
So $\ind_{r-1}^{r}$ is defined by tensoring over $H_{r-1}$ with $H_{r}$
viewed as an $(H_{r-1}, H_{r})$-bimodule and 
$\res_{r-1}^{r}$ is defined by tensoring over $H_{r}$ with $H_{r}$
viewed as an $(H_{r}, H_{r-1})$-bimodule;
equivalently, $\res_{r-1}^{r}$ is the functor $\Hom_{H_{r}}(H_{r}, ?)$.
Adjointness of tensor and hom implies that
induction is left adjoint to the restriction functor
$\res_{r-1}^{r}$. It is also right adjoint; cf. \cite[Theorem 2.6]{DJ}.
The {\em Jucys-Murphy element} 
\begin{equation}\label{JMstupid}
\JM_r := S_{r-1} \cdots S_2 S_1 S_1 S_2 \cdots
S_{r-1} \in H_{r}
\end{equation} 
centralizes $H_{r-1}$, so left multiplication
by it defines an
endomorphism of the $(H_{r-1},
H_{r})$-bimodule $H_{r}$. For any $i \in \k$, let $i\text{-\!}\ind_{r-1}^{r}$
be the {\em $i$-induction functor} defined by tensoring with the
summand of this bimodule that arises as the generalized $i$-eigenspace
of this endomorphism. Let $i\text{-\!}\res_{r-1}^{r}$ be the biadjoint
{\em $i$-restriction functor}; explicitly, $i\text{-\!}\res_{r-1}^{r} M$
may be realized as the generalized $i$-eigenspace of $\JM_r$ on $\res^{r}_{r-1} M$. 

The following ``classical'' branching rules\footnote{Probably the best way to prove them
is by applying the ``Schur functor'' to an analogous result for quantum
$GL_n$.} describe the effect of
these functors on the Specht module $\SS_\lambda$. 
In formulating the result, we identify partition $\lambda$ with
its 
{\em Young diagram}, that is, the set
$\{(i,j)\:|\:i \geq 1, 1 \leq j \leq \lambda_r\}$, and
define the {\em content} of the {\em node} $\mathsf{A} = (i,j) \in \NN\times\NN$
from $\cont(\mathsf{A}) := q^{2(j-i)} \in \k$.
For example, here is the Young diagram $\lambda = (5,3,2)$ with
its nodes labeled by their contents:
$$
\diagram{$ 1$&$ q^2$&$
  q^4$&$ q^6$&$ q^8$\cr $
  q^{-2}$&$ 1$&$q^2$\cr
  $ q^{-4}$ & $ q^{-2}$ \cr}.
$$
Let $I_1$ be the set of all
possible contents of nodes of partitions. More generally, for any $c
\in \k^\times$, we let
\begin{equation}\label{Ic}
I_c := \{c q^{2n}\:|\:n \in \ZZ\}  \subseteq \k^\times.
\end{equation}

\begin{lemma}\label{classical}
The following hold for each $i \in \k^\times$:
\begin{enumerate}
\item For $\lambda \vdash (r-1)$, the $H_{r}$-module $i\text{-\!}\ind_{r-1}^{r} \SS_\lambda$ has a
multiplicity-free filtration
with sections $\SS_\mu$ for $\mu \vdash r$
obtained by adding a node
of content $i$ to the Young diagram of $\la$.
\item For $\lambda \vdash r$,
the $H_{r}$-module $i\text{-\!}\res^{r}_{r-1} \SS_\lambda$ has 
multiplicity-free filtration
with sections 
$\cong \SS_{\mu}$ for $\mu \vdash (r-1)$
obtained by removing a node of content $i$ from the Young diagram of $\la$.
\end{enumerate}
In both cases, the filtration should be ordered according to the
usual dominance ordering on the partitions labelling the sections,
most dominant at the bottom.
Hence:
\begin{equation}\label{decomp}
\ind_{r-1}^r = \bigoplus_{i \in I_1} 
i\text{-\!}\ind_{r-1}^r,
\qquad
\res_{r-1}^r = \bigoplus_{i \in I_1} i\text{-\!}\res_{r-1}^r.
\end{equation}
\end{lemma}

The results just explained extend immediately to
$H_{r,s} = H_r \otimes H_s$.
For these algebras, there are two commuting $i$-induction functors
$i\text{-\!}\ind_{r-1,s}^{r,s}$ and 
$i\text{-\!}\ind_{r,s-1}^{r,s}$, defined
by tensoring with the bimodules that arise by taking the generalized
$i$-eigenspaces of 
the endomorphisms of 
$H_{r,s}$ defined by left multiplication by $\JM_r \otimes
1$ or $1 \otimes \JM_s$, respectively.
The biadjoint $i$-restriction functors are denoted $i\text{-\!}\res^{r,s}_{r-1,s}$
 and $i\text{-\!}\res^{r,s}_{r,s-1}$.
Lemma~\ref{classical} extends in an obvious way to 
describe
the effect of these functors on the modules
$\SS_{\la^\up} \boxtimes \SS_{\la^\down}$.

The next step is to use the Morita equivalences from
Lemma~\ref{cartanlem} to transport the branching rules for $H_{r,s}$ just
described to the
algebra $OS^\circ$.
Let 
\begin{align}
\i^\circ_\down:&OS^\circ \rightarrow OS^\circ, \qquad f \mapsto \down\otimes f,\\
\i^\circ_\up:&OS^\circ\rightarrow OS^\circ, \qquad 
f \mapsto \up \otimes f
\end{align} 
be the algebra homomorphisms associated to
the functors $\down \otimes-:\OS^\circ(z,t) \rightarrow \OS^\circ(z,t)$ and $\up\otimes-:\OS^\circ(z,t)
\rightarrow \OS^\circ(z,t)$.
These are not {\em locally unital} algebra homomorphisms:
they send the idempotent $1_\a$ to $1_{\up \a}$ and to $1_{\down \a}$, respectively.
Then let
\begin{align}\label{dib1}
{_\up} OS^\circ &:= \bigoplus_{\a,\b \in \words} 1_{\up \a} OS^\circ 1_\b,
&
OS^\circ_\up &:= \bigoplus_{\a,\b \in \words} 1_\a OS^\circ 1_{\up \b},\\
{_\down} OS^\circ &:= \bigoplus_{\a,\b \in \words} 1_{\down \a} OS^\circ 1_\b,&
OS^\circ_\down &:= \bigoplus_{\a,\b \in \words} 1_\a OS^\circ 1_{\down \b},\label{dib2}
\end{align}
which we view as $(OS^\circ, OS^\circ)$-bimodules
with left and right actions of $a,b \in OS^\circ$ on $f$ defined by $a\cdot
f\cdot b :=
\i^\circ_\up(a) f b,
a f \i^\circ_\up(b), \i^\circ_\down(a) f b$ and $a f \i^\circ_\down(b)$,
respectively.
Tensoring with these bimodules give us four endofunctors of $\Mod{OS^\circ}$:
\begin{align}\label{fan1}
E^\up := ? \otimes_{OS^\circ} {_\up} OS^\circ&:\Mod{OS^\circ} \rightarrow
\Mod{OS^\circ},\\
F^\up := ?\otimes_{OS^\circ} OS^\circ_\up&:\Mod{OS^\circ}\rightarrow \Mod{OS^\circ},\\
F^\down := ? \otimes_{OS^\circ} {_\down} OS^\circ&:\Mod{OS^\circ} \rightarrow
\Mod{OS^\circ},\\
E^\down := ?\otimes_{OS^\circ} OS^\circ_\down&:\Mod{OS^\circ}\rightarrow
\Mod{OS^\circ}.\label{fan4}
\end{align}
The functors $E^\up$ and $F^\down$ 
send $OS^\circ_{r,s}$-modules to $OS^\circ_{r+1,s}$- and
$OS^\circ_{r,s+1}$-modules, respectively; they will be called {\em
  induction functors}.
The functors $F^\up$ and $E^\down$
send $OS^\circ_{r,s}$-modules to $OS^\circ_{r-1,s}$- and
$OS^\circ_{r,s-1}$-modules, respectively; they will be called {\em
  restriction functors}.
This terminology is justified by the following lemma.

\begin{lemma}\label{fan}
The following diagrams commute up to natural isomorphisms:
\begin{align*}
&\begin{CD}
\Mod H_{r+1,s} &@>\Upsilon_{r+1,s}>>&\Mod OS_{r+1,s}^\circ\\
@AA\ind_{r,s}^{r+1,s}A \searrow^{\!\!\!\alpha}&@AAE^\up A\\
\Mod H_{r,s}&@>>\Upsilon_{r,s}>&\Mod OS_{r,s}^\circ,
\end{CD}
&\qquad
&\begin{CD}
\Mod H_{r+1,s} &@>\Upsilon_{r+1,s}>>&\Mod OS_{r+1,s}^\circ\\
@VV\res_{r,s}^{r+1,s}V ^{\beta\!\!\!\!\!\!}\nearrow&@VVF^\up V\\
\Mod H_{r,s}&@>>\Upsilon_{r,s}>&\Mod OS_{r,s}^\circ,
\end{CD}\\
&\begin{CD}
\Mod H_{r,s+1} &@>\Upsilon_{r,s+1}>>&\Mod OS_{r,s+1}^\circ\\
@AA\ind_{r,s}^{r,s+1}A \searrow^{\!\!\!\gamma}&@AAF^\down A\\
\Mod H_{r,s}&@>>\Upsilon_{r,s}>&\Mod OS_{r,s}^\circ,
\end{CD}
&\qquad
&\begin{CD}
\Mod H_{r,s+1} &@>\Upsilon_{r,s+1}>>&\Mod OS_{r,s+1}^\circ\\
@VV\res_{r,s}^{r,s+1}V ^{\delta\!\!\!\!\!}\nearrow&@VVE^\down V\\
\Mod H_{r,s}&@>>\Upsilon_{r,s}>&\Mod OS_{r,s}^\circ.
\end{CD}
\end{align*}
Hence, the functors $E^\up$ and $F^\up$ are biadjoint, as are the
functors
$E^\down$ and $F^\down$.
\end{lemma}

\begin{proof}
First we construct the 
isomorphism $\alpha:\Upsilon_{r+1,s} \circ \ind_{r,s}^{r+1,s}
\stackrel{\sim}{\rightarrow} E^\up \circ \Upsilon_{r,s}$.
The northwest functor is defined by tensoring over $H_{r,s}$
with the $(H_{r,s},
OS^\circ)$-bimodule
$$
H_{r+1,s} \otimes_{H_{r+1,s}} 1_{\down^s \up^{r+1}} OS^\circ
\cong 1_{ \down^s\up^{r+1}} OS^\circ.
$$
The southeast functor is defined by tensoring with
$$
1_{\down^s \up^r} OS^\circ \otimes_{OS^\circ}
{_\up} OS^\circ
\cong
1_{\up \down^s \up^r} OS^\circ.
$$ 
The following gives an isomorphism between
these two bimodules:
\begin{align*}
1_{\down^s \up^{r+1}} OS^\circ
&\stackrel{\sim}{\rightarrow}
1_{\up\down^s\up^r} OS^\circ,
\\
\mathord{
\begin{tikzpicture}[baseline=2.5mm]
\draw[darkblue,thick,yshift=-5pt,xshift=-5pt] (0,.3) rectangle ++(10pt,10pt);
\node at (0,.3) {$\scriptstyle f$};
\draw[-,thick,double,darkblue] (0,.12) to (0,-.3);
\draw[<-,ultra thick,darkblue] (-.12,.47) to (-.32,1.09);
\draw[->,thick,darkblue] (0,.47) to (0,1.09);
\draw[->, ultra thick,darkblue] (.12,.47) to (.36,1.09);
\node at (0,-.42) {$\a$};
\node at (-.31,1.2) {$\scriptstyle s$};
\node at (.35,1.2) {$\scriptstyle r$};
\end{tikzpicture}
}
&\mapsto 
\mathord{
\begin{tikzpicture}[baseline=2.5mm]
\draw[->,thick,darkblue] (0,.47) to [out=60,in=-50] (-.32,1.09);
\draw[-,line width=3pt,white] (-.12,.47) to [out=100,in=-140] (0.05,1.09);
\draw[<-,ultra thick,darkblue] (-.12,.47) to [out=100,in=-140] (0.05,1.09);
\draw[->,ultra thick,darkblue] (.12,.47) to (.36,1.09);
\node at (0,-.42) {$\a$};
\node at (0.05,1.2) {$\scriptstyle s$};
\node at (.36,1.2) {$\scriptstyle r$};
\draw[darkblue,thick,yshift=-5pt,xshift=-5pt] (0,.3) rectangle ++(10pt,10pt);
\node at (0,.3) {$\scriptstyle f$};
\draw[-, thick,double,darkblue] (0,.12) to (0,-.3);
\end{tikzpicture}
}
\end{align*}
for any $f \in 1_{\down^s\up^{r+1}} OS^\circ 1_\a$.
This establishes the existence of $\alpha$.

To deduce the existence of $\beta$,
we claim that $F^\up$ is right adjoint to $E^\up$.
To see this, $F^\up M = M \otimes_{OS^\circ} OS^\circ_\up
= \bigoplus_{\a \in \words} M 1_{\up \a}
\cong 
\bigoplus_{\a \in \words} \Hom_{OS^\circ}(1_{\up \a} OS^\circ, M)
$. So, by adjointness of tensor and hom, $F^\up$ is right adjoint
to
$\bigoplus_{\a \in \words} ? \otimes_{OS^\circ} 1_{\up \a} OS^\circ = ?
\otimes_{OS^\circ} {_\up} OS^\circ = E^\up$.
Since $\res^{r+1,s}_{r,s}$ is right adjoint to $\ind^{r+1,s}_{r,s}$,
and the horizontal functors in our diagrams are equivalences of
categories, we can now deduce the existence of the desired isomorphism $\beta$ using the previous
paragraph and unicity of right adjoints.

The construction of $\gamma$ is very similar to that of $\alpha$. In
fact, it is even easier since both of the $(H_{r,s}, OS^\circ)$-bimodules
being considered turn out to be the same bimodule
$1_{\down^{s+1}\up^r} OS^\circ$, so we can take $\gamma$ to be induced by
the identity map.
Then we get $\delta$ from $\gamma$ as in the previous paragraph.
\end{proof}

Now we need versions of Jucys-Murphy elements for $OS^\circ$, which
extend the Jucys-Murphy elements of $H_{r,s}$.
For $\varnothing \neq \b \in \words$, define $X^\circ(\b) \in 1_\b OS^\circ 1_\b$ by setting
$X^\circ(\up) := 1_\up$,
$X^\circ(\down) := t^{-2} 1_\down$,
and then recursively defining
\begin{align}\label{chile2}
X^\circ(\up \up \b) &:= 
\mathord{
\begin{tikzpicture}[baseline=2.5mm]
\draw[->,thick,darkblue] (0,.47) to [out=90,in=-90] (-.42,.9);
\draw[-,line width=4pt,white] (-.42,.27) to [out=90,in=-120] (-0.05,.9); 
\draw[->,thick,darkblue] (-.42,.27) to [out=90,in=-120] (-0.05,.9);
\draw[-,thick,darkblue] (-.42,.27) to [out=-90,in=120] (-0.05,-.3);
\draw[-,line width=4pt,white] (0,.12) to [out=-90,in=90] (-.42,-.3);
\draw[-,thick,darkblue] (0,.12) to [out=-90,in=90] (-.42,-.3);
\node at (0.16,.3) {$\scriptstyle X^\circ(\up \b)$};
\draw[-, thick,double,darkblue] (.25,.12) to (.25,-.3);
\draw[-, thick,double,darkblue] (.25,.48) to (.25,.9);
\draw[darkblue,thick,yshift=-5pt,xshift=-5pt] (-.1,.3) rectangle ++(26pt,10pt);
\node at (0.25,-.43) {$\b$};
\node at (0.25,1.07) {$\b$};
\end{tikzpicture}}\:,
&
X^\circ(\up \down \b) &:= 
\mathord{
\begin{tikzpicture}[baseline=2.5mm]
\draw[->,thick,darkblue] (0,.47) to [out=90,in=-90] (-.42,.9);
\draw[-,line width=4pt,white] (-.42,.27) to [out=90,in=-120] (-0.05,.9); 
\draw[-,thick,darkblue] (-.42,.27) to [out=90,in=-120] (-0.05,.9);
\draw[darkblue,thick,yshift=-5pt,xshift=-5pt] (-.1,.3) rectangle ++(26pt,10pt);
\draw[-,thick,darkblue] (0,.12) to [out=-90,in=90] (-.42,-.3);
\draw[-,line width=4pt,white] (-.42,.27) to [out=-90,in=120] (-0.05,-.3);
\draw[->,thick,darkblue] (-.42,.27) to [out=-90,in=120] (-0.05,-.3);
\node at (0.16,.3) {$\scriptstyle X^\circ(\up \b)$};
\draw[-, thick,double,darkblue] (.25,.12) to (.25,-.3);
\draw[-, thick,double,darkblue] (.25,.48) to (.25,.9);
\node at (0.25,-.43) {$\b$};
\node at (0.25,1.07) {$\b$};
\end{tikzpicture}}\:,\\
X^\circ(\down \down \b) &:= 
\mathord{
\begin{tikzpicture}[baseline=2.5mm]
\draw[-,thick,darkblue] (-.42,.27) to [out=90,in=-120] (-0.05,.9);
\draw[-,line width=4pt,white] (0,.47) to [out=90,in=-90] (-.42,.9);
\draw[-,thick,darkblue] (0,.47) to [out=90,in=-90] (-.42,.9);
\draw[darkblue,thick,yshift=-5pt,xshift=-5pt] (-.1,.3) rectangle ++(26pt,10pt);
\draw[->,thick,darkblue] (0,.12) to [out=-90,in=90] (-.42,-.3);
\draw[-,line width=4pt,white] (-.42,.27) to [out=-90,in=120] (-0.05,-.3);
\draw[->,thick,darkblue] (-.42,.27) to [out=-90,in=120] (-0.05,-.3);
\node at (0.16,.3) {$\scriptstyle X^\circ(\down \b)$};
\draw[-, thick,double,darkblue] (.25,.12) to (.25,-.3);
\draw[-, thick,double,darkblue] (.25,.48) to (.25,.9);
\node at (0.25,-.43) {$\b$};
\node at (0.25,1.07) {$\b$};
\end{tikzpicture}}\:,
&
X^\circ(\down \up \b) &:= 
\mathord{
\begin{tikzpicture}[baseline=2.5mm]
\draw[->,thick,darkblue] (-.42,.27) to [out=90,in=-120] (-0.05,.9);
\draw[-,line width=4pt,white] (0,.47) to [out=90,in=-90] (-.42,.9);
\draw[-,thick,darkblue] (0,.47) to [out=90,in=-90] (-.42,.9);
\draw[darkblue,thick,yshift=-5pt,xshift=-5pt] (-.1,.3) rectangle ++(26pt,10pt);
\draw[-,thick,darkblue] (-.42,.27) to [out=-90,in=120] (-0.05,-.3);
\draw[-,line width=4pt,white] (0,.12) to [out=-90,in=90] (-.42,-.3);
\draw[->,thick,darkblue] (0,.12) to [out=-90,in=90] (-.42,-.3);
\node at (0.16,.3) {$\scriptstyle X^\circ(\down \b)$};
\draw[-, thick,double,darkblue] (.25,.12) to (.25,-.3);
\draw[-, thick,double,darkblue] (.25,.48) to (.25,.9);
\node at (0.25,-.43) {$\b$};
\node at (0.25,1.07) {$\b$};
\end{tikzpicture}}\:,\label{chile5}
\end{align}
for any $\a \in \words$; cf. (\ref{JMa})--(\ref{JMb}).

Let ${_{\up}}X^\circ:{_\up} OS^\circ \rightarrow {_\up} OS^\circ$
and $X^\circ_{\up}:OS^\circ_\up \rightarrow OS^\circ_\up$
be the linear endomorphisms defined on $1_{\up \a} OS^\circ$ 
or
$OS^\circ 1_{\up \a}$ 
by
left or right multiplication by $X^\circ(\up \a)$, respectively.
Similarly, replacing $\up$ by $\down$ everywhere, we define
linear endomorphisms 
${_{\down}}X^\circ$ and $X^\circ_{\down}$ of ${_\down} OS^\circ$ and $OS^\circ_\down$.

\begin{lemma}\label{chile0}
All of the linear endomorphisms ${_{\up}}X^\circ, X^\circ_{\up}, {_{\down}}X^\circ$ and $X^\circ_{\down}$
are $(OS^\circ, OS^\circ)$-bimodule endomorphisms.
\end{lemma}

\begin{proof}
We just explain the argument for ${_{\up}}X^\circ$. It is obvious that this
defines a right $OS^\circ$-module homomorphism. To see that it also
commutes with the left action of $OS^\circ$, it suffices to check that it
commutes with left multiplication by
any element of $OS^\circ$ defined by
a crossing of a neighboring pairs of strands (excluding the leftmost strand).
This quickly reduces by induction to checking the following four
identities:
\begin{align*}
X^\circ(\up\up\up) \circ \:
\mathord{
\begin{tikzpicture}[baseline = -.5mm]
	\draw[->,thick,darkblue] (-0.6,-.3) to (-0.6,.4);
	\draw[->,thick,darkblue] (0.28,-.3) to (-0.28,.4);
 	\draw[line width=4pt,white,-] (-0.28,-.3) to (0.28,.4);
	\draw[thick,darkblue,->] (-0.28,-.3) to (0.28,.4);
\end{tikzpicture}
}
&= 
\mathord{
\begin{tikzpicture}[baseline = -.5mm]
	\draw[->,thick,darkblue] (-0.6,-.3) to (-0.6,.4);
	\draw[->,thick,darkblue] (0.28,-.3) to (-0.28,.4);
 	\draw[line width=4pt,white,-] (-0.28,-.3) to (0.28,.4);
	\draw[thick,darkblue,->] (-0.28,-.3) to (0.28,.4);
\end{tikzpicture}
}
\circ X^\circ(\up\up\up),
&
X^\circ(\up\down\down) \circ \:
\mathord{
\begin{tikzpicture}[baseline = -.5mm]
	\draw[->,thick,darkblue] (-0.6,-.3) to (-0.6,.4);
	\draw[<-,thick,darkblue] (0.28,-.3) to (-0.28,.4);
 	\draw[line width=4pt,white,-] (-0.28,-.3) to (0.28,.4);
	\draw[thick,darkblue,<-] (-0.28,-.3) to (0.28,.4);
\end{tikzpicture}
}
&= 
\mathord{
\begin{tikzpicture}[baseline = -.5mm]
	\draw[->,thick,darkblue] (-0.6,-.3) to (-0.6,.4);
	\draw[<-,thick,darkblue] (0.28,-.3) to (-0.28,.4);
 	\draw[line width=4pt,white,-] (-0.28,-.3) to (0.28,.4);
	\draw[thick,darkblue,<-] (-0.28,-.3) to (0.28,.4);
\end{tikzpicture}
}
\circ X^\circ(\up\down\down),\\
X^\circ(\up\down\up) \circ \:
\mathord{
\begin{tikzpicture}[baseline = -.5mm]
	\draw[->,thick,darkblue] (-0.6,-.3) to (-0.6,.4);
	\draw[thick,darkblue,->] (-0.28,-.3) to (0.28,.4);
	\draw[-,line width=4pt,white] (0.28,-.3) to (-0.28,.4);
	\draw[<-,thick,darkblue] (0.28,-.3) to (-0.28,.4);
\end{tikzpicture}
}
&= 
\mathord{
\begin{tikzpicture}[baseline = -.5mm]
	\draw[->,thick,darkblue] (-0.6,-.3) to (-0.6,.4);
	\draw[thick,darkblue,->] (-0.28,-.3) to (0.28,.4);
	\draw[-,line width=4pt,white] (0.28,-.3) to (-0.28,.4);
	\draw[<-,thick,darkblue] (0.28,-.3) to (-0.28,.4);
\end{tikzpicture}
}
\circ X^\circ(\up\up\down),
&
X^\circ(\up\up\down) \circ \:
\mathord{
\begin{tikzpicture}[baseline = -.5mm]
	\draw[->,thick,darkblue] (-0.6,-.3) to (-0.6,.4);
	\draw[->,thick,darkblue] (0.28,-.3) to (-0.28,.4);
 	\draw[line width=4pt,white,-] (-0.28,-.3) to (0.28,.4);
	\draw[thick,darkblue,<-] (-0.28,-.3) to (0.28,.4);
\end{tikzpicture}
}
&= 
\mathord{
\begin{tikzpicture}[baseline = -.5mm]
	\draw[->,thick,darkblue] (-0.6,-.3) to (-0.6,.4);
	\draw[->,thick,darkblue] (0.28,-.3) to (-0.28,.4);
 	\draw[line width=4pt,white,-] (-0.28,-.3) to (0.28,.4);
	\draw[thick,darkblue,<-] (-0.28,-.3) to (0.28,.4);
\end{tikzpicture}
}
\circ X^\circ(\up\down\up).
\end{align*}
These are all straightforward on drawing the diagrams for these $X^\circ$'s explicity.
\end{proof}

For $i \in \k^\times$,
let  $E_i^\up$ be the subfunctor of $E$ that is defined by tensoring
with the bimodule that is the generalized $i$-eigenspace of
${_{\up}}X^\circ:{_\up} OS^\circ \rightarrow {_\up} OS^\circ$.
Define $F_i^\up, F_i^\down$ and $E_i^\down$ similarly
using the endomorphisms $X^\circ_{\up}, {_{\down}}X^\circ$ and $X^\circ_{\down}$.

\begin{lemma}\label{fur}
The following diagrams commute up to natural isomorphisms:
\begin{align*}
&\begin{CD}
\Mod H_{r+1,s} &@>\Upsilon_{r+1,s}>>&\Mod OS_{r+1,s}^\circ\\
@AAi\text{-\!}\ind_{r,s}^{r+1,s}A &&@AAE_i^{\up} A\\
\Mod H_{r,s}&@>>\Upsilon_{r,s}>&\Mod OS_{r,s}^\circ,
\end{CD}
&\qquad
&\begin{CD}
\Mod H_{r+1,s} &@>\Upsilon_{r+1,s}>>&\Mod OS_{r+1,s}^\circ\\
@VV i\text{-\!}\res_{r,s}^{r+1,s}V& &@VVF_i^{\up} V\\
\Mod H_{r,s}&@>>\Upsilon_{r,s}>&\Mod OS_{r,s}^\circ,
\end{CD}\\
&\begin{CD}
\Mod H_{r,s+1} &@>\Upsilon_{r,s+1}>>&\Mod OS_{r,s+1}^\circ\\
@AA t^{-2}i^{-1}\text{-\!}\ind_{r,s}^{r,s+1}A &&@AAF_i^\down A\\
\Mod H_{r,s}&@>>\Upsilon_{r,s}>&\Mod OS_{r,s}^\circ,
\end{CD}
&\qquad
&\begin{CD}
\Mod H_{r,s+1} &@>\Upsilon_{r,s+1}>>&\Mod OS_{r,s+1}^\circ\\
@VVt^{-2}i^{-1}\text{-\!}\res_{r,s}^{r,s+1}V &&@VVE_i^{\down} V\\
\Mod H_{r,s}&@>>\Upsilon_{r,s}>&\Mod OS_{r,s}^\circ.
\end{CD}
\end{align*}
Hence, the functors $E_i^{\up}$ and $F_i^{\up}$ are biadjoint, as are the
functors
$E_i^{\down}$ and $F_i^{\down}$.
Moreover:
\begin{equation}\label{decomp1}
E^\up = \bigoplus_{i \in I_1} E_i^\up,
\quad
F^\up = \bigoplus_{i \in I_1} F_i^\up,\quad
F^\down = \bigoplus_{i \in I_{t^{-2}}} F_i^\down,
\quad
E^\down = \bigoplus_{i \in I_{t^{-2}}} E_i^\down.
\end{equation}
\end{lemma}

\begin{proof}
Consider the first diagram.
In the proof of Lemma~\ref{fan}, the isomorphism of functors
$\alpha$ was induced by an explicit bimodule isomorphism
$1_{\down^s \up^{r+1}} OS^\circ \rightarrow 1_{\up \down^s \up^r} OS^\circ$.
This isomorphism intertwines the endomorphism of
$1_{\down^s \up^{r+1}} OS^\circ$ 
defined by left multiplication by $\i_{r+1,s}(\JM_{r+1} \otimes 1)$ with
the endomorphism of 
$1_{\up \down^s \up^r} OS^\circ$ defined by left multiplication by $X^\circ(\up
  \down^s \up^r)$; the appropriate picture needed to see this is as
follows:
$$
\mathord{
\begin{tikzpicture}[baseline =0mm]
	\draw[->,thick,darkblue] (0.6,-0.1) to[out=90,in=-90] (-.6,.5);
	\draw[-,line width=4pt,white] (-.3,-.5) to (-.3,.5);
	\draw[-,line width=4pt,white] (.4,-.5) to (.4,.5);
	\draw[<-,ultra thick,darkblue] (-.3,-.5) to (-.3,.5);
	\draw[->,ultra thick,darkblue] (.4,-.5) to (.4,.5);
	\draw[-,line width=4pt,white] (0,-.5) to[out=90,in=-90] (.6,-.1);
	\draw[-,thick,darkblue] (0,-.5) to[out=90,in=-90] (.6,-.1);
      \node at (-.3,.65) {$\scriptstyle s$};
      \node at (.4,.65) {$\scriptstyle r$};
\end{tikzpicture}
}
=
\mathord{
\begin{tikzpicture}[baseline =0mm]
	\draw[->,thick,darkblue] (0.6,0) to[out=90,in=-90] (-.5,.5);
	\draw[-,thick,darkblue] (0.1,-.5) to[out=80,in=-90] (-.6,-.1);
	\draw[-,thick,darkblue] (-.6,-.1) to[out=90,in=160] (.1,0);
	\draw[-,line width=4pt,white] (-.3,-.5) to (-.3,.5);
	\draw[-,line width=4pt,white] (.4,-.5) to (.4,.5);
	\draw[<-,ultra thick,darkblue] (-.3,-.5) to (-.3,.5);
	\draw[->,ultra thick,darkblue] (.4,-.5) to (.4,.5);
	\draw[-,line width=4pt,white] (.1,0) to[out=-20,in=-90] (.6,0);
	\draw[-,thick,darkblue] (.1,0) to[out=-20,in=-90] (.6,0);
      \node at (-.3,.65) {$\scriptstyle s$};
      \node at (.4,.65) {$\scriptstyle r$};
\end{tikzpicture}
}
\:.
$$
Consequently, this bimodule isomorphism restricts to an isomorphism
between the appropriate
summands of these bimodules, showing that the
restriction of $\alpha$ gives the desired natural transformation.

The second diagram follows from the first by unicity of adjoints on
observing that
$F_i^\up$ is right adjoint to $E_i^\up$, which follows from the explicit construction of the adjunction in the
second paragraph of the proof of Lemma~\ref{fan}.

The third diagram 
is established in the same way as the first diagram.
One needs to check that the endomorphisms of
$1_{\down^{s+1} \up^{r}} OS^\circ$ 
defined by left multiplication by $t^{-2} \i_{r,s+1}(1 \otimes
\JM_{s+1})^{-1}$
and by $X^\circ(\down^{s+1} \up^r)$ are equal, which is clear from the
following picture:
$$
t^{-2}\left(\mathord{
\begin{tikzpicture}[baseline =0mm]
	\draw[->,ultra thick,darkblue] (.8,-.5) to (.8,.5);
	\draw[-,thick,darkblue] (.6,0) to[out=90,in=-90] (0,.5);
	\draw[-,line width=4pt,white] (.3,-.5) to (.3,.5);
	\draw[<-,ultra thick,darkblue] (.3,-.5) to (.3,.5);
	\draw[-,line width=4pt,white] (0,-.5) to[out=90,in=-90] (.6,0);
	\draw[<-,thick,darkblue] (0,-.5) to[out=90,in=-90] (.6,0);
      \node at (.8,.65) {$\scriptstyle r$};
      \node at (.3,.65) {$\scriptstyle s$};
\end{tikzpicture}}\right)^{-1}=
 t^{-2} \mathord{
\begin{tikzpicture}[baseline =0mm]
	\draw[->,ultra thick,darkblue] (.7,-.5) to (.7,.5);
	\draw[-,line width=4pt,white] (0,-.5) to[out=90,in=-90] (.9,0);
	\draw[<-,thick,darkblue] (0,-.5) to[out=90,in=-90] (.9,0);
	\draw[-,line width=4pt,white] (.3,-.5) to (.3,.5);
	\draw[<-,ultra thick,darkblue] (.3,-.5) to (.3,.5);
	\draw[-,line width=4pt,white] (.9,0) to[out=90,in=-90] (0,.5);
	\draw[-,thick,darkblue] (.9,0) to[out=90,in=-90] (0,.5);
      \node at (.7,.65) {$\scriptstyle r$};
      \node at (.3,.65) {$\scriptstyle s$};
\end{tikzpicture}
}
\:.
$$

The fourth  diagram follows by adjunction as before.

The final statement of the lemma follows using these diagrams plus facts we have
already discussed about the induction and restriction functors for
$H_{r,s}$.
\end{proof}

We assemble the results so far into the following theorem, which describes
all of the branching rules for the functors
$E_i^\up$, $F_i^\up$,
$F_i^\down$
and $E_i^\down$.

\begin{lemma}\label{newfirst}
The following hold for $i \in \k^\times$
and $\LA \in \Par_{r,s}$.
\begin{enumerate}
\item $E_i^\up \SS(\LA)$ has a multiplicity-free filtration
with sections $\SS(\MU)$
for $\MU \in \Par_{r+1,s}$ obtained by adding a node of
content $i$ to the Young diagram of $\la^\up$.
\item $F_i^\up \SS(\LA)$ has a multiplicity-free filtration
with sections $\SS(\MU)$
for $\MU \in \Par_{r-1,s}$ obtained by removing a node of
content $i$ from the Young diagram of $\la^\up$.
\item $F_i^\down \SS(\LA)$ has a multiplicity-free filtration
with sections $\SS(\MU)$
for $\MU \in \Par_{r,s+1}$ obtained by adding a node of
content $t^{-2} i^{-1}$ to the Young diagram of $\la^\down$.
\item $E_i^\down \SS(\LA)$ has a multiplicity-free filtration
with sections $\SS(\MU)$
for $\MU \in \Par_{r,s-1}$ obtained by removing a node of
content $t^{-2} i^{-1}$ from the Young diagram of $\la^\down$.
\end{enumerate}
In all cases, the filtrations should be ordered according to the
usual dominance ordering on the partitions labelling the sections,
most dominant at the bottom.
\end{lemma}

\begin{proof}
This follows from Lemmas~\ref{fur} and \ref{classical}.
\end{proof}

Now we turn our attention at last to $OS$ itself.
Mimicking the definitions made above for $OS^\circ$,
we  write $\i_\down:OS \rightarrow OS$ and
$\i_\up:OS\rightarrow OS$ for the algebra homomorphisms associated to
the functors $\down \otimes-:\OS(z,t) \rightarrow \OS(z,t)$ and $\up\otimes-:\OS(z,t)
\rightarrow \OS(z,t)$. 
Then let
\begin{align}\label{dob1}
{_\up} OS &:= \bigoplus_{\a,\b \in \words} 1_{\up \a} OS 1_\b,
&
OS_\up &:= \bigoplus_{\a,\b \in \words} 1_\a OS 1_{\up \b},\\
{_\down} OS &:= \bigoplus_{\a,\b \in \words} 1_{\down \a} OS 1_\b,&
OS_\down &:= \bigoplus_{\a,\b \in \words} 1_\a OS 1_{\down \b},\label{dob2}
\end{align}
which we view as $(OS, OS)$-bimodules
with left and right actions of $a,b \in OS$ on $f$ defined by $a\cdot
f\cdot b :=
\i_\up(a) f b,
a f \i_\up(b), \i_\down(a) f b$ and $a f \i_\down(b)$, respectively.
A key difference to the situation for $OS^\circ$ emerges right away:

\begin{lemma}\label{easy}
We have that $OS_\up \cong {_\down}OS$ and $OS_\down \cong {_\up}OS$
as $(OS, OS)$-bimodules.
\end{lemma}

\begin{proof}
The mutually inverse bimodule isomorphisms
$OS_\down \rightarrow {_\up} OS$ and ${_\up} OS \rightarrow OS_\down$
are defined on diagrams by the maps
$$
\mathord{
\begin{tikzpicture}[baseline=2.5mm]
\draw[darkblue,thick,yshift=-5pt,xshift=-5pt] (0,.3) rectangle ++(10pt,10pt);
\node at (0,.3) {$\scriptstyle f$};
\draw[->, thick,darkblue] (-.12,.12) to (-.12,-.3);
\draw[-, thick,double,darkblue] (.12,.12) to (.12,-.3);
\draw[-, thick,double,darkblue] (0,.47) to (0,.89);
\node at (0,1.02) {$\a$};
\node at (.12,-.44) {$\b$};
\end{tikzpicture}
}
\mapsto
\mathord{
\begin{tikzpicture}[baseline=2.5mm]
\draw[darkblue,thick,yshift=-5pt,xshift=-5pt] (0,.3) rectangle ++(10pt,10pt);
\node at (0,.3) {$\scriptstyle f$};
\draw[-, thick,darkblue] (-.22,-.2) to [out=0,in=-90] (-.12,.12);
\draw[<-, thick,darkblue] (-.3,.89) to [out=-90,in=180] (-.22,-.2);
\draw[-, thick,double,darkblue] (.12,.12) to (.12,-.3);
\draw[-, thick,double,darkblue] (0,.47) to (0,.89);
\node at (0,1.02) {$\a$};
\node at (.12,-.44) {$\b$};
\end{tikzpicture}
}
\qquad\text{and}\qquad
\mathord{
\begin{tikzpicture}[baseline=2.5mm]
\draw[darkblue,thick,yshift=-5pt,xshift=-5pt] (0,.3) rectangle ++(10pt,10pt);
\node at (0,.3) {$\scriptstyle f$};
\draw[-, thick,double,darkblue] (0,.12) to (0,-.3);
\draw[->,thick,darkblue] (-.12,.47) to (-.12,.89);
\draw[-, thick,double,darkblue] (.12,.47) to (.12,.89);
\node at (.12,1.02) {$\a$};
\node at (0,-.44) {$\b$};
\end{tikzpicture}
}
\mapsto
\mathord{
\begin{tikzpicture}[baseline=2.5mm]
\draw[darkblue,thick,yshift=-5pt,xshift=-5pt] (0,.3) rectangle ++(10pt,10pt);
\node at (0,.3) {$\scriptstyle f$};
\draw[-,thick,darkblue] (-.12,.47) to [out=90,in=0] (-.22,.8);
\draw[<-,thick,darkblue] (-.3,-.3) to [out=90,in=180] (-.22,.8);
\draw[-, thick,double,darkblue] (0,.12) to (0,-.3);
\draw[-, thick,double,darkblue] (.12,.47) to (.12,.89);
\node at (.12,1.02) {$\a$};
\node at (0,-.44) {$\b$};
\end{tikzpicture}
},
$$
respectively. The isomorphism $OS_\up \cong {_\down} OS$ is
constructed similarly.
\end{proof}

This means that we only need to define {\em two} functors:
\begin{align}\label{edef}
E:= ?\otimes_{OS} {_\up}OS
\,\cong\:
 ?\otimes_{OS} OS_{\down} &:\Mod OS \rightarrow \Mod OS,\\
 F:= ?\otimes_{OS} {_\down}OS\,\cong\: ?\otimes_{OS} OS_{\up} &:\Mod OS \rightarrow \Mod OS.
\end{align}
Note that
\begin{equation}\label{theee}
E (1_\a OS) = 1_\a OS \otimes_{OS} ({_\up} OS) = 1_\a({_\up}
OS) = 1_{\up \a} OS,
\end{equation} and
similarly $F(1_\a OS) = 1_{\down \a} OS$.
By adjointness of tensor and hom, the functor $E$ has a canonical 
right adjoint
\begin{equation}\label{Ra}
\bigoplus_{\a \in \words} \Hom_{OS}(1_{\up \a}OS, ?)
\cong \:? \otimes_{OS} OS_{\up}, 
\end{equation}
with the isomorphism here
sending $f$ 
in the $\a$th summand to 
$f(1_{\up \a}) \otimes 1_{\up \a}$.
In view of Lemma~\ref{easy}, the functor
$? \otimes_{OS} OS_{\up}$ on the right hand side of (\ref{Ra}) is isomorphic to $F$.
Thus, we see that $(E,F)$ is an adjoint pair.
Explicitly, the unit $\operatorname{Id} \rightarrow FE$ 
of this adjunction is 
defined by the bimodule homomorphism
\begin{align}\label{unit}
&OS \rightarrow {_\up}OS\otimes_{OS}{_\down}OS,&
&\mathord{
\begin{tikzpicture}[baseline=2.5mm]
\draw[darkblue,thick,yshift=-5pt,xshift=-5pt] (0,.3) rectangle ++(10pt,10pt);
\node at (0,.3) {$\scriptstyle f$};
\draw[-, thick,double,darkblue] (0,.12) to (0,-.1);
\draw[-, thick,double,darkblue] (0,.47) to (0,.69);
\node at (0,.85) {$\a$};
\node at (0,-.27) {$\b$};
\end{tikzpicture}
}
\mapsto
\mathord{
\begin{tikzpicture}[baseline=2.5mm]
\draw[darkblue,thick,yshift=-5pt,xshift=-5pt] (0,.3) rectangle ++(10pt,10pt);
\node at (0,.3) {$\scriptstyle f$};
\draw[->, thick,darkblue] (-.35,-.1) to (-.35,.69);
\draw[-, thick,double,darkblue] (0,.12) to (0,-.1);
\draw[-, thick,double,darkblue] (0,.47) to (0,.69);
\node at (0,.85) {$\a$};
\node at (0,-.27) {$\b$};
\end{tikzpicture}
}
\otimes
\mathord{
\begin{tikzpicture}[baseline=2.5mm]
\draw[-, thick,darkblue] (-.5,.4) to [out=180,in=-90] (-.7,.69);
\draw[->, thick,darkblue] (-.5,.4) to [out=0,in=-90] (-.3,.69);
\draw[-, thick,double,darkblue] (-0.05,-.1) to (-0.05,.69);
\node at (-0.05,-.27) {$\b$};
\end{tikzpicture}
},
\\\intertext{and the counit
$EF \rightarrow \operatorname{Id}$ is defined by}
&{_\down}OS \otimes_{OS} {_\up} OS
\rightarrow OS,
&&
\mathord{
\begin{tikzpicture}[baseline=2.5mm]
\draw[darkblue,thick,yshift=-5pt,xshift=-5pt] (0,.3) rectangle ++(10pt,10pt);
\node at (0,.3) {$\scriptstyle f$};
\draw[-, thick,double,darkblue] (0,.12) to (0,-.3);
\draw[<-,thick,darkblue] (-.12,.47) to (-.12,.89);
\draw[-, thick,double,darkblue] (.12,.47) to (.12,.89);
\node at (.12,1.05) {$\a$};
\node at (0,-.45) {$\b$};
\end{tikzpicture}
}\otimes\:\,
\mathord{
\begin{tikzpicture}[baseline=2.5mm]
\draw[darkblue,thick,yshift=-5pt,xshift=-5pt] (0,.3) rectangle ++(10pt,10pt);
\node at (0,.3) {$\scriptstyle g$};
\draw[-, thick,double,darkblue] (0,.12) to (0,-.3);
\draw[->,thick,darkblue] (-.12,.47) to (-.12,.89);
\draw[-, thick,double,darkblue] (.12,.47) to (.12,.89);
\node at (.12,1.07) {$\b$};
\node at (0,-.45) {$\c$};
\end{tikzpicture}
}
\mapsto
\mathord{
\begin{tikzpicture}[baseline=2.5mm]
\draw[-, thick,double,darkblue] (.12,.7) to (.12,.89);
\draw[darkblue,thick,yshift=-5pt,xshift=-5pt] (0,.53) rectangle ++(10pt,10pt);
\node at (0,.53) {$\scriptstyle f$};
\draw[darkblue,thick,yshift=-5pt,xshift=-5pt] (0,0) rectangle ++(10pt,10pt);
\node at (0,0) {$\scriptstyle g$};
\draw[-, thick,double,darkblue] (0,-.18) to (0,-.3);
\draw[-,thick,darkblue] (-.12,.17) to [out=90,in=-90] (-.4,.5);
\draw[-,thick,darkblue] (-.4,.5) to[out=90,in=180] (-.25,.9);
\draw[->,thick,darkblue] (-.25,.9) to[out=0,in=90] (-.12,.7);
\draw[-, thick,double,darkblue] (.12,.17) to (.12,.35);
\node at (0,-.45) {$\c$};
\node at (.12,1.05) {$\a$};
\end{tikzpicture}
}.\label{counit}
\end{align}
Reversing the roles of $\up$ and $\down$ in this argument, we get another canonical
adjunction making
$(F,E)$ is an adjoint pair.
So $E$ and $F$ are biadjoint, hence, they are exact, and
preserve locally finite-dimensional, finitely generated, finitely
cogenerated, projective and injective objects;
cf. \cite[Theorem 2.11]{BD}.

Recalling (\ref{JM}),
let ${_{\up}}X:{_\up} OS \rightarrow {_\up} OS$
be the bimodule endomorphism defined on $1_{\up \b} OS$ 
by
left multiplication by $X(\up \b)$, 
and 
let 
$X_{\down}:OS_\down\rightarrow OS_\down$ be defined on $1_\b OS_{\down \b}$ by right multiplication
by $X(\down \b)$.
These are intertwined by the isomorphism from Lemma~\ref{easy}; this
depends on (\ref{d4}).
Similarly, switching $\up$ with $\down$ everywhere, we get
endomorphisms ${_{\down}}X:{_\down}OS\rightarrow {_\down} OS$ and
$X_{\up}:OS_\up \rightarrow OS_\up$, which are again intertwined by the
isomorphism from Lemma~\ref{easy}.

\begin{lemma}\label{newz}
There are short exact sequence of $(OS^\circ, OS)$-bimodules
\begin{align}
0
&\longrightarrow
OS^\circ_\up \otimes_{OS^\sharp} OS 
\stackrel{\alpha}{\longrightarrow}
OS^\circ \otimes_{OS^\sharp} OS_\up
\stackrel{\beta}{\longrightarrow}
{_\down}OS^\circ\otimes_{OS^\sharp} OS
\longrightarrow 0,\label{ses1}\\
0
&\longrightarrow
OS^\circ_\down \otimes_{OS^\sharp} OS 
\stackrel{\alpha}{\longrightarrow}
OS^\circ \otimes_{OS^\sharp} OS_\down
\stackrel{\beta}{\longrightarrow}
{_\up}OS^\circ\otimes_{OS^\sharp} OS
\longrightarrow 0.\label{ses2}
\end{align}
The maps $\alpha$ and $\beta$ in the first sequence 
satisfy $\alpha \circ (X^\circ_{\up} \otimes \id)= (\id
\otimes X_{\up})\circ \alpha$ and $\beta \circ (\id \otimes X_{\up}) = ({_{\down}}X^\circ \otimes \id) \circ \beta$.
The maps in the second sequence have analogous properties.
\end{lemma}

\begin{proof}
We just go through the details for the first short exact sequence. 
The bimodule homomorphisms $\alpha$ and $\beta$ are defined on pure
tensors as follows:
\begin{align}\label{alpha}
\alpha:
\mathord{
\begin{tikzpicture}[baseline=2.5mm]
\draw[darkblue,thick,yshift=-5pt,xshift=-5pt] (0,.3) rectangle ++(10pt,10pt);
\node at (0,.3) {$\scriptstyle f$};
\draw[->, thick,darkblue] (-.1,-.3) to (-.1,.12);
\draw[-, thick,double,darkblue] (0.1,.12) to (0.1,-.3);
\draw[-, thick,double,darkblue] (0,.47) to (0,.85);
\node at (0.1,-.45) {$\b$};
\node at (0,1.02) {$\a$};
\end{tikzpicture}
}
\otimes
\mathord{
\begin{tikzpicture}[baseline=2.5mm]
\draw[darkblue,thick,yshift=-5pt,xshift=-5pt] (0,.3) rectangle ++(10pt,10pt);
\node at (0,.3) {$\scriptstyle g$};
\draw[-, thick,double,darkblue] (0,.12) to (0,-.3);
\draw[-, thick,double,darkblue] (0,.47) to (0,.85);
\node at (0,1.05) {$\b$};
\node at (0,-.45) {$\c$};
\end{tikzpicture}
}
&\mapsto
\mathord{
\begin{tikzpicture}[baseline=2.5mm]
\draw[darkblue,thick,yshift=-5pt,xshift=-5pt] (0,.3) rectangle ++(10pt,10pt);
\node at (0,.3) {$\scriptstyle f$};
\draw[->, thick,darkblue] (-.1,-.3) to (-.1,.12);
\draw[-, thick,double,darkblue] (0.1,.12) to (0.1,-.3);
\draw[-, thick,double,darkblue] (0,.47) to (0,.85);
\node at (0.1,-.45) {$\b$};
\node at (0,1.02) {$\a$};
\end{tikzpicture}
}
\otimes
\mathord{
\begin{tikzpicture}[baseline=2.5mm]
\draw[darkblue,thick,yshift=-5pt,xshift=-5pt] (0,.3) rectangle ++(10pt,10pt);
\node at (0,.3) {$\scriptstyle g$};
\draw[->, thick,darkblue] (-.35,-.3) to (-.35,.85);
\draw[-, thick,double,darkblue] (0,.12) to (0,-.3);
\draw[-, thick,double,darkblue] (0,.47) to (0,.85);
\node at (0,1.05) {$\b$};
\node at (0,-.45) {$\c$};
\end{tikzpicture}
}\:,&
\beta:
\mathord{
\begin{tikzpicture}[baseline=2.5mm]
\draw[darkblue,thick,yshift=-5pt,xshift=-5pt] (0,.3) rectangle ++(10pt,10pt);
\node at (0,.3) {$\scriptstyle f$};
\draw[-, thick,double,darkblue] (0,.12) to (0,-.3);
\draw[-, thick,double,darkblue] (0,.47) to (0,.85);
\node at (0,-.45) {$\b$};
\node at (0,1.02) {$\a$};
\end{tikzpicture}
}
\otimes
\mathord{
\begin{tikzpicture}[baseline=2.5mm]
\draw[darkblue,thick,yshift=-5pt,xshift=-5pt] (0,.3) rectangle ++(10pt,10pt);
\node at (0,.3) {$\scriptstyle g$};
\draw[<-, thick,darkblue] (-0.1,.12) to (-0.1,-.3);
\draw[-, thick,double,darkblue] (0.1,.12) to (0.1,-.3);
\draw[-, thick,double,darkblue] (0,.47) to (0,.85);
\node at (0,1.05) {$\b$};
\node at (0.1,-.45) {$\c$};
\end{tikzpicture}
}
&\mapsto
\mathord{
\begin{tikzpicture}[baseline=2.5mm]
\draw[darkblue,thick,yshift=-5pt,xshift=-5pt] (0,.3) rectangle ++(10pt,10pt);
\node at (0,.3) {$\scriptstyle f$};
\draw[->, thick,darkblue] (-.35,.85) to (-.35,-.3);
\draw[-, thick,double,darkblue] (0,.12) to (0,-.3);
\draw[-, thick,double,darkblue] (0,.47) to (0,.85);
\node at (0,1.02) {$\a$};
\node at (0,-.45) {$\b$};
\end{tikzpicture}
}
\otimes
\mathord{
\begin{tikzpicture}[baseline=2.5mm]
\draw[darkblue,thick,yshift=-5pt,xshift=-5pt] (0,.3) rectangle ++(10pt,10pt);
\node at (0,.3) {$\scriptstyle g$};
\draw[-, thick,darkblue] (-.4,.85) to (-.4,.1);
\draw[-, thick,darkblue] (-.4,.1) to [out=-90,in=180] (-.25,-.1);
\draw[->, thick,darkblue] (-.25,-.1) to [out=0,in=-90] (-.1,.12);
\draw[-, thick,double,darkblue] (0.1,.12) to (0.1,-.3);
\draw[-, thick,double,darkblue] (0,.47) to (0,.85);
\node at (0,1.05) {$\b$};
\node at (0.1,-.45) {$\c$};
\end{tikzpicture}
}\:.
\end{align}
It is straightforward to see these are well-defined bimodule
homomorphisms.
Also $\beta \circ \alpha = 0$. 
Indeed, if we apply $\beta \circ \alpha$ to a pure tensor as above, we
produce a pure tensor of the form $f \otimes g$ 
such that
the strand of $g$ starting in
the top left corner is a rightward cup. This cup commutes past the tensor to give zero
since we are viewing ${_\down}OS^\circ$ as a right $OS^\sharp$-module by inflation.

To show that the sequence is exact, we pick bases.
Recall that $\words_{r,s}$ denotes words which have exactly $r$ letters $\up$ and $s$
letters $\down$.
For $\a,\b \in \words_{r,s}$, let $A_{\a,\b}$ be a basis for $1_\a OS^\circ 1_\b$ 
consisting of reduced lifts of matchings. 
Similarly, for $\b \in \words_{r,s}$ and $\c \in \words_{r+t,s+t}$ for $t \geq 0$, 
let $B_{\b,\c}$ be a basis for $1_\b OS^- 1_\c$ consisting of reduced
lifts of matchings.
By Lemma~\ref{td}, we see that
\begin{align*}
P &:= \left\{f \otimes g\:\bigg|\:
\begin{array}{l}
f \in A_{\a, \up \b}, g \in B_{\b,\c}\text{ for }
r,s,t\geq 0\text{ and }\\
\a \in \words_{r+1,s}, \b \in \words_{r,s}, \c \in \words_{r+t,s+t},
\end{array}
\right\},\\
Q &:= \left\{f \otimes g\:\bigg|\:
\begin{array}{l}
f \in A_{\a, \b}, g \in B_{\b,\up \c}\text{ for }
r,s,t\geq 0\text{ with }s+t\geq 1\text{ and }\\
\a, \b \in \words_{r,s}, \c \in \words_{r+t,s+t-1},
\end{array}
\right\},\\
R &:= \left\{f \otimes g\:\bigg|\:
\begin{array}{l}
f \in A_{\down \a, \b}, g \in B_{\b,\c}\text{ for }
r,s,t\geq 0\text{ and }\\
\a \in \words_{r,s}, \b \in \words_{r,s+1}, \c \in \words_{r+t,s+t+1},
\end{array}
\right\}
\end{align*}
are bases for $OS^\circ_\up \otimes_{OS^\sharp} OS$,
$OS^\circ \otimes_{OS^\sharp} OS_\up$ and ${_\down}OS^\circ\otimes_{OS^\sharp}
OS$, respectively.
Then we partition the set $Q$ as $Q_1\sqcup Q_2$
so that $Q_1$ consists of all $f \otimes g \in Q$ such that the
reduced lift $g$ has a propagating upward strand on its left edge,
and $Q_2$ consists of all remaining elements of $Q$.
Note for each $f \otimes g \in Q_2$ that 
the strand of $g$ starting in
the bottom left corner is a rightward cap. Then it is clear that the
map $\alpha$ maps $P$ bijectively onto $Q_1$ and $\beta$ maps $Q_2$
bijectively onto $R$. This completes the proof.

Now we check that
$\alpha \circ (X^\circ_{\up}\otimes \id) = (\id \otimes X_{\up}) \circ \alpha$.
Take
$f \otimes g\in
OS^\circ_\up \otimes_{OS^\sharp} OS$.
We must show that
$(f \circ X^\circ(\up \b)) \otimes (\up g) = f \otimes ((\up g) \circ X(\up
  \c))$ for $f \in 1_\a OS^\circ 1_{\up \b}$ and $g \in 1_\b OS 1_{\c}$.
We can move $X^\circ(\up \b)$ over the first tensor product and commute $X(\up \c)$
with $\up g$, to reduce to checking that
$1_{\up \b} \otimes X(\up \b) =1_{\up \b} \otimes X^\circ(\up \b).$
The morphism $X^\circ(\up \b)$ can be transformed into $X(\up \b)$ 
by using the quadratic relation to switch 
some positive crossings to negative crossings. This produces some
error terms which involve caps at the top of the picture, which become
zero when commuted back over the tensor product.
(This argument can be made more formal by using induction on the
length of the word $\b$, using the recursions (\ref{JMa}) and (\ref{chile2}).)

The proof that
$\beta \circ (\id \otimes X_{\up}) =
({_{\down}}X^\circ \otimes \id)\circ \beta$ is similar; one needs to use also
(\ref{d5}) and Lemma~\ref{chile0}.
\end{proof}

Finally, we refine the functors $E$ and $F$.
For $i \in \k^\times$, let $E_i$ be the subfunctor of $E$ that is
defined by tensoring with the bimodule that is the generalized
$i$-eigenspace of ${_{\up}}X:{_\up}OS \rightarrow {_\up} OS$;
equivalently, it may be defined by tensoring
with the generalized $i$-eigenspace of 
$X_{\down}:OS_\down \rightarrow OS_\down$.
Similarly, switching $\up$ and $\down$ everywhere, defines a
subfunctor $F_i$ of $F$.
Let
\begin{equation}\label{I}
I := I_1 \cup I_{t^{-2}} = \left\{q^{2n}, t^{-2} q^{-2n}\:\big|\:n\in \ZZ\right\} \subset \k^\times.
\end{equation}

\begin{lemma}\label{phew}
There are short exact sequences of functors
for each $i \in \k^\times$:
\begin{align}\label{sesI}
0&\longrightarrow
\Delta \circ F_i^\up \longrightarrow F_i \circ \Delta \longrightarrow
\Delta \circ F_i^\down \longrightarrow 0,\\\label{sesII}
0&\longrightarrow
\Delta \circ E_i^\down \longrightarrow E_i \circ \Delta \longrightarrow
\Delta \circ E_i^\up \longrightarrow 0.
\end{align}
Moreover,
the functors $E_i$ and $F_i$ are biadjoint, and we have that
\begin{equation}\label{decomp2}
E = \bigoplus_{i \in I} E_i,
\qquad
F = \bigoplus_{i \in I} F_i.
\end{equation}
\end{lemma}

\begin{proof}
Note to start with that although ${_\down} OS$ is not
finite-dimensional (or even a direct sum of finite-dimensional
bimodules as was the case for ${_\down} OS^\circ$), it is locally
finite-dimensional in the sense that it is the direct sum of
the finite-dimensional vector spaces $1_{\down \a} OS 1_\b$ for $\a,\b \in
\words$. The endomorphism ${_{\down}}X$ leaves each of these
finite-dimensional vector
spaces invariant. This is enough to see that each generalized $i$-eigenspace of
${_{\down}}X$ is a summand of the bimodule ${_\down} OS$.
However, until we have proved (\ref{decomp2}), there may also be summands
arising from generalized
eigenspaces corresponding to eigenvalues of ${_{\down}}X$ not in $I
\subset \k^\times$, and there could
also be summands arising from non-linear irreducible factors of the
characteristic polynomial.
Similar remarks apply to $F_i$.

To define an adjunction making $(E_i,F_i)$ into an adjoint pair,
we project the adunction for $(E,F)$ onto the summands $E_i$ and $F_i$.
To see that this does the job, 
one needs to use the explicit forms for the unit and counit of the adjunction $(E,F)$
given in (\ref{unit})--(\ref{counit}). The key point is that
${_{\up}}X \otimes \operatorname{id} = \operatorname{id} \otimes
{_{\down}}X$ as an endomorphism of ${_\up} OS \otimes_{OS} {_\down} OS$
and ${_{\down}}X \otimes \operatorname{id} = \operatorname{id} \otimes
{_{\up}}X$ as an endomorphism of ${_\down} OS \otimes_{OS} {_\up} OS$.
A similar argument produces an adjunction $(F_i,E_i)$ the other way
around.
It then follows that $E_i$ and $F_i$ are both exact; one can also see
this since they are summands of the exact functors $E$ and$F$.

The short exact sequences from Lemma~\ref{newz} may be viewed
equivalently as short exact sequences of functors
\begin{align*}
0&\longrightarrow
\Delta \circ F^\up \longrightarrow F \circ \Delta \longrightarrow
\Delta \circ F^\down \longrightarrow 0,\\
0&\longrightarrow
\Delta \circ E^\down \longrightarrow E \circ \Delta \longrightarrow
\Delta \circ E^\up \longrightarrow 0.
\end{align*}
Similarly, using the final assertion of the lemma, we get
(\ref{sesI})--(\ref{sesII}) from
Lemma~\ref{newz} on passing to the appropriate generalized
eigenspaces.

Finally, we must establish (\ref{decomp2}).
The short exact sequences of functors obtained in the previous paragraph plus
(\ref{decomp1})
imply that (\ref{decomp2}) holds on any standard module $\Delta(\LA)$.
By exactness and Corollary~\ref{BGG}, it follows that it also holds on
any indecomposable projective module.
Hence, it is true on any module.
\end{proof}

With these branching rules in hand, we can proceed to the definition
of the formal character of a locally finite-dimensional $OS$-module.

First, we must refine the idempotents $1_\a$ for $\a \in \words$.
Let $\Words$ be the set of words in the
letters $\{\up_i, \down_i\:|\:i \in I\}$.
Thus, an element of $\Words$ has the form
$\a_\bi = (\a_n)_{i_n}\cdots (\a_1)_{i_1}$
for words $\a = \a_n \cdots \a_1 \in \words$
and $\bi = i_n \cdots i_1 \in \langle I \rangle$.
Take a word $\a_\bi \in \Words$ of length $n$.
Let $X_i$ be the Jucys-Murphy element in $1_\a OS 1_\a$ that is defined by a dot
on the $i$th strand from the right, so that
 $X_1,\dots,X_n$ generate a commutative subalgebra of the
finite-dimensional algebra $1_\a OS 1_\a$.
It follows that there is an idempotent $1_{\a_\bi} \in 1_\a OS 1_\a$
which projects any $1_\a OS 1_\a$-module onto the simultaneous
generalized eigenspaces of $X_1,\dots,X_n$
corresponding to eigenvalues $i_1,\dots,i_n$, respectively.
For a given $\a$, all but finitely many $1_{\a_\bi}$ are zero.

Now define the {\em formal character} of a locally finite-dimensional
$OS$-module $M$ by
\begin{equation}
\ch M := \sum_{\a_\bi\in\Words} (\dim M 1_{\a_\bi}) \a_\bi,
\end{equation}
which is an element of the ring of (possibly infinite) $\ZZ$-linear combinations of elements of
the monoid $\Words$.
From the proof of the following lemma plus (\ref{decomp2}), one sees
that $1_\a = \sum_{\bi} 1_{\a_\bi}$.
Note also that $\ch$ is additive on short exact sequences.

\begin{lemma}\label{basically}
$\ch M = \sum_{i \in I} \down_i (\ch E_i M) 
+ \sum_{i \in I} \up_i (\ch F_i M).$
\end{lemma}

\begin{proof}
Take $\a_\bi\in \Words$
and suppose that $\a = \up \b, \bi = i \bj$.
We claim that $\dim M 1_{\a_\bi} = \dim (F_i M) 1_{\b_\bj}$.
The lemma follows from this together with
the analogous statement 
argument with $\up$ replaced with $\down$ and $F_i$ replaced with
$E_i$. 
To prove the claim, 
using (\ref{theee}), we have that
\begin{align*}
M 1_{\a}\cong\Hom_{OS}(1_{\a} OS, M)&\cong\Hom_{OS}(E (1_\b OS), M)\\
&\cong \Hom_{OS}(1_\b OS, F M)
\cong (FM) 1_\b.
\end{align*}
Under this isomorphism, the generalized $i$-eigenspace of 
$\mathord{
\begin{tikzpicture}[baseline = -1mm]
      \node at (0,0) {$\color{darkblue}\scriptstyle\bullet$};
	\draw[<-,thick,darkblue] (0,0.25) to (0,-0.25);
\end{tikzpicture}
} \otimes 1_\b$
corresponds to the summand $(F_i M) 1_\b$.
The result follows.
\end{proof}

\begin{lemma}
The characters 
$\left\{\ch \L(\LA)\:\big|\:\LA \in \RPar\right\}$
of the irreducible $OS$-modules
are linearly independent.
\end{lemma}

\begin{proof}
Take $\LA \in \RPar_{r,s}$.
As $\L(\LA)$ is the shortest word module of type $\LA$, its formal
character is a sum $A_\lambda$ of words of the form $\down_{i_{r+s}}
\cdots \down_{i_{r+1}} \up_{i_r} \cdots \up_{i_1}$, 
plus a sum $B_\lambda$ of words that are
obtained from the ones in $A_\lambda$ by properly shuffling the $\down$'s and $\up$'s,
plus a sum $C_\lambda$ of strictly longer words.
By unitriangularity, it suffices to show that the ``leading terms''
$A_\lambda$ are linearly independent for fixed $r,s$ and all $\LA \in
\RPar_{r,s}$.
But $A_\lambda$ is just the product of the formal characters of
$\D_{\la^\down}$ and $\D_{\la^\up}$ in the usual sense of the 
Hecke algebras $H_s$ and $H_r$.
So these words are linearly independent by the well-known linear
independence of irreducible characters for the Hecke
algebra\footnote{This may be proved in
the same way as is explained for the symmetric group in \cite[Lemma 11.2.5]{Kbook}.}.
\end{proof}

Now we define the ($t$-shifted) {\em bipartition graph} to be the $I$-colored directed
graph with vertices $\Par$ and an edge $\LA \stackrel{i}{\rightarrow}
\MU$ if one of the following holds:
\begin{itemize}
\item $\MU$ is obtained from $\LA$ by adding a node of content $i$ to
 $\la^\up$;
\item $\LA$ is obtained from $\MU$ by adding a node of
  content $t^{-2}i^{-1}$
to $\mu^\down$.
\end{itemize}
A small piece of this graph is displayed in Figure 2.
\begin{figure}[hb]
\[
\mathord{
\begin{tikzpicture}[scale=0.7]
	\node (a) at (0, 4) {\small$\NOTHING$};
	\node (b) at (-4, 2) {$\left(\text{\tiny$\yng(1)$} , {\scriptstyle\varnothing}\right)$};
	\node (c) at (4, 2) {$\left({\scriptstyle\varnothing}, \text{\tiny$\yng(1)$}\right)$};
	\node (d) at (-8, 0) {$\left(\text{\tiny$\yng(1,1)$}, {\scriptstyle\varnothing}\right)$};
	\node (e) at (-4, 0) {$\left(\text{\tiny$\yng(2)$}, {\scriptstyle\varnothing}\right)$};
	\node (f) at (0, 0) {$\left({\text{\tiny$\yng(1)$}, \text{\tiny$\yng(1)$}}\right)$};
	\node (g) at (4, 0) {$\left({\scriptstyle\varnothing}, \text{\tiny$\yng(1,1)$}\right)$};
	\node (h) at (8, 0) {$\left({\scriptstyle\varnothing}, \text{\tiny$\yng(2)$}\right)$};
	\node (i) at (-9, -4) {$\left({\text{\tiny$\yng(1,1,1)$}}, {\scriptstyle\varnothing}\right)$};
	\node (j) at (-7, -4) {$\left({\text{\tiny$\yng(2,1)$}}, {\scriptstyle\varnothing}\right)$};
	\node (k) at (-5, -4) {$\left({\text{\tiny$\yng(3)$}}, {\scriptstyle\varnothing}\right)$};
	\node (l) at (-3, -4) {$\left({\text{\tiny$\yng(1,1)$}, \text{\tiny$\yng(1)$}}\right)$};
	\node (m) at (-1, -4) {$\left({\text{\tiny$\yng(2)$}, \text{\tiny$\yng(1)$}}\right)$};
	\node (n) at (1, -4) {$\left({\text{\tiny$\yng(1)$}, \text{\tiny$\yng(1,1)$}}\right)$};
	\node (o) at (3, -4) {$\left({\text{\tiny$\yng(1)$}, \text{\tiny$\yng(2)$}}\right)$};
	\node (p) at (5, -4) {$\left({\scriptstyle\varnothing}, \text{\tiny$\yng(1,1,1)$}\right)$};
	\node (q) at (7, -4) {$\left({\scriptstyle\varnothing}, \text{\tiny$\yng(2,1)$}\right)$};
	\node (r) at (9, -4) {$\left({\scriptstyle\varnothing}, \text{\tiny$\yng(3)$}\right)$};
	\draw[->] (a) to node[above]{$\scriptstyle 1$} (b);
	\draw[<-] (a) to node[above]{$\scriptstyle\:\:\:t^{-2}$} (c);
	
	\draw[->] (b) to node[above]{$\scriptstyle q^{-2}$} (d);
	\draw[->] (b) to node[right]{$\scriptstyle q^2$} (e);
	\draw[<-] (b) to node[above]{$\scriptstyle \:\:\:t^{-2}$} (f);
	
	\draw[->] (c) to node[above]{$\scriptstyle 1$} (f);
	\draw[<-] (c) to node[right]{$\!\scriptstyle t^{-2}q^2$}(g);
	\draw[<-] (c) to node[above]{$\:\:\:\:\:\scriptstyle t^{-2}q^{-2}$} (h);
	
	\draw[->] (d) to node[left]{$\scriptstyle q^{-4}$} (i);
	\draw[->] (d) to node[right]{$\scriptstyle \!q^2$} (j);
	\draw[<-] (d) to node[near start, above]{$\scriptstyle \:\:\:\:t^{-2}$} (l);
	
	\draw[->] (e) to node[near end, below right]{$\scriptstyle \!\!\!\!\!q^{-2}$} (j);
	\draw[->] (e) to node[right]{$\scriptstyle\! q^4$} (k);
	\draw[<-] (e) to node[near start, above right]{$\scriptstyle \!\!\!t^{-2}$} (m);
	
	\draw[->] (f) to node[above left]{$\scriptstyle q^{-2}\!\!\!\!\!\!$} (l);
	\draw[->] (f) to node[below left]{$\scriptstyle q^2$} (m);
	\draw[<-] (f) to node[below left,near end]{$\scriptstyle t^{-2}q^2\!\!\!$} (n);
	\draw[<-] (f) to node[near start, above right]{$\scriptstyle\!\!\!\! t^{-2}q^{-2}$} (o);
	
	\draw[->] (g) to node[above left]{$\scriptstyle 1\!\!\!$} (n);
	\draw[<-] (g) to node[left]{$\scriptstyle t^{-2}q^{4}\!$} (p);
	\draw[<-] (g) to node[near start, above right]{$\scriptstyle\!\!\!\!\! t^{-2}q^{-2}$} (q);
	
	\draw[->] (h) to node[near start, above]{$\scriptstyle 1\!\!\!$} (o);
	\draw[<-] (h) to node[left]{$\scriptstyle t^{-2}q^2\!\!$} (q);
	\draw[<-] (h) to node[right]{$\scriptstyle \!t^{-2}q^{-4}$} (r);
\end{tikzpicture}
}
\]
\caption[Bipartition graph up to word length $3$]{Bipartition graph up to
  word length $3$}
\end{figure}
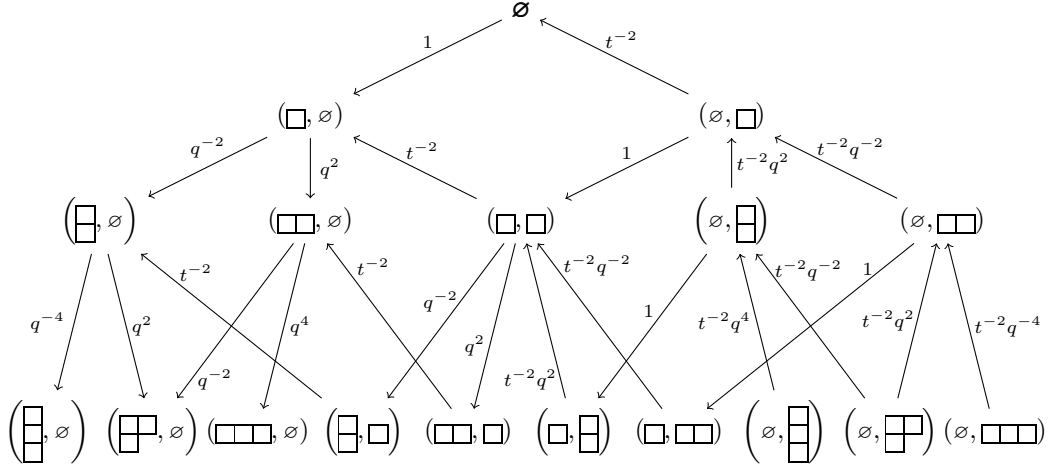

By a {\em path} $\gamma:\LA \rightsquigarrow \MU$ we mean an undirected path in
the bipartition graph
starting at $\LA$ and ending
at $\MU$.
The {\em type} of
such a path $\gamma$ is $\type(\gamma):=\a_\bi \in
\Words$ 
where $\bi = i_n \cdots i_1$ records the colors on the edges of the path
$\LA \stackrel{i_1}{\rule{5mm}{0.5pt}} \cdots \stackrel{i_n}{\rule{5mm}{0.5pt}} \MU$
and $\a = \a_n \cdots \a_1\in\words$ is defined so $\a_m = \up$ or
$\down$ according to whether
the $m$th edge is traversed forwards or backwards according to its direction.
For example, the path
$$
\left(\,\text{\small$\yng(1)$}\,,\,\varnothing\right)
\stackrel{\:t^{-2}\!}{\longleftarrow}  \left(\,\text{\small$\yng(1)$}\,,\,\text{\small$\yng(1)$}\,\right)
\stackrel{1}{\longleftarrow}  \left(\varnothing\,,\,\text{\small$\yng(1)$}\,\right)
\stackrel{t^{-2}}{\longrightarrow} \NOTHING
\stackrel{t^{-2}}{\longleftarrow}
\left(\varnothing\,,\,\text{\small$\yng(1)$}\,\right)$$
is of type $\down_{t^{-2}}\up_{t^{-2}} \down_1 \down_{t^{-2}}$.

\begin{theorem}\label{characters}
For $\LA \in \Par$, we have that
$\displaystyle\ch \tilde \Delta(\LA) =
\sum_{\gamma:\NOTHING\rightsquigarrow \LA} \type(\gamma).$
\end{theorem}

\begin{proof}
Note the infinite
sum in the theorem makes sense since there are only finitely
many paths of any given length.
From (\ref{sesI})--(\ref{sesII}) and Lemma~\ref{newfirst},
we get some explicit $\tilde\Delta$-filtrations of
 $E_i \tilde\Delta(\LA)$
and $F_i \tilde\Delta(\LA)$ with sections $\tilde\Delta(\MU)$ 
for each $\MU \stackrel{i}{\leftarrow}
\LA$ or
$\MU \stackrel{i}{\rightarrow} \LA$, respectively.
Applying Lemma~\ref{basically}, we deduce that
$$
\ch \tilde\Delta(\LA) = \sum_{i\in I}\Big(\sum_{\MU
  \stackrel{i}{\leftarrow} \LA} \down_i \ch \tilde\Delta(\MU)+
 \sum_{\MU
  \stackrel{i}{\rightarrow} \LA} \up_i \ch \tilde\Delta(\MU)\Big).
$$
Now use induction on path length.
\end{proof}

\begin{corollary}\label{Canpath}
Take $\LA \in \Par$ and $\MU \in \RPar$.
If $\L(\MU)$ is a composition factor of $\tilde\Delta(\LA)$
then there is a path $\gamma:\NOTHING\rightsquigarrow\LA$ 
and a minimal length path $\delta:\NOTHING\rightsquigarrow \MU$
such that $\type(\gamma)=\type(\delta)$.
\end{corollary}

\begin{proof}
Pick any word $\a_\bi \in \words_I$ that appears with non-zero
coefficient in the formal character of the shortest word space of
$\L(\MU)$.
Since 
$[\widetilde\Delta(\LA):\L(\MU)]$ and
$[\widetilde\Delta(\MU):\L(\MU)]$ are both non-zero, 
$\a_\bi$ also has non-zero coefficients in $\ch \widetilde\Delta(\LA)$
and $\ch \widetilde\Delta(\MU)$.
So
Theorem~\ref{characters} implies
that there are paths $\gamma:\NOTHING \rightsquigarrow \LA$ and 
$\delta:\NOTHING \rightsquigarrow \MU$ of the same type
$\a_\bi$. Moreover, $\delta$ is of minimal length amongst all paths
$\NOTHING \rightsquigarrow \MU$.
\end{proof}

\begin{corollary}\label{morocco}
Suppose that $\MU \in \RPar_{r,s}$. 
If {\em either} $t \notin \{\pm q^n\:|\:n \in \ZZ\}$, {\em or}
$e=0$, $t = q^{n}$ for $n \in \NN$ and $h(\MU) \leq n$,
then we have that $\P(\MU) = \Delta(\MU)$.
\end{corollary}

\begin{proof}
In view of Corollary~\ref{BGG}, it suffices to show that
$[\bar\Delta(\LA):\L(\MU)] = \delta_{\LA,\MU}$ for all $ \LA \in \RPar$.
Since $\bar\Delta(\LA)$ and $\L(\MU)$ have the same shortest word
spaces, this follows if we can show for $\LA \in \RPar$ that
$[\bar\Delta(\LA):\L(\MU)] \neq 0
\Rightarrow \LA \in \RPar_{r,s}$.
So suppose that $[\bar\Delta(\LA):\L(\MU)]\neq 0$.
Corollary~\ref{Canpath} implies that 
there is a path $\gamma:\NOTHING\rightsquigarrow\LA$
of the same type as a minimal length path $\delta:\NOTHING
\rightsquigarrow \MU$.
Being of minimal length means that $\delta$ is some permutation of the word
$\up_{i_{r+s}}\cdots\up_{i_{r+1}}\down_{i_{r}}\cdots\down_{i_{1}}$,
where $i_1,\dots,i_r$ are the $\up$-contents of the nodes of $\mu^\up$
and $i_{r+1},\dots,i_{r+s}$ are the $\down$-contents of the nodes of
$\mu^\down$.

If $t \notin \{\pm q^n\:|\:n \in \ZZ\}$, then the set of possible
$\up$-contents of nodes of partitions is disjoint from the set of
possible $\down$-contents.
So any path of type $\delta$ starting at $\NOTHING$ necessarily 
ends at an element of $\Par_{r,s}$. We deduce that $\LA
\in \RPar_{r,s}$ as required.
Instead, suppose that $e=0, t = q^n$ for $n \in \NN$, and $h(\MU)
\leq |n|$. 
Recalling that $h(\MU)$ is the total number of non-zero
parts in both $\mu^\up$ and $\mu^\down$,
these assumptions imply that 
$i_1,\dots,i_{r+s}$ are all distinct, and there is a unique path of type $\delta$
starting at $\NOTHING$. This shows that $\gamma=\delta$, hence, $\LA = \MU$.
\end{proof}

\begin{proof}[Proof of Theorem~\ref{selection}]
Corollary~\ref{morocco} shows that $\P(\LA) = \Delta(\LA)$ for all $\LA
\in \RPar$. 
So the standardization functor $\Delta$ 
sends the indecomposable projectives $\{\Y(\LA)\:|\:\LA\in\RPar\}$ in $\Mod OS^\circ$
to the indecomposable projectives $\{\P(\LA)\:|\:\LA \in \RPar\}$ in $\Mod OS$.
Since this functor is also exact, it follows that it is an equivalence
of categories.
\end{proof}

\begin{proof}[Proof of Theorem~\ref{sscase}]
Suppose that $q$ is not a root of unity and $t \notin \{\pm
q^n\:|\:n \in \ZZ\}$.
The first assumption means that $e=0$, so $OS^\circ$ is semisimple.
Hence, $OS$, or equivalently $\dot\OS(z,t)$, is semisimple thanks to Theorem~\ref{selection}.
The parametrization of indecomposable objects in $\dot\OS(z,t)$
follows from Theorem~\ref{class}: up to isomorphism they correspond to the irreducible
projective modules
$\{\Delta(\LA)\:|\:\LA \in \Par\}$.
Moreover, Theorem~\ref{thegg} shows in this case that
$K_0(\dot\OS(z,t))$ may be identified with $\SYM\otimes_\ZZ\SYM$ so that
$[\Delta(\LA)] \leftrightarrow \chi_\LA$.

It remains to show that $OS$ is not semisimple 
for all other parameter choices.
If $q$ is a root of unity and $t \notin \{\pm q^n\:|\:n \in \ZZ\}$,
this follows from Theorem~\ref{selection}, since Hecke algebras are not
semisimple at roots of unity.
Finally, suppose that $q$ is arbitrary but $t = \pm q^n$ for some $n \in \ZZ$.
Using the isomorphisms (\ref{pi})--(\ref{signaut}), we may as well assume that
$t = q^n$ for $n \in \NN$.
Example~\ref{swex} shows that $\bar\Delta(((n),\varnothing))$ is
reducible, since it has a composition factor isomorphic to
$\L(((n+1),(1)))$.
Since $\bar\Delta(((n),\varnothing))$ is a finitely generated module with
irreducible head, it is therefore not completely reducible,
and $OS$ is not semisimple.
\end{proof}

\section{Categorical action}\label{tpc}

Recall 
that $q \in \k^\times$ is either not a root of unity (in which case $e=0$), or that
$q^2$ is a primitive $e$th root of unity for some $e > 1$.
We are going to show that $\Mod  OS$ has the structure of a tensor
product categorification in the general sense of Losev and Webster
\cite{LW}.
This is most interesting when $t \in \{\pm q^n\:|\:n \in \ZZ\}$ (so
that the $\g$-module $V(-\Lambda_0|\Lambda_{t^{-2}})$ is reducible), but there is no need to impose this assumption.

The set $I$ from (\ref{I})
will now be used to index the simple roots of a symmetric
Kac-Moody algebra $\g$ (over ground field $\CC$),
namely, the Kac-Moody algebra with
Cartan matrix $(c_{i,j})_{i,j \in I}$ defined by (\ref{cartan}).
The Lie algebra $\g$ is generated by its Cartan subalgebra $\h$
and Chevalley generators $\{e_i, f_i\:|\:i \in I\}$ subject to the
Serre relations.
Let 
$$
P := \left\{\La\in\h^*\:\big|\:\langle h_i, \La\rangle \in \ZZ\text{ for
  all }i \in I\right\}.
$$
The simple roots are 
$\{\alpha_i\:|\:i \in I\}$, and
we have that $\langle h_i, \alpha_j \rangle = c_{i,j}$
where $h_i := [e_i,f_i]$.
The fundamental dominant weights are $\{\Lambda_i\:|\:i \in I\}$.
For $i \in I$, let $V(\La_i)$ (resp. $V(-\La_i)$) denote the
integrable highest (resp. lowest) weight module
of highest weight $\La_i$ (resp. lowest weight $-\La_i$).

Let $\g^\up  = \{x^\up\:|\:x \in \g\}$ and $\g^\down = \{x^\down\:|\:x
\in \g\}$ be two copies of $\g$
with Cartan subalgebras $\h^\up$ and $\h^\down$, respectively.
There is a Lie algebra homomorphism
\begin{equation}\label{De}
\Delta:\g \rightarrow \g^\up \oplus \g^\down,
\qquad x \mapsto x^\up + x^\down.
\end{equation}
Identifying $U(\g^\up\oplus\g^\down)$ with $U(\g)\otimes U(\g)$, this
homomorphism amounts to the usual comultiplication on $U(\g)$.
Let $\Fock$ be the $\CC$-vector space with
basis $\{v_\LA\:|\:\LA \in \Par\}$.
The following makes $\Fock$ into a $\g^\up \oplus \g^\down$-module:
\begin{itemize}
\item 
For $i \in I^\up$ we let $e_i^\up v_\LA$ (resp. $f_i^\up v_\LA$)
be the vector $\sum_\MU v_\MU$
summing over all bipartitions $\MU$ obtained from $\LA$ by
adding (resp. removing) a node of $\up$-content $i$ to (resp. from)
$\la^\up$.
\item For $i \in I^\down$ we let $e_i^\down v_\LA$ (resp. $f_i^\down v_\LA$)
be the vector $\sum_\MU v_\MU$
summing over all bipartitions $\MU$ obtained from $\LA$ by
removing (resp. adding) a node of $\down$-content $i$ from (resp. to) $\la^\down$.
\item
The actions of the Cartan subalgebras $\h^\up$ and $\h^\down$ are defined so that $v_\LA$ is 
a weight vector of the following weights for $\h^\up$ or $\h^\down$, respectively:
\begin{align}\label{css1}
\wt^\up(\LA)
&:=-\La_1 + \sum_{A \in \la^\up}
\alpha_{\cont(A)},\\\label{css2}
\wt^\down(\LA)
&:=\La_{t^{-2}} - \sum_{A \in \la^\down}
\alpha_{t^{-2}\cont(A)^{-1}}.
\end{align}
\end{itemize}

Let $V(-\Lambda_1|\Lambda_{t^{-2}})$ be the
$\g^\up\oplus\g^\down$-submodule
of $\Fock$ generated by $v_{\NOTHING}$.
When $e = 0$, we have that $V(-\Lambda_1|\Lambda_{t^{-2}}) = \Fock$,
but it is a proper submodule otherwise. Since $v_{\NOTHING}$ is a lowest weight vector for
$\g^\up$ of weight $-\La_1$ and a highest weight vector for
$\g^\down$ of weight $\La_{t^{-2}}$,
$V(-\Lambda_1|\Lambda_{t^{-2}})$ is isomorphic to the
irreducible $\g^\up\oplus\g^\down$-module $V(-\Lambda_1) \boxtimes 
V(\Lambda_{t^{-2}})$ (with $\g^\up$ acting on the first tensor factor
and $\g^\down$ acting on the second).

For $\LA \in \RPar_{r,s}$ and $r,s \geq 0$, we let
\begin{equation}\label{bla}
b_\LA(e,p) := \sum_{\MU \in \Par_{r,s}} [\SS(\MU):\D(\LA)]
v_\MU \in \Fock,
\end{equation}
so called because it depends on both $e$ and $p$.
The following lemma is a reinterpretation of a well-known result about the
representation theory of Hecke algebras.
It shows in particular that the vectors
$\{b_\LA(e,p)\:|\:\LA\in\RPar\}$ give a
basis for $V(-\Lambda_1|\Lambda_{t^{-2}})$.

\begin{lemma}\label{bunnarong}
There is a vector space isomorphism
\begin{align*}
\CC\otimes_\ZZ K_0(\pMod OS^\circ) &\stackrel{\sim}{\rightarrow}
V(-\Lambda_1|\Lambda_{t^{-2}}),
\qquad [\Y(\LA)] \mapsto b_\LA(e,p).
\end{align*}
This map intertwines the endomorphisms 
of $\CC\otimes_{\ZZ} K_0(\pMod OS^\circ)$
induced by the endofunctors
$E_i^\up, E_i^\down, F_i^\up, F_i^\down$
from (\ref{fan1})--(\ref{fan4})
with the actions of the Chevalley generators $e_i^\up, e_i^\down,
f_i^\up, f_i^\down$ of $\g^\up\oplus\g^\down$ on 
$V(-\Lambda_1|\lambda_{t^{-2}})$.
\end{lemma}

\begin{proof}
Since the rectangular matrix $([\SS(\MU):\D(\LA)])$
is unitriangular, the elements $b_\LA(e,p)$ for $\LA \in \RPar$ are
linearly independent. So
the linear map $$
f:\CC\otimes_{\ZZ} K_0(\pMod OS^\circ)\rightarrow \Fock, \qquad [\Y(\LA)] \mapsto
b_\LA(e,p)
$$ 
is injective.
In the next paragraph, we show that $f$ intertwines 
$[E_i^\up], [E_i^\down], [F_i^\up], [F_i^\down]$ with $e_i^\up,
e_i^\down, f_i^\up, f_i^\down$, respectively.
Actually, we prove an equivalent dual statement.

Let $K_0(\fdMod OS^\circ)$ be the Grothendieck group of the Abelian
category $\fdMod OS^\circ$, which has basis given by the classes
$\{[\D(\LA)]\:|\:\LA \in \RPar\}$.
We have the non-degenerate Cartan pairing
$$
\langle\cdot,\cdot\rangle:K_0(\pMod OS^\circ) \times K_0(\fdMod OS^\circ) \rightarrow \ZZ
$$
such that $\langle [\Y(\LA)], [\D(\MU)]\rangle = \delta_{\LA,\MU}$ for
$\LA,\MU \in \RPar$.
Lemma~\ref{fur} implies that the linear maps $[E_i^\up]$ and $[E_i^\down]$ are biadjoint to
$[F_i^\up]$ and 
$[F_i^\down]$, respectively.
There is also a non-degenerate symmetric bilinear form $\langle\cdot,\cdot\rangle$ on
$\Fock$ defined so that $\{v_\LA\:|\:\LA \in \Par\}$ is an
orthonormal basis. Again, $e_i^\up$ and $e_i^\down$ are biadjoint to
$f_i^\up$ and $f_i^\down$, respectively, as is clear from the
explicit 
definition of their actions on the basis.
Let $$
f^*:\Fock \rightarrow \CC\otimes_\ZZ K_0(\fdMod OS^\circ)
$$
be the dual map to $f$. It sends $v_\LA \mapsto \sum_{\MU \in
  \RPar} [\SS(\LA):\D(\MU)] [\D(\MU)]$, i.e., to the isomorphism
class $[\SS(\LA)]$ of the Specht module.
Now it is clear from Lemma~\ref{newfirst} that $f^*$ intertwines $f_i^\up, f_i^\down, e_i^\up,
e_i^\down$ with $[F_i^\up], [F_i^\down], [E_i^\up], [E_i^\down]$,
respectively.

The proof so far shows that $\CC\otimes_\ZZ K_0(\pMod OS^\circ)$ has the
structure of an integrable $\g^\up\oplus\g^\down$-module.
It remains to show that the image of $f$ is the submodule
$V(-\Lambda_1|\Lambda_{t^{-2}})$.
This follows because $f$ sends
$[\Y(\NOTHING)]$ to the generator $v_\NOTHING$
of $V(-\Lambda_1|\Lambda_{t^{-2}})$,
and $\CC \otimes_\ZZ K_0(\pMod OS^\circ)$ is actually generated 
as a $\g^\up\oplus\g^\down$-module by this vector.
The latter assertion is a consequence of the analogous statement for the Hecke algebra, 
which is well known; e.g., see \cite[Corollary 4.34]{BD}.
 \end{proof}

\begin{remark}\label{cbrem}
When $p=0$, the basis $\{b_{\LA}(e,p)\:|\:\La \in \RPar\}$
is the monomial basis consisting of pure tensors in Lusztig's
canonical bases for $V(-\Lambda_1)$ and $V(\Lambda_{t^{-2}})$.
This follows from \cite{Ariki}.
When $p > 0$, the decomposition numbers
$[\SS(\LA):\D(\MU)]$ are not known, so it is hard to compute this
basis explicitly.
\end{remark}

Using the homomorphism $\Delta$ from (\ref{De}), we can instead
view $\Fock$ also as a $\g$-module. For this action, $v_\LA$ is of
weight 
\begin{equation}\label{wtdef}
\wt(\LA) := 
\wt^\up(\LA) + \wt^{\down}(\LA)
\end{equation}
with respect to the Cartan subalgebra $\h$.
Also set
\begin{equation}\label{WTdef}
\WT(\LA) := (\wt^\up(\LA), \wt^\down(\LA)) \in P \times P.
\end{equation}
The cyclic $\g^\up\oplus\g^\down$-submodule 
$V(-\Lambda_1|\Lambda_{t^{-2}})$ of $\Fock$
becomes a $\g$-submodule isomorphic to the tensor product $V(-\Lambda_1)\otimes
V(\Lambda_{t^{-2}})$.
A simple induction on weights shows that the vector $v_\NOTHING$ also
generates this module over $\g$. However, 
it is not an 
 irreducible $\g$-module when $t \in \{\pm q^n\:|\:n \in \ZZ\}$.
For the statement of the next lemma, it may be helpful to recall that $K_0(\pMod OS)$ is identified with 
$K_0(\deltaMod OS)$ by Corollary~\ref{exactsubcat}.

\begin{lemma}\label{int}
The functors $E_i$ and $F_i$ send modules with $\Delta$-flags to
modules with $\Delta$-flags, hence, they induce endomorphisms of
$K_0(\deltaMod OS)$. Moreover, 
there is a vector space isomorphism
$$
\CC\otimes_\ZZ K_0(\deltaMod OS)\stackrel{\sim}{\rightarrow}
V(-\Lambda_1|\Lambda_{t^{-2}}),\quad
[\Delta(\LA)] \mapsto b_\LA(e,p)
$$
which intertwines these endomorphisms with the actions of the
Chevalley generators $e_i,f_i$.
\end{lemma}

\begin{proof}
The given linear isomorphism fits into a commutative diagram
$$
\begin{CD}
\CC\otimes_\ZZ K_0(\pMod OS^\circ)&@>\sim>>&V(-\Lambda_1|\Lambda_{t^{-2}})\\
@V [\Delta] VV&&@|\\
\CC \otimes_\ZZ K_0(\deltaMod OS)&@>\sim>>&V(-\Lambda_1|\Lambda_{t^{-2}})
\end{CD}
$$
where the top map is the isomorphism from Lemma~\ref{bunnarong}.
Lemma~\ref{phew} implies that $E_i$ and $F_i$ preserve $\Delta$-flags.
Moreover, it shows that
$[E_i] \circ [\Delta] = [\Delta] \circ [E_i^\up] + \Delta \circ
[E_i^\down]$. Since the top map intertwines
$[E_i^\up],[E_i^\down]$ with $e_i^\up,e_i^\down$, we deduce from
(\ref{De})
that the bottom map intertwines $[E_i]$ with $e_i$, and similarly for
$[F_i]$ and $f_i$.
\end{proof}

Let $\leq$ be the usual dominance order on $P$:
$\rho \leq \sigma$ if $\sigma - \rho$ is a sum of simple roots.
Then, 
we introduce the {\em inverse dominance order} on $P \times P$
by declaring that
\begin{equation*}
(\rho,\sigma) \leq (\rho',\sigma')
\Leftrightarrow \rho+\sigma = \rho'+\sigma'
\text{ and }\rho \geq \rho'
\Leftrightarrow \rho+\sigma = \rho'+\sigma'
\text{ and }\sigma \leq \sigma'.
\end{equation*}
Recalling (\ref{wtdef})--(\ref{WTdef}),
the next result is the {\em linkage principle}.

\begin{theorem}\label{linkage1}
For $\LA \in \Par$ and $\MU \in \RPar$, we have that
$$
[\tilde\Delta(\LA):\L(\MU)] \neq 0\Rightarrow \WT(\MU) \leq \WT(\LA).
$$
\end{theorem}

\begin{proof}
Suppose that $\MU \in \RPar_{r,s}$
is chosen so that $[\tilde\Delta(\LA):\L(\MU)] \neq 0$.
By Corollary~\ref{Canpath}, 
there is a path $\gamma:\NOTHING\rightsquigarrow \LA$ and a minimal
length path
$\delta:\NOTHING\rightsquigarrow\MU$
with $\type(\gamma) = \type(\delta)$.
We show that the existence of such a pair of paths implies that
$\WT(\MU)\leq
\WT(\LA)$ by induction on $r+s$.
The base case $r+s=0$ is trivial as then $\LA = \MU = \NOTHING$.
For the induction step, remove the last edge from each of the paths
$\gamma$ and $\delta$, to obtain shorter paths
$\gamma':\NOTHING\rightsquigarrow \LA'$
and $\delta':\NOTHING\rightsquigarrow\MU'$.
We assume that this last edge is directed in the forward direction,
i.e., 
$\LA' \stackrel{i}{\rightarrow} \LA$
and $\MU' \stackrel{i}{\rightarrow} \MU$;
the argument is entirely similar if it goes backwards.
By induction $\WT(\MU') \leq\WT(\LA')$, i.e., $\wt(\MU') = \wt(\LA')$
and $\wt^\down(\MU') \leq \wt^\down(\LA')$.
The assumption on the last edge means that
$\MU$ is obtained from $\MU'$ by adding a node of
$\up$-content $i$ to $(\mu')^\up$,
and similarly for $\LA$.
We deduce that $\wt(\MU) = \wt(\MU') + \alpha_i = \wt(\LA') +\alpha_i
= \wt(\LA)$ and $\wt^\down(\MU) = \wt^\down(\MU') \leq \wt^\down(\LA')
= \wt^\down(\LA)$.
Hence, $\WT(\MU) \leq \WT(\LA)$.
\end{proof}

\begin{corollary}\label{linkage3}
For $\LA, \MU \in \RPar$ with $\LA \neq \MU$, we have that
$$
[\bar\Delta(\LA):\L(\MU)] \neq 0 \Rightarrow \WT(\MU) < \WT(\LA).
$$
\end{corollary}

\begin{proof}
By shortest word theory, if $\LA \in \RPar_{r,s}$  we must have
that $\MU \in \RPar_{r+d,s+d}$ for some $d > 0$.
Hence, $\WT(\MU) \neq \WT(\LA)$.
Now we are done since $\WT(\MU) \leq \WT(\LA)$ by Theorem~\ref{linkage1}.
\end{proof}

\begin{corollary}\label{linkage2}
Suppose that $\L(\LA)$ and $\L(\MU)$ belong to the same block of $\Mod
OS$
for some
$\LA,\MU \in \RPar$.
Then we have that $\wt(\LA) = \wt(\MU)$.
\end{corollary}

\begin{proof}
It suffices to show that $
\Hom_{OS}(\P(\LA), \P(\MU)) \neq {\bf 0}\Rightarrow
\wt(\LA)\neq\wt(\MU)$.
To see this, we apply Corollary~\ref{BGGplus}
to see if 
$[\P(\MU):\L(\LA)] \neq 0$ that there exists $\NU \in \Par$ such that
$[\tilde\Delta(\NU):\L(\LA)] [\tilde\Delta(\NU):\L(\MU)] \neq 0$.
By Theorem~\ref{linkage1}, this implies that
$\WT(\LA) \leq \WT(\NU) \geq \WT(\MU)$.
Hence, $\wt(\LA) = \wt(\NU) = \wt(\MU)$.
\end{proof}

The two functions $\wt:\Par \rightarrow P$ and $\WT:\Par \rightarrow P
\times P$ give partitions
\begin{equation}
\Par = \coprod_{\omega \in P} \Par_\omega
=\coprod_{(\rho,\sigma)\in P \times P} \Par_{\rho,\sigma}
\end{equation}
where $\Par_\omega := \wt^{-1}(\omega)$ and $\Par_{\rho,\sigma} :=
\WT^{-1}((\rho,\sigma))$.
Define $\RPar_\omega$ and $\RPar_{\rho,\sigma}$ similarly.
There are corresponding block decompositions
\begin{align}\label{blocka}
\Mod OS &= \:\:\:\:\:\prod_{\omega \in P}\:\:\: \Mod OS_\omega,\\
\Mod OS^\circ &= \!\prod_{(\rho,\sigma) \in P \times P} \!\!\!\Mod
OS^\circ_{\rho,\sigma}\label{blockb}
\end{align}
defined by letting $\Mod OS_\omega$ be the Serre subcategory of $\Mod
OS$ consisting of all modules $M$ such that 
$\Hom_{OS}(\P(\LA),M) \neq 0 \Rightarrow \wt(\LA) = \omega$,
and $\Mod OS^\circ_{\rho,\sigma}$ be the Serre subcategory of $\Mod OS^\circ$
consisting of $M$ 
such that 
$\Hom_{OS^\circ}(\Y(\LA),M) \neq 0 \Rightarrow \WT(\LA) = (\rho,\sigma)$.
For $\Mod OS$, the existence of this decomposition 
depends on Corollary~\ref{linkage2} and the general theory of blocks in
locally Schurian categories discussed in \cite[(L9)--(L10)]{BD}.
For $\Mod OS^\circ$, this decomposition refines the one arising
from the algebra decomposition $OS^\circ = \bigoplus_{r,s \geq 0}
OS^\circ_{r,s}$. In view of Lemma~\ref{cartanlem}, it is a
reformulation of the usual block 
decomposition of the Hecke algebras \cite[Theorem 4.13]{DJ2}.

As well as these block decompositions, we can use the inverse
dominance ordering on $P \times P$
to introduce a {\em stratification} on $\Mod OS$ in the sense of \cite[$\S$2]{LW}.
This is defined by letting 
$\Mod OS_{\leq(\rho,\sigma)}$ be the Serre
subcategory of $\Mod OS$ consisting of all $M$ such that
$\Hom_{OS}(\P(\LA), M) \neq 0 \Rightarrow \WT(\LA) \leq (\rho,\sigma)$.
Define $\Mod OS_{< (\rho,\sigma)}$ similarly.
It is important to note that the set $\coprod_{(\rho',\sigma')\geq
  (\rho,\sigma)} \Par_{\rho',\sigma'}$ is {\em finite}.
Indeed, if $\rho$ is obtained from $-\La_1$ by adding $r$ simple roots
and $\sigma$ is obtained from $\La_{t^{-2}}$ by subtracting $s$ simple
roots,
then it is clear from (\ref{css1})--(\ref{css2}) that $\Par_{\rho,\sigma}
\subseteq \Par_{r,s}$; hence,
$\coprod_{(\rho',\sigma')\geq
  (\rho,\sigma)} \Par_{\rho',\sigma'} \subseteq
\coprod_{d=0}^{\min(r,s)} \Par_{r-d,s-d}$ which is finite.
We say that the stratification is {\em upper-finite} because of this property.

For $(\rho,\sigma) \in P\times P$, let
\begin{equation}\label{pirhosigma}
\pi_{\rho,\sigma}:\Mod OS_{\leq(\rho,\sigma)} \rightarrow \Mod
OS^\circ_{\rho,\sigma}
\end{equation}
be the exact functor defined first by restriction to $OS^\circ$ then
projection onto the block parametrized by $(\rho,\sigma)$.
Composing the inclusion of this block into $\Mod OS^\circ$ with either
$\Delta$ or $\nabla$ defines exact functors
\begin{align}\Delta_{\rho,\sigma}&:\Mod OS^\circ_{\rho,\sigma} \rightarrow \Mod
OS_{\leq(\rho,\sigma)},\label{deltaa}\\
\nabla_{\rho,\sigma}&:\Mod OS^\circ_{\rho,\sigma} \rightarrow \Mod
OS_{\leq(\rho,\sigma)}\label{deltab}.
\end{align}
These are left and right adjoint to $\pi_{\rho,\sigma}$,
respectively.

\begin{lemma}\label{Assgr}
For $(\rho,\sigma) \in P \times P$, the functor $\pi_{\rho,\sigma}$
annihilates all irreducible modules $\L(\LA)$ with $\WT(\LA) < (\rho,\sigma)$.
Hence, it induces an exact functor
$$
\bar\pi_{\rho,\sigma} :\Mod OS_{\leq(\rho,\sigma)} / \Mod OS_{<
  (\rho,\sigma)}
\rightarrow \Mod OS^\circ_{\rho,\sigma}.
$$
In fact, this induced functor is an equivalence of categories.
\end{lemma}

\begin{proof}
If $\Par_{\rho,\sigma} \subseteq \Par_{r,s}$ then
$\Par_{<(\rho,\sigma)} \subseteq \coprod_{d > 0} \Par_{r+d,s+d}$.
Hence, for $\LA \in \RPar_{\rho,\sigma}$, 
the restriction of $\L(\LA)$ to $OS^\circ$ belongs to $\prod_{d >0}
\Mod OS^\circ_{r+d,s+d}$ and its projection to $\Mod OS^\circ_{\rho,\sigma}
\subseteq \Mod OS^\circ_{r,s}$ is certainly zero.
Since $\pi_{\rho,\sigma}$ is also exact, we get the induced 
functor $\bar\pi_{\rho,\sigma}$ by the universal property of Serre
quotients.

The irreducible objects in the Serre quotient category 
$\Mod OS_{\leq(\rho,\sigma)} / \Mod OS_{<
  (\rho,\sigma)}$ are represented by
$\{\L(\LA)\:|\:\LA \in \RPar_{\rho,\sigma}\}$.
For $\LA \in \RPar_{\rho,\sigma}$,
the projective cover of $\L(\LA)$ in $\Mod OS_{\leq(\rho,\sigma)}$ is
the largest quotient of $\P(\LA)$ which belongs to this subcategory.
In view of Lemma~\ref{BGG} and Corollary~\ref{linkage3}, this largest quotient is $\Delta(\LA)$.
We deduce that
the objects $\{\Delta(\LA)\:|\:\LA \in \RPar_{\rho,\sigma}\}$
give a complete set of pairwise inequivalent indecomposable projective
objects in
$\Mod OS_{\leq(\rho,\sigma)} / \Mod OS_{<
  (\rho,\sigma)}$.

By shortest word theory and considerations like in the first
paragraph of the proof, 
the exact functor $\bar\pi_{\rho,\sigma}$ sends $\Delta(\LA)$ to 
$\Y(\LA)$. So it induces a bijection between isomorphism classes of
indecomposable
projective objects in its source and target categories.
It follows that it is an equivalence.
\end{proof}

All of this puts us in the setup of \cite[Definition 2.1]{LW}, except that our algebra
$OS$ is locally finite-dimensional rather than finite-dimensional, and our
ordering is upper-finite rather than finite. The formal definition of standardly
stratified category from {\em loc. cit.} is 
generalized to include this slightly more general situation in
\cite[$\S$6.2.1]{EL}.

\begin{theorem}\label{sstrat}
The category $\Mod OS$ with its irreducible objects $\{\L(\LA)\:|\:\LA
\in \RPar\}$
and the stratification 
defined by the function $\WT:\RPar \rightarrow P \times P$
and the inverse dominance ordering $\leq$
is an upper-finite standardly stratified category with
associated graded category $\Mod OS^\circ$.
In case $e = 0$, it is an
upper-finite highest weight category.
\end{theorem}

\begin{proof}
We have already discussed the stratification and shown that it is
upper-finite. Lemma~\ref{Assgr} identifies the
associated graded category with $\Mod OS^\circ$.
Also we know already that
the standardization functor $\Delta_{\rho,\sigma}$ is exact.
It just remains to show that $\P(\LA)$ has a finite filtration with
$\Delta(\LA)$ at the top and other sections of the form $\Delta(\MU)$
for $\MU$ with $\WT(\MU) > \WT(\LA)$. This follows from
Lemma~\ref{BGG}
and Corollary~\ref{linkage3}.
It is highest weight rather than standardly stratified in
case $e=0$ since then each non-zero stratum $\Mod OS^\circ_{\rho,\sigma}$ is semisimple
with just one irreducible object (up to isomorphism).
\end{proof}

We refer the reader to \cite{BD} for the necessary background on
2-representations of 2-Kac-Moody categories used freely in the
proofs of the next two theorems.
Although these notions are essentially due to Rouquier \cite{Rou}, we
are applying them in a locally Schurian
setting not originally considered there.
In particular, the proof of the following theorem depends crucially on 
the (very slight) extension of Rouquier's ``control by $K_0$''  developed in \cite[Theorem 4.27]{BD}.

\begin{proof}[Proof of Theorem~\ref{tpctheorem}]
See \cite[Remark 3.6]{LW} for the notion of tensor product
categorification. In the context of Theorem~\ref{tpctheorem} it means 
the
following data:
\begin{itemize}
\item[(1)] A locally Schurian category $\mathcal C$ with isomorphism
  classes of irreducible objects labelled by $\RPar$, i.e., the
  indexing set for the 
basis of 
$V(-\Lambda_1|\Lambda_{t^{-2}})$ from Remark~\ref{cbrem}.
\item[(2)] A nilpotent categorical action 
making $\mathcal C$ into a 2-representation of the associated
Kac-Moody 2-category $\mathfrak{U}(\g)$.
\end{itemize}
Then we need to verify the following axioms:
\begin{itemize}
\item[(3)] The category $\mathcal C$ is standardly stratified with respect to the
function $\WT:\RPar \rightarrow P \times P$ and the inverse dominance ordering $\leq$ on $P \times P$.
\item[(4)]
For $(\rho,\sigma)\in P \times P$, 
the Serre quotient $\mathcal C_{(\rho,\sigma)} := \mathcal C_{\leq (\rho,\sigma)} / \mathcal C_{< (\rho,\sigma)}$ 
is equivalent
to the category of modules over the $(\rho,\sigma)$-weight
subcategory of the minimal categorification of the irreducible $\g^\up \oplus
\g^\down$-module $V(-\Lambda_1|\Lambda_{t^{-2}})$.
\item[(5)]
There is compatibility between the categorical $\g$-action
on $\mathcal C$ and the categorical $\g^\up\oplus\g^\down$-action on
the associated graded category in the sense that there are short
exact sequences as in (\ref{sesI})--(\ref{sesII}).
\end{itemize}
We must show that $\mathcal C := \Mod OS$ admits this structure.
It is locally Schurian and we have parametrized the irreducibles by
$\RPar$ above, so (1) holds. The main work still needed is to verify
(2) and (4); this is done in the next two paragraphs.
Then axiom (3) is Theorem~\ref{sstrat}, while (5) follows immediately
from Lemma~\ref{newz}.

To verify (2),
we use 
\cite[Theorem 4.27]{BD} to reduce to checking
the conditions
of \cite[Definition 4.25]{BD}. We need the following data:
\begin{itemize}
\item[(6)] A weight
decomposition of the category $\Mod OS$. 
\item[(7)] Biadjoint endofunctors
$E = \bigoplus_{i \in I}E_i$ and $F = \bigoplus_{i \in I} F_i$.
\item[(8)] 
Natural transformations 
$\mathord{
\begin{tikzpicture}[baseline = -2]
	\draw[->,thick,darkred] (0.08,-.15) to (0.08,.3);
      \node at (0.08,0.05) {$\color{darkred}\bullet$};
   \node at (0.08,-.25) {$\scriptstyle{i}$};
\end{tikzpicture}
}:E_i \rightarrow E_i$
and 
$\mathord{
\begin{tikzpicture}[baseline = -2]
	\draw[->,thick,darkred] (0.18,-.15) to (-0.18,.3);
	\draw[->,thick,darkred] (-0.18,-.15) to (0.18,.3);
   \node at (-0.18,-.25) {$\scriptstyle{i}$};
   \node at (0.18,-.25) {$\scriptstyle{j}$};
\end{tikzpicture}
}: E_i \circ E_j \rightarrow E_j \circ E_i$ for each $i,j \in I$
inducing an action of the quiver Hecke algebra $QH_r$
 associated to $\g$ 
on powers of $E$.
\end{itemize}
Then there are two additional axioms to check:
\begin{itemize}
\item[(9)] The endomorphisms $[E_i]$ and $[F_i]$ make
$\CC \otimes_\ZZ K_0(\pMod OS)$ into a well-defined $\g$-module with $\omega$-weight
space $\CC \otimes_\ZZ K_0(\pMod OS_\omega)$.
\item[(10)] For each $i \in I$ and each finitely generated $OS$-module
  $M$, the endomorphism $
\bigg(\mathord{
\begin{tikzpicture}[baseline = -2]
	\draw[->,thick,darkred] (0.08,-.15) to (0.08,.3);
      \node at (0.08,0.05) {$\color{darkred}\bullet$};
   \node at (0.08,-.25) {$\scriptstyle{i}$};
\end{tikzpicture}
}\bigg)_M:E_i M \rightarrow E_i M$ is nilpotent.
\end{itemize}
The weight decomposition (6) comes from (\ref{blocka}).
We have already constructed the functors needed
for (7) in Lemma~\ref{phew}.
For (8), we instead construct natural transformations $\mathord{
\begin{tikzpicture}[baseline = -1]
	\draw[->,thick,darkblue] (0.08,-.15) to (0.08,.3);
      \node at (0.08,0.06) {$\color{darkblue}\bullet$};
\end{tikzpicture}
}:E \rightarrow
E$ and $\mathord{
\begin{tikzpicture}[baseline = -1]
	\draw[->,thick,darkblue] (0.18,-.15) to (-0.18,.3);
	\draw[-,line width=4pt, white] (-0.18,-.15) to (0.18,.3);
	\draw[->,thick,darkblue] (-0.18,-.15) to (0.18,.3);
\end{tikzpicture}
}:E^2 \rightarrow E^2$ inducing an action of the affine Hecke algebra $AH_r$ on powers of $E$.
This is good enough due to the existence of an
isomorphism\footnote{There are various versions of this isomorphism in the
  literature. We will not make a specific choice here since any one of them suffices for our purposes.}
$\widehat{AH}_r \cong \widehat{QH}_r$ between completions constructed
in \cite{BK,Rou,W}.
Recalling the definition (\ref{edef}), we define $\mathord{
\begin{tikzpicture}[baseline = -1]
	\draw[->,thick,darkblue] (0.08,-.15) to (0.08,.3);
      \node at (0.08,0.06) {$\color{darkblue}\bullet$};
\end{tikzpicture}
}$ 
by setting $\Big(\mathord{
\begin{tikzpicture}[baseline = -1]
	\draw[->,thick,darkblue] (0.08,-.15) to (0.08,.3);
      \node at (0.08,0.06) {$\color{darkblue}\bullet$};
\end{tikzpicture}
}\Big)_M := \operatorname{id} \otimes {_{\up}}X:M \otimes_{OS}
{_\up}OS \rightarrow M \otimes_{OS} {_\up} OS$.
To define $\mathord{
\begin{tikzpicture}[baseline = -1]
	\draw[->,thick,darkblue] (0.18,-.15) to (-0.18,.3);
	\draw[-,line width=4pt, white] (-0.18,-.15) to (0.18,.3);
	\draw[->,thick,darkblue] (-0.18,-.15) to (0.18,.3);
\end{tikzpicture}
}$, we may identify ${_\up} OS \otimes_{OS} {_\up} OS$ with
${_{\up\up}}OS$ in the natural notation, then let $\Big(\mathord{
\begin{tikzpicture}[baseline = -1]
	\draw[->,thick,darkblue] (0.18,-.15) to (-0.18,.3);
	\draw[-,line width=4pt, white] (-0.18,-.15) to (0.18,.3);
	\draw[->,thick,darkblue] (-0.18,-.15) to (0.18,.3);
\end{tikzpicture}
}\Big)_M:M \otimes_{OS}
{_{\up\up}}OS
\rightarrow M \otimes_{OS} {_{\up\up}} OS$ be defined on $M 1_\a
\otimes {_{\up\up}} OS$ 
by left
multiplication by $\operatorname{id} \otimes \begin{tikzpicture}[baseline = -.5mm]
	\draw[->,thick,darkblue] (0.2,-.2) to (-0.2,.3);
	\draw[line width=4pt,white,-] (-0.2,-.2) to (0.2,.3);
	\draw[thick,darkblue,->] (-0.2,-.2) to (0.2,.3);
\draw[-,thick,double,darkblue] (.47,-.2) to (.47,.3);
\node at (.47,-.35) {$\a$};
\end{tikzpicture}$.
Axiom (9) follows from Lemma~\ref{int}.
For (10), note
that 
$$
\bigg(\mathord{
\begin{tikzpicture}[baseline = -2]
	\draw[->,thick,darkred] (0.08,-.15) to (0.08,.3);
      \node at (0.08,0.05) {$\color{darkred}\bullet$};
   \node at (0.08,-.25) {$\scriptstyle{i}$};
\end{tikzpicture}
}\bigg)_M = \left(\Big(\mathord{
\begin{tikzpicture}[baseline = -1]
	\draw[->,thick,darkblue] (0.08,-.15) to (0.08,.3);
      \node at (0.08,0.06) {$\color{darkblue}\bullet$};
\end{tikzpicture}
}\Big)_M - i \id\right)\bigg|_{E_i M}$$ 
according to the isomorphism
$\widehat{QH}_r \cong \widehat{AH}_r$. 
It therefore
suffices to show that there is a bound on the Jordan block sizes
of $\operatorname{id} \otimes {_{\up}}X:M \otimes_{OS} {_\up} OS
\rightarrow M \otimes_{OS} {_\up} OS$ for any finitely-generated
$OS$-module $M$.
This follows by the local finite-dimensionality
discussed in the proof of Lemma~\ref{phew}.
 
Finally, we need to verify (4). The categorical action of $\g^\up
\oplus \g^\down$ on $\Mod OS^\circ$ is constructed in a similar way to the
previous paragraph. The required endofuntors come from
(\ref{fan1})--(\ref{fan4}), 
the block decomposition is (\ref{blockb}), and we get ``control by $K_0$''
from Lemma~\ref{bunnarong}.
 In fact, due to Lemmas~\ref{cartanlem} and \ref{fan}, this 
is just a reformulation of 
the familiar categorical action on modules over Hecke algebras
constructed originally in \cite[$\S$7.2]{CR}.
It is a minimal categorification since $OS^\circ_{0,0} = \k$.
\end{proof}

\begin{proof}[Proof of Theorem~\ref{evalthm}]
Theorem~\ref{tpctheorem} implies that $\pMod OS$ is a 2-representation
of $\mathfrak{U}(\g)$.
Thus, 
letting $\mathfrak{C}at_\k$ be the 2-category of $\k$-linear categories,
there is a strict $\k$-linear 2-functor
$\mathfrak{U}(\g) \rightarrow \mathfrak{C}at_\k$
sending $\La \in P$ (i.e., an object of $\mathfrak{U}(\g)$)
to the block
$\pMod OS_\La$,
a 1-morphism $\underline{X}:\La \rightarrow \omega$
to a functor $X:\pMod OS_\La \rightarrow \pMod OS_{\omega}$,
and a 2-morphism $\underline{\eta}:\underline{X} \rightarrow \underline{Y}$
to a natural transformation $\eta:X \rightarrow Y$.
Noting that
$\wt(\NOTHING) = \La_{t^{-2}} - \La_{1}$, the universal property of
$\mathcal R(\La_{t^{-2}}-\La_{1})$
produces a strongly equivariant functor
$$
\Theta:\mathcal R(\La_{t^{-2}}-\La_1)
\rightarrow \pMod OS
$$ 
sending an object $\underline{X} \in \mathcal R(\La_{t^{-2}}-\La_1)$ (i.e., a
$1$-morphism
$\underline{X}:\La_{t^{-2}}-\La_1 \rightarrow \omega$ in
$\mathfrak{U}(\g)$
for some $\omega \in P$)
to the projective $OS$-module 
$X \Delta(\NOTHING)$,
and a morphism $\underline{\eta}:\underline{X} \rightarrow \underline{Y}$ (i.e., a $2$-morphism in
$\mathfrak{U}(\g)$)
to the $OS$-module homomorphism $\eta_{\Delta(\NOTHING)}:X
\Delta(\NOTHING)
\rightarrow  Y \Delta(\NOTHING)$.
Because the unit object of $\OS(z,t)$ corresponds to the
projective module $\Delta(\NOTHING)$ in $\pMod
OS$, this is the essentially the same as the functor appearing in the
theorem we are trying to prove.

In this paragraph, we check that $\Theta$ sends the 2-morphisms (\ref{ideal}) to zero.
For the first one, Lemmas~\ref{newfirst} and \ref{phew} imply that
$E_i \Delta(\NOTHING)$ is zero (so we get done trivially) unless $i =
1$, and also
$E_1 \Delta(\NOTHING) \cong \Delta(((1),\varnothing))$.
The relation follows in the 
non-trivial case $i=1$ because
$\Big(\mathord{
\begin{tikzpicture}[baseline = -1mm]
	\draw[->,darkred,thick] (0.08,-.2) to (0.08,.3);
      \node at (0.08,0) {$\color{darkred}\bullet$};
   \node at (0.1,-.32) {$\scriptstyle{1}$};
\end{tikzpicture}
}\Big)_{\Delta(\NOTHING)}$ is a nilpotent element of $\End_{OS}(\Delta(((1),\varnothing))) \cong \k$.
The second relation follows similarly.
For the final relation, we may assume that $t=\pm 1$, and need to show
that 
$\Big(\mathord{
\begin{tikzpicture}[baseline = 1mm]
  \draw[<-,thick,darkred] (0,0.4) to[out=180,in=90] (-.2,0.2);
  \draw[-,thick,darkred] (0.2,0.2) to[out=90,in=0] (0,.4);
 \draw[-,thick,darkred] (-.2,0.2) to[out=-90,in=180] (0,0);
  \draw[-,thick,darkred] (0,0) to[out=0,in=-90] (0.2,0.2);
 \node at (0,-.12) {$\scriptstyle{1}$};
\end{tikzpicture}
}\Big)\Big|_{\Delta(\NOTHING)}:\Delta(\NOTHING)\rightarrow \Delta(\NOTHING)$ is zero.
This endomorphism
is the
composition of two morphisms 

\vspace{-2mm}

$$
\Delta(\NOTHING) \stackrel{f}{\rightarrow}
E_{1} F_{1} \Delta(\NOTHING)\stackrel{g}{\rightarrow}
\Delta(\NOTHING)
$$
(the cup and the cap).
By Lemmas~\ref{newfirst} and \ref{phew}, the projective module
$E_{1} F_{1} \Delta(\NOTHING)$ has a two step $\Delta$-flag
with
$\Delta(\NOTHING)$ at the bottom and $\Delta(((1),(1)))$ at the top.
By Example~\ref{swex} with $n=0$, we know that
$[\bar\Delta(\NOTHING):\L(((1),(1)))] \neq 0$, so deduce
by BGG reciprocity that
$E_1 F_1 \Delta(\NOTHING) = \P(((1),(1)))$, i.e., it is indecomposable.
So the first morphism $f$ must be a scalar multiple of 
an inclusion of $\Delta(\NOTHING)$
into $E_{1} F_{1} \Delta(\NOTHING)$, and the second morphism must contain $\Delta(\NOTHING)$ in its
kernel.
Hence, $g \circ f = 0$ as required.

It follows that the functor $\Theta$ factors through the quotient to
induce a $\k$-linear functor
$$
\bar\Theta:\dot\VV(-\Lambda_1|\Lambda_{t^{-2}}) \rightarrow
\pMod OS.
$$
To show that this is an equivalence, we will show in the next two
paragraphs that
$\bar\Theta$ induces an isomorphism
\begin{equation}\label{magic}
\bar\Theta:
\Hom_{\dot\VV(-\Lambda_1|\Lambda_{t^{-2}})}(\underline{X},\underline{Y})
\stackrel{\sim}{\rightarrow}
\Hom_{OS}(X \Delta(\NOTHING), Y \Delta(\NOTHING))
\end{equation}
for any $\omega \in P$ and $\underline{X},
\underline{Y}:\La_{t^{-2}}-\La_1 \rightarrow \omega$ 
obtained as compositions\footnote{The infinite sums when $e=0$ make sense as $\underline{E}_i
1_\La$
and $\underline{F}_i 1_\La$ are
zero for all but finitely many $i \in I$.}
of the generating morphisms 
$\underline{E} = \bigoplus_{i \in I}
\underline{E}_i$ and $\underline{F} = \bigoplus_{i \in I}
\underline{F}_i$ in $\dot\VV(-\La_1|\La_{t^{-2}})$.
Let us see how the theorem follows from this.
Recall that $V(-\Lambda_1|\Lambda_{t^{-2}})$ is generated as a $\g$-module
by the vector $v_\NOTHING$. So, using Lemma~\ref{int} plus the natural
positivity of the actions of $[E]$ and $[F]$ on the basis coming from
indecomposable projectives, any $P$ in $\pMod OS$
isomorphic to a
summand of $X \Delta(\NOTHING)$ for some composition $X$ of $E$'s
and $F$'s. Let $e \in \End_{OS}(X \Delta(\NOTHING))$ be the
projection onto this summand.
The inverse image of $e$ under (\ref{magic}) gives
an idempotent in $\End_{\dot{\mathcal
    L}(-\Lambda_1|\Lambda_{t^{-2}})}(\underline{X})$. This defines an object of
$\dot\VV(-\Lambda_1|\Lambda_{t^{-2}})$ whose image under
$\bar\Theta$ is isomorphic to $P$. This shows that $\bar\Theta$ is
dense. It is full and faithful by (\ref{magic}).

So now we must prove (\ref{magic}).
Suppose that $x$ (resp. $x'$) letters of $\underline{X}$
and $y$ (resp. $y'$) letters of $\underline{Y}$ are equal to
$\underline{F}$ (resp. $\underline{E}$).
We may assume further that $r := x'+y=x+y'$, since otherwise both sides of
(\ref{magic}) are zero by weight considerations.
We observe for each $\La \in P$ 
that there is an isomorphism $\rho:\underline{E}\, \underline{F} 1_\La
\cong \underline{F} \,\underline{E} 1_\La$ in 
$\dot\VV(-\La_1|\La_{t^{-2}})$.
To prove this, for all $i,j \in I$, 
the relations in $\mathfrak{U}(\g)$ give canonical isomorphisms
$\underline{E}_i \underline{F}_j 1_\La \oplus 1_\La^{\oplus m_{i,j}} \cong \underline{F}_j
\underline{E}_i 1_\La \oplus 1_\La^{\oplus n_{i,j}}$ for
$m_{i,j},n_{i,j} \in \NN$, one of which is zero. Summing these
isomorphisms over all $i,j \in I$
gives a canonical isomorphism $\underline{E} \, \underline{F} 1_\La \oplus
1_\La^{\oplus m} \cong \underline{F}\, \underline{E}\oplus 1_\La^{\oplus n}$
for some $m,n \in \NN$. In fact, by weight considerations,
we have that $m=n$.
Then we use Krull-Schmidt, which holds because 
$\dot\VV(-\La_1|\La_{t^{-2}})$ is a finite-dimensional
category thanks to \cite[Corollary 4.17]{BD},
to deduce that the existence of the desired (non-canonical) isomorphism
$\rho:\underline{E} \,\underline{F} 1_\La 
\stackrel{\sim}{\rightarrow} \underline{F}\,
\underline{E} 1_\La$.
Then, using these isomorphisms plus isomorphisms coming from the
adjunction 2-morphisms in $\mathfrak{U}(\g)$,
we can construct a vector space isomorphism
$$
\theta:
\Hom_{\dot\VV(-\Lambda_1|\Lambda_{t^{-2}})}(\underline{X},\underline{Y})
\stackrel{\sim}{\rightarrow}
\Hom_{\dot\VV(-\Lambda_1|\Lambda_{t^{-2}})}(\underline{E}^r,\underline{E}^r)
$$
in just the same way as was done in (\ref{exactly}).
Applying $\bar\Theta$, we get also an isomorphism $\phi$ making the
left hand square of the following diagram commute:
\begin{equation}\label{queen}
\begin{CD}
\Hom_{\dot\VV(-\Lambda_1|\Lambda_{t^{-2}})}(\underline{X},\underline{Y}) &@>\sim>\theta>&
\Hom_{\dot\VV(-\Lambda_1|\Lambda_{t^{-2}})}(\underline{E}^r,
\underline{E}^r)
&@<\j_r<<&QH_r\\
@V\bar\Theta VV&&@VV\bar\Theta V&&@VV\psi V\\
\Hom_{OS}(X \Delta(\NOTHING), X \Delta(\NOTHING)) &@>\sim >\phi>&
\Hom_{OS}(E^r \Delta(\NOTHING), E^r \Delta(\NOTHING))&@<\sim<\i_r <&H_r.
\end{CD}
\end{equation}
Using this square, we are reduced 
to showing that the middle vertical map is an isomorphism.

To complete the argument, we already have the isomorphism $\i_r$ in
this diagram; it comes from
(\ref{randwick}).
Let $\j_r$ be the canonical homomorphism coming from the
categorical action (item (8) in the proof of
Theorem~\ref{tpctheorem}), then
define $\psi$ so that the right hand
square commutes.
We claim that $\j_r$ is surjective. To see this, \cite[Proposition
3.11]{KL3} shows that
$\Hom_{\dot\VV(-\Lambda_1|\Lambda_{t^{-2}})}(\underline{E}^r,
\underline{E}^r)$
is generated as a right $\End_{\dot{\mathcal
    L}(-\Lambda_1|\Lambda_{t^{-2}})}(1_{\Lambda_{t^{-2}}-\Lambda_1})$-module
by the image of $\j_r$.
But
$\End_{\dot{\mathcal
    L}(-\Lambda_1|\Lambda_{t^{-2}})}(1_{\Lambda_{t^{-2}}-\Lambda_1})$
is just the field $\k$ since there are enough relations in (\ref{ideal}) to
see that any dotted bubble is a scalar.
Moreover, $\ker \j_r$ contains the ideal $J_r$ of $QH_r$ generated by
$\big\{x^{\delta_{i_1,1}}_1 1_{\bi}\:\big|\:\bi = (i_1,\dots,i_r)\in I^r\big\}$
by the first relation from (\ref{ideal}), so $\psi$ induces
$\bar\psi: QH_r / J_r \rightarrow H_r$.
Since $J_r$ is the cyclotomic ideal of $QH_r$ associated to the dominant
weight $\La_1$, we get that $\bar\psi$ is an isomorphism 
by the main result of \cite{BK}.
It follows that $\bar\Theta$ is an isomorphism too.
\end{proof}

\section{Modifications in the degenerate case}\label{sdeg}

Assume in this section that $\k$ is a field of characteristic $p \geq 0$.
As we have said already in the introduction, when $z=0$, the
category $\OS(z,t)$ should be replaced with the oriented Brauer
category $\OB(\delta)$ studied in \cite{BCNR}.

\begin{proof}[Proof of Theorem~\ref{websup}]
This follows by the same general argument as used to prove Theorem~\ref{webs} (also
Remark~\ref{websrem}).
Instead of the quantized Schur-Weyl duality used before, one uses classical 
Schur-Weyl duality in its ``characteristic free'' form
established in \cite[Theorems 4.1--4.2]{dCP}.
\end{proof}

Now we discuss the degenerate analog of the results in sections 5, 6 and 7.

For section 5, we work with the locally finite-dimensional locally unital algebra
$$
OB = \bigoplus_{\a,\b \in \words} 1_\a OB 1_\b
\qquad\text{where}\qquad
1_\a OB 1_\b = \Hom_{\OB(\delta)}(\b,\a).
$$
It has a triangular decomposition $$
OB \cong OB^+ \otimes_\K OB^0
\otimes_\K OB^-$$ 
like in Lemma~\ref{td}. This actually becomes 
easier since there is no longer any need to be careful about upward strands passing
underneath downward strands when defining $OB^0$.
The subsequent arguments in section 5 then
go through easily on replacing the Hecke algebra $H_r$ with the
group algebra $\k \Sym_r$ of the symmetric group
and $e$ with $p$.

The results in section 6 go through too, but this needs a little more
work since the definitions of the various Jucys-Murphy elements from
(\ref{arun}), (\ref{JMstupid}) and
(\ref{chile2})--(\ref{chile5}) need some modifications, and the
details in the proofs of Lemmas~\ref{fur} and \ref{newz} then
need to be rechecked carefully.
The affine Hecke algebra $AH_r$ becomes the {\em degenerate affine
  Hecke algebra} $dAH_r$ whose polynomial generators $x_1,\dots,x_r$
satisfy the relations
\begin{align}
x_i x_j &= x_j x_i,
&
s_i x_{i+1} &= x_i s_i + 1
\end{align}
in place of (\ref{AHR}).
The unique homomorphism $dAH_r \twoheadrightarrow \k \Sym_r$
sending $s_i \mapsto s_i$ and $x_1 \mapsto 0$ 
sends $x_r$ to the {\em Jucys-Murphy element}
\begin{equation}
\jm_r := \sum_{i=1}^{r-1} (i\:\:r) \in \k \Sym_r.
\end{equation}
These elements are the replacements for (\ref{JMstupid}).
Then the contents of nodes of an ordinary Young diagram (which should
always be interpreted as
elements of the field $\k$) are as in the
following example
$$
\diagram{$ 0$&$ 1$&$2$&$3$&$4$\cr $
  -1$&$ 0$&$1$\cr
  $ -2$ & $ -1$ \cr}.
$$
In place of (\ref{Ic}), we set
\begin{equation}\label{Id}
I_c := \{c + n\:|\:n \in \ZZ\}  \subseteq \k
\end{equation}
for $c \in \k$.
The appropriate analog of Lemma~\ref{classical} uses $I_{0} \subseteq \k$ in
place of $I_1 \subseteq \k^\times$. 
It is a classical result in the (modular) representation theory of the
symmetric group.

The Jucys-Murphy elements of $\OB(\delta)$
are the images of corresponding elements of the
{\em affine oriented Brauer category} $\AOB(\delta)$ introduced in
\cite{BCNR}. This strict $\k$-linear monoidal category is defined by
by adjoining an additional generating
morphism $\mathord{
\begin{tikzpicture}[baseline = -1mm]
      \node at (0,0) {$\color{darkblue}\scriptstyle\bullet$};
	\draw[<-,thick,darkblue] (0,0.25) to (0,-0.25);
\end{tikzpicture}
}$ to $\OB(\delta)$,
subject to the relation ($\text{dA}$) from Figure 1.
The analog of Lemma~\ref{notmon} is explained in \cite[Theorem
3.3]{BCNR}: there is a $\k$-linear functor
$\beta:\AOB(\delta)\rightarrow \OB(\delta)$ sending diagrams with no dots to
the same diagrams in $\OB(\delta)$, and sending $\mathord{
\begin{tikzpicture}[baseline = -1mm]
      \node at (0,0) {$\color{darkblue}\scriptstyle\bullet$};
	\draw[<-,thick,darkblue] (0,0.25) to (0,-0.25);
\end{tikzpicture}
}\mapsto 0$.
The following computes the image of
$\mathord{
\begin{tikzpicture}[baseline = -1mm]
      \node at (0,0.025) {$\color{darkblue}\scriptstyle\bullet$};
	\draw[->,thick,darkblue] (0,0.25) to (0,-0.25);
\end{tikzpicture}
}$ (which is defined so that (\ref{d5})--(\ref{d4}) hold):
$$
\mathord{
\begin{tikzpicture}[baseline = -0.5mm]
	\draw[->,thick,darkblue] (0,0.6) to (0,-0.6);
     \node at (0,0) {$\color{darkblue}\bullet$};
\end{tikzpicture}
}
=
\mathord{
\begin{tikzpicture}[baseline = -0.5mm]
	\draw[-,thick,darkblue] (0,0.6) to (0,0.3);
	\draw[-,thick,darkblue] (0,0.3) to [out=-90,in=180] (.3,-0.2);
	\draw[-,thick,darkblue] (0.3,-0.2) to [out=0,in=-90](.5,0);
	\draw[-,thick,darkblue] (0.5,0) to [out=90,in=0](.3,0.2);
	\draw[-,thick,darkblue] (0.3,.2) to [out=180,in=90](0,-0.3);
	\draw[->,thick,darkblue] (0,-0.3) to (0,-0.6);
      \node at (0,.4) {$\color{darkblue}\bullet$};
\end{tikzpicture}
}
=
\mathord{
\begin{tikzpicture}[baseline = -0.5mm]
	\draw[-,thick,darkblue] (0,0.6) to (0,0.3);
	\draw[-,thick,darkblue] (0,0.3) to [out=-90,in=180] (.3,-0.2);
	\draw[-,thick,darkblue] (0.3,-0.2) to [out=0,in=-90](.5,0);
	\draw[-,thick,darkblue] (0.5,0) to [out=90,in=0](.3,0.2);
	\draw[-,thick,darkblue] (0.3,.2) to [out=180,in=90](0,-0.3);
	\draw[->,thick,darkblue] (0,-0.3) to (0,-0.6);
      \node at (0.48,.02) {$\color{darkblue}\bullet$};
\end{tikzpicture}
}
-
\mathord{
\begin{tikzpicture}[baseline = -0.5mm]
	\draw[->,thick,darkblue] (0,0.6) to (0,-0.6);
\end{tikzpicture}
}
\:\:\:\mathord{
\begin{tikzpicture}[baseline = 1.25mm]
  \draw[-,thick,darkblue] (0.2,0.2) to[out=90,in=0] (0,.4);
  \draw[-,thick,darkblue] (0,0.4) to[out=180,in=90] (-.2,0.2);
\draw[-,thick,darkblue] (-.2,0.2) to[out=-90,in=180] (0,0);
  \draw[->,thick,darkblue] (0,0) to[out=0,in=-90] (0.2,0.2);
\end{tikzpicture}
}\:\mapsto\:
-\delta\:
\mathord{
\begin{tikzpicture}[baseline = -0.5mm]
	\draw[->,thick,darkblue] (0,0.6) to (0,-0.6);
\end{tikzpicture}
}\,.
$$
Then we define $x(\b) \in \Hom_{\OB(\delta)}(\b,\b)$ in the same way as (\ref{JM})
for any $\varnothing \neq \b \in \words$.
There is no longer such a nice diagrammatic interpretation of these
elements like (\ref{arun}), but there
is a recursive definition as in (\ref{JMa})--(\ref{JMb}): we have that
$x(\up) = 0, x(\down) = -\delta 1_\down$, and 
\begin{align}\label{JManew}
x(\up \up \b) &:= 
\mathord{
\begin{tikzpicture}[baseline=2.5mm]
\draw[->,thick,darkblue] (0,.47) to [out=90,in=-90] (-.42,.9);
\draw[->,thick,darkblue] (-.42,.27) to [out=90,in=-120] (-0.05,.9);
\draw[-,thick,darkblue] (-.42,.27) to [out=-90,in=120] (-0.05,-.3);
\draw[-,thick,darkblue] (0,.12) to [out=-90,in=90] (-.42,-.3);
\node at (0.16,.3) {$\scriptstyle x(\up \b)$};
\draw[-, thick,double,darkblue] (.25,.12) to (.25,-.3);
\draw[-, thick,double,darkblue] (.25,.48) to (.25,.9);
\draw[darkblue,thick,yshift=-5pt,xshift=-5pt] (-.08,.3) rectangle ++(24pt,10pt);
\node at (0.25,-.44) {$\b$};
\node at (0.25,1.07) {$\b$};
\end{tikzpicture}}
+
\mathord{
\begin{tikzpicture}[baseline=2.5mm]
\draw[->,thick,darkblue] (0,-.3) to (-.3,.9);
\draw[->,thick,darkblue] (-.3,-.3) to (0,.9);
\draw[-, thick,double,darkblue] (.25,.9) to (.25,-.3);
\node at (0.25,-.44) {$\b$};
\end{tikzpicture}}\,,
&
x(\up \down \b) &:= 
\mathord{
\begin{tikzpicture}[baseline=2.5mm]
\draw[->,thick,darkblue] (0,.47) to [out=90,in=-90] (-.42,.9);
\draw[-,thick,darkblue] (-.42,.27) to [out=90,in=-120] (-0.05,.9);
\draw[->,thick,darkblue] (-.42,.27) to [out=-90,in=120] (-0.05,-.3);
\draw[-,thick,darkblue] (0,.12) to [out=-90,in=90] (-.42,-.3);
\node at (0.16,.3) {$\scriptstyle x(\up \b)$};
\draw[-, thick,double,darkblue] (.25,.12) to (.25,-.3);
\draw[-, thick,double,darkblue] (.25,.48) to (.25,.9);
\draw[darkblue,thick,yshift=-5pt,xshift=-5pt] (-.08,.3) rectangle ++(24pt,10pt);
\node at (0.25,-.44) {$\b$};
\node at (0.25,1.07) {$\b$};
\end{tikzpicture}}-
\mathord{
\begin{tikzpicture}[baseline=2.5mm]
\draw[<-,thick,darkblue] (0,-.3) to [out=90,in=0] (-.15,0.1);
\draw[-,thick,darkblue] (-.15,0.1) to [out=180,in=90] (-.3,-.3);
\draw[<-,thick,darkblue] (-.3,.9) to [out=-90,in=-180] (-.15,.5);
\draw[-,thick,darkblue] (-.15,.5) to [out=0,in=-90] (0,.9);
\draw[-, thick,double,darkblue] (.25,.9) to (.25,-.3);
\node at (0.25,-.44) {$\b$};
\end{tikzpicture}}
\:,\\
x(\down \down \b) &:= 
\mathord{
\begin{tikzpicture}[baseline=2.5mm]
\draw[-,thick,darkblue] (-.42,.27) to [out=90,in=-120] (-0.05,.9);
\draw[-,thick,darkblue] (0,.47) to [out=90,in=-90] (-.42,.9);
\draw[darkblue,thick,yshift=-5pt,xshift=-5pt] (-.08,.3) rectangle ++(24pt,10pt);
\draw[->,thick,darkblue] (0,.12) to [out=-90,in=90] (-.42,-.3);
\draw[->,thick,darkblue] (-.42,.27) to [out=-90,in=120] (-0.05,-.3);
\node at (0.16,.3) {$\scriptstyle x(\down \b)$};
\draw[-, thick,double,darkblue] (.25,.12) to (.25,-.3);
\draw[-, thick,double,darkblue] (.25,.48) to (.25,.9);
\node at (0.25,-.44) {$\b$};
\node at (0.25,1.07) {$\b$};
\end{tikzpicture}}-
\mathord{
\begin{tikzpicture}[baseline=2.5mm]
\draw[<-,thick,darkblue] (0,-.3) to (-.3,.9);
\draw[<-,thick,darkblue] (-.3,-.3) to (0,.9);
\draw[-, thick,double,darkblue] (.25,.9) to (.25,-.3);
\node at (0.25,-.44) {$\b$};
\end{tikzpicture}}\,,
&
x(\down \up \b) &:= 
\mathord{
\begin{tikzpicture}[baseline=2.5mm]
\draw[->,thick,darkblue] (-.42,.27) to [out=90,in=-120] (-0.05,.9);
\draw[-,thick,darkblue] (0,.47) to [out=90,in=-90] (-.42,.9);
\draw[darkblue,thick,yshift=-5pt,xshift=-5pt] (-.08,.3) rectangle ++(24pt,10pt);
\draw[->,thick,darkblue] (0,.12) to [out=-90,in=90] (-.42,-.3);
\draw[-,thick,darkblue] (-.42,.27) to [out=-90,in=120] (-0.05,-.3);
\node at (0.16,.3) {$\scriptstyle x(\down \b)$};
\draw[-, thick,double,darkblue] (.25,.12) to (.25,-.3);
\draw[-, thick,double,darkblue] (.25,.48) to (.25,.9);
\node at (0.25,-.44) {$\b$};
\node at (0.25,1.07) {$\b$};
\end{tikzpicture}}+
\mathord{
\begin{tikzpicture}[baseline=2.5mm]
\draw[-,thick,darkblue] (0,-.3) to [out=90,in=0] (-.15,0.1);
\draw[->,thick,darkblue] (-.15,0.1) to [out=180,in=90] (-.3,-.3);
\draw[-,thick,darkblue] (-.3,.9) to [out=-90,in=-180] (-.15,.5);
\draw[->,thick,darkblue] (-.15,.5) to [out=0,in=-90] (0,.9);
\draw[-, thick,double,darkblue] (.25,.9) to (.25,-.3);
\node at (0.25,-.44) {$\b$};
\end{tikzpicture}}
\:,\label{JMbnew}
\end{align}
for any word $\b$.
Finally, the Jucys-Murphy elements $x^\circ(\b)$ of $\OB^\circ(\delta)$, i.e., the
subcategory consisting of all objects but only morphisms represented
by diagrams with no cups or caps, 
are defined from
$x^\circ(\up) := 0, x^\circ(\down) := -\delta 1_\down$, and
\begin{align}\label{JMcnew}
x^\circ(\up \up \b) &:= 
\mathord{
\begin{tikzpicture}[baseline=2.5mm]
\draw[->,thick,darkblue] (0,.47) to [out=90,in=-90] (-.42,.9);
\draw[->,thick,darkblue] (-.42,.27) to [out=90,in=-120] (-0.05,.9);
\draw[-,thick,darkblue] (-.42,.27) to [out=-90,in=120] (-0.05,-.3);
\draw[-,thick,darkblue] (0,.12) to [out=-90,in=90] (-.42,-.3);
\node at (0.17,.3) {$\scriptstyle x^\circ(\up \b)$};
\draw[-, thick,double,darkblue] (.25,.12) to (.25,-.3);
\draw[-, thick,double,darkblue] (.25,.48) to (.25,.9);
\draw[darkblue,thick,yshift=-5pt,xshift=-5pt] (-.1,.3) rectangle ++(26pt,10pt);
\node at (0.25,-.44) {$\b$};
\node at (0.25,1.07) {$\b$};
\end{tikzpicture}}
+
\mathord{
\begin{tikzpicture}[baseline=2.5mm]
\draw[->,thick,darkblue] (0,-.3) to (-.3,.9);
\draw[->,thick,darkblue] (-.3,-.3) to (0,.9);
\draw[-, thick,double,darkblue] (.25,.9) to (.25,-.3);
\node at (0.25,-.44) {$\b$};
\end{tikzpicture}}\,,
&
x^\circ(\up \down \b) &:= 
\mathord{
\begin{tikzpicture}[baseline=2.5mm]
\draw[->,thick,darkblue] (0,.47) to [out=90,in=-90] (-.42,.9);
\draw[-,thick,darkblue] (-.42,.27) to [out=90,in=-120] (-0.05,.9);
\draw[->,thick,darkblue] (-.42,.27) to [out=-90,in=120] (-0.05,-.3);
\draw[-,thick,darkblue] (0,.12) to [out=-90,in=90] (-.42,-.3);
\node at (0.17,.3) {$\scriptstyle x^\circ(\up \b)$};
\draw[-, thick,double,darkblue] (.25,.12) to (.25,-.3);
\draw[-, thick,double,darkblue] (.25,.48) to (.25,.9);
\draw[darkblue,thick,yshift=-5pt,xshift=-5pt] (-.1,.3) rectangle ++(26pt,10pt);
\node at (0.25,-.44) {$\b$};
\node at (0.25,1.07) {$\b$};
\end{tikzpicture}}\:,\\
x^\circ(\down \down \b) &:= 
\mathord{
\begin{tikzpicture}[baseline=2.5mm]
\draw[-,thick,darkblue] (-.42,.27) to [out=90,in=-120] (-0.05,.9);
\draw[-,thick,darkblue] (0,.47) to [out=90,in=-90] (-.42,.9);
\draw[darkblue,thick,yshift=-5pt,xshift=-5pt] (-.1,.3) rectangle ++(26pt,10pt);
\draw[->,thick,darkblue] (0,.12) to [out=-90,in=90] (-.42,-.3);
\draw[->,thick,darkblue] (-.42,.27) to [out=-90,in=120] (-0.05,-.3);
\node at (0.17,.3) {$\scriptstyle x^\circ(\down \b)$};
\draw[-, thick,double,darkblue] (.25,.12) to (.25,-.3);
\draw[-, thick,double,darkblue] (.25,.48) to (.25,.9);
\node at (0.25,-.44) {$\b$};
\node at (0.25,1.07) {$\b$};
\end{tikzpicture}}-
\mathord{
\begin{tikzpicture}[baseline=2.5mm]
\draw[<-,thick,darkblue] (0,-.3) to (-.3,.9);
\draw[<-,thick,darkblue] (-.3,-.3) to (0,.9);
\draw[-, thick,double,darkblue] (.25,.9) to (.25,-.3);
\node at (0.25,-.44) {$\b$};
\end{tikzpicture}}\,,
&
x^\circ(\down \up \b) &:= 
\mathord{
\begin{tikzpicture}[baseline=2.5mm]
\draw[->,thick,darkblue] (-.42,.27) to [out=90,in=-120] (-0.05,.9);
\draw[-,thick,darkblue] (0,.47) to [out=90,in=-90] (-.42,.9);
\draw[darkblue,thick,yshift=-5pt,xshift=-5pt] (-.1,.3) rectangle ++(26pt,10pt);
\draw[->,thick,darkblue] (0,.12) to [out=-90,in=90] (-.42,-.3);
\draw[-,thick,darkblue] (-.42,.27) to [out=-90,in=120] (-0.05,-.3);
\node at (0.17,.3) {$\scriptstyle x^\circ(\down \b)$};
\draw[-, thick,double,darkblue] (.25,.12) to (.25,-.3);
\draw[-, thick,double,darkblue] (.25,.48) to (.25,.9);
\node at (0.25,-.44) {$\b$};
\node at (0.25,1.07) {$\b$};
\end{tikzpicture}}\:.\label{JMdnew}
\end{align}
We leave it to the reader to verify with these new definitions
that Lemmas~\ref{fur} and \ref{newz} go through; see also \cite{R}. In
the statement of Lemma~\ref{fur}, one should replace 
$t^{-2}i^{-1}$ with $-i-\delta$,
$I_1$ with $I_0$, and $I_{t^{-2}}$ with $I_{-\delta}$.
Also
 the set $I$ from (\ref{I}) becomes
\begin{equation}
I := I_{0} \cup I_{-\delta} = \{n, -n - \delta\:|\:n \in \ZZ\} \subseteq \k.
\end{equation}
Adjusting the subsequent combinatorics in analogous ways, all of
the other results of section 6 follow as before.

Moving on to section 7, 
the Lie algebra $\g$ is the Kac-Moody algebra associated to the Cartan
matrix
$(c_{i,j})_{i,j \in I}$ defined by (\ref{cartan2}).
The module $V(-\Lambda_1|\Lambda_{t^{-2}})$ becomes
$V(-\Lambda_0|\Lambda_{-\delta})$, and
the degenerate analogs of (\ref{css1})--(\ref{css2}) are
\begin{align}\label{css1}
\wt^\up(\LA)
&:=-\La_0 + \sum_{A \in \la^\up}
\alpha_{\cont(A)},\\\label{css2}
\wt^\down(\LA)
&:=\La_{-\delta} - \sum_{A \in \la^\down}
\alpha_{-\!\cont(A)-\delta},.
\end{align}
There are no other significant discrepancies.

\begin{proof}[Proof of Theorem~\ref{aisha2}]
This is the same as the proof of
Theorem~\ref{evalthm} given in the previous section.
\end{proof}

\end{document}